\documentclass[11pt]{amsart}
\usepackage{amsfonts,latexsym,amsthm,amssymb,graphicx}
\usepackage[all]{xy}
\usepackage[usenames]{color} 
\usepackage{xypic,amscd,epsfig}
\usepackage{hyperref}
\usepackage{pdfsync}
\usepackage[capitalize,nameinlink]{cleveref}

\DeclareFontFamily{OT1}{rsfs}{}
\DeclareFontShape{OT1}{rsfs}{n}{it}{<-> rsfs10}{}
\DeclareMathAlphabet{\mathscr}{OT1}{rsfs}{n}{it}

\CompileMatrices

\newtheorem{theorem}{Theorem}[section]
\newtheorem{lemma}[theorem]{Lemma}
\newtheorem{corol}[theorem]{Corollary}
\newtheorem{prop}[theorem]{Proposition}

\newtheorem{conj}{Conjecture}

\newtheorem{definition}[theorem]{Definition}

{\theoremstyle{remark} \newtheorem{remark}[theorem]{Remark}
\newtheorem{example}[theorem]{Example}}

\newcommand{\nc}{\newcommand}
\nc{\rnc}{\renewcommand}
\nc{\bb}[1]{{\mathbb #1}}
\nc{\bbA}{\bb{A}}\nc{\bbB}{\bb{B}}\nc{\bbC}{\bb{C}}\nc{\bbD}{\bb{D}}
\nc{\bbE}{\bb{E}}\nc{\bbF}{\bb{F}}\nc{\bbG}{\bb{G}}\nc{\bbH}{\bb{H}}
\nc{\bbI}{\bb{I}}\nc{\bbJ}{\bb{J}}\nc{\bbK}{\bb{K}}\nc{\bbL}{\bb{L}}
\nc{\bbM}{\bb{M}}\nc{\bbN}{\bb{N}}\nc{\bbO}{\bb{O}}\nc{\bbP}{\bb{P}}
\nc{\bbQ}{\bb{Q}}\nc{\bbR}{\bb{R}}\nc{\bbS}{\bb{S}}\nc{\bbT}{\bb{T}}
\nc{\bbU}{\bb{U}}\nc{\bbV}{\bb{V}}\nc{\bbW}{\bb{W}}\nc{\bbX}{\bb{X}}
\nc{\bbY}{\bb{Y}}\nc{\bbZ}{\bb{Z}}
\nc{\mbf}[1]{{\mathbf #1}}
\nc{\bfA}{\mbf{A}}\nc{\bfB}{\mbf{B}}\nc{\bfC}{\mbf{C}}\nc{\bfD}{\mbf{D}}
\nc{\bfE}{\mbf{E}}\nc{\bfF}{\mbf{F}}\nc{\bfG}{\mbf{G}}\nc{\bfH}{\mbf{H}}
\nc{\bfI}{\mbf{I}}\nc{\bfJ}{\mbf{J}}\nc{\bfK}{\mbf{K}}\nc{\bfL}{\mbf{L}}
\nc{\bfM}{\mbf{M}}\nc{\bfN}{\mbf{N}}\nc{\bfO}{\mbf{O}}\nc{\bfP}{\mbf{P}}
\nc{\bfQ}{\mbf{Q}}\nc{\bfR}{\mbf{R}}\nc{\bfS}{\mbf{S}}\nc{\bfT}{\mbf{T}}
\nc{\bfU}{\mbf{U}}\nc{\bfV}{\mbf{V}}\nc{\bfW}{\mbf{W}}\nc{\bfX}{\mbf{X}}
\nc{\bfY}{\mbf{Y}}\nc{\bfZ}{\mbf{Z}}
\nc{\bfa}{\mbf{a}}\nc{\bfb}{\mbf{b}}\nc{\bfc}{\mbf{c}}\nc{\bfd}{\mbf{d}}
\nc{\bfe}{\mbf{e}}\nc{\bff}{\mbf{f}}\nc{\bfg}{\mbf{g}}\nc{\bfh}{\mbf{h}}
\nc{\bfi}{\mbf{i}}\nc{\bfj}{\mbf{j}}\nc{\bfk}{\mbf{k}}\nc{\bfl}{\mbf{l}}
\nc{\bfm}{\mbf{m}}\nc{\bfn}{\mbf{n}}\nc{\bfo}{\mbf{o}}\nc{\bfp}{\mbf{p}}
\nc{\bfq}{\mbf{q}}\nc{\bfr}{\mbf{r}}\nc{\bfs}{\mbf{s}}\nc{\bft}{\mbf{t}}
\nc{\bfu}{\mbf{u}}\nc{\bfv}{\mbf{v}}\nc{\bfw}{\mbf{w}}\nc{\bfx}{\mbf{x}}
\nc{\bfy}{\mbf{y}}\nc{\bfz}{\mbf{z}}
\nc{\mcal}[1]{{\mathcal #1}}
\nc{\calA}{\mcal{A}}\nc{\calB}{\mcal{B}}\nc{\calC}{\mcal{C}}\nc{\calD}{\mcal{D}}
\nc{\calE}{\mcal{E}} \nc{\calF}{\mcal{F}}\nc{\calG}{\mcal{G}}\nc{\calH}{\mcal{H}}
\nc{\calI}{\mcal{I}}\nc{\calJ}{\mcal{J}}\nc{\calK}{\mcal{K}}\nc{\calL}{\mcal{L}}
\nc{\calM}{\mcal{M}}\nc{\calN}{\mcal{N}}\nc{\calO}{\mcal{O}}\nc{\calP}{\mcal{P}}
\nc{\calQ}{\mcal{Q}}\nc{\calR}{\mcal{R}}\nc{\calS}{\mcal{S}}\nc{\calT}{\mcal{T}}
\nc{\calU}{\mcal{U}}\nc{\calV}{\mcal{V}}\nc{\calW}{\mcal{W}}\nc{\calX}{\mcal{X}}
\nc{\calY}{\mcal{Y}}\nc{\calZ}{\mcal{Z}}
\nc{\fA}{\frak{A}}\nc{\fB}{\frak{B}}\nc{\fC}{\frak{C}} \nc{\fD}{\frak{D}}
\nc{\fE}{\frak{E}}\nc{\fF}{\frak{F}}\nc{\fG}{\frak{G}}\nc{\fH}{\frak{H}}
\nc{\fI}{\frak{I}}\nc{\fJ}{\frak{J}}\nc{\fK}{\frak{K}}\nc{\fL}{\frak{L}}
\nc{\fM}{\frak{M}}\nc{\fN}{\frak{N}}\nc{\fO}{\frak{O}}\nc{\fP}{\frak{P}}
\nc{\fQ}{\frak{Q}}\nc{\fR}{\frak{R}}\nc{\fS}{\frak{S}}\nc{\fT}{\frak{T}}
\nc{\fU}{\frak{U}}\nc{\fV}{\frak{V}}\nc{\fW}{\frak{W}}\nc{\fX}{\frak{X}}
\nc{\fY}{\frak{Y}}\nc{\fZ}{\frak{Z}}
\nc{\fa}{\frak{a}}\nc{\fb}{\frak{b}}\nc{\fc}{\frak{c}} \nc{\fd}{\frak{d}}
\nc{\fe}{\frak{e}}\nc{\fFf}{\frak{f}}\nc{\fg}{\frak{g}}\nc{\fh}{\frak{h}}
\nc{\fri}{\frak{i}}\nc{\fj}{\frak{j}}\nc{\fk}{\frak{k}}\nc{\fl}{\frak{l}}
\nc{\fm}{\frak{m}}\nc{\fn}{\frak{n}}\nc{\fo}{\frak{o}}\nc{\fp}{\frak{p}}
\nc{\fq}{\frak{q}}\nc{\fr}{\frak{r}}\nc{\fs}{\frak{s}}\nc{\ft}{\frak{t}}
\nc{\fu}{\frak{u}}\nc{\fv}{\frak{v}}\nc{\fw}{\frak{w}}\nc{\fx}{\frak{x}}
\nc{\fy}{\frak{y}}\nc{\fz}{\frak{z}}

\newcommand{\C}{{\mathbb C}}

\newcommand{\Pbb}{{\mathbb{P}}}

\newcommand{\Z}{{\mathbb Z}}

\newcommand{\cF}{{\mathcal F}}

\newcommand{\cO}{{\mathcal O}}

\newcommand{\Fl}{\mathrm{Fl}}
\newcommand{\Gr}{\mathrm{Gr}}
\newcommand{\K}{\mathrm{K}}

\newcommand{\ssm}{{s_{\text{SM}}}}
\newcommand{\csm}{{c_{\text{SM}}}}

\newcommand{\csmh}{{c^\hbar_{\text{SM}}}}

\newcommand{\csmT}{{c^{{T}}_{\text{SM}}}}
\newcommand{\csmTv}{{c^{{T},\vee}_{\text{SM}}}}

\newcommand{\csmTh}{{c^{{T},\hbar}_{\text{SM}}}}

\newcommand{\one}{1\hskip-3.5pt1}

\newcommand{\parcoh}{\partial^\mathrm{coh}}
\newcommand{\operL}{\mathfrak{L}}

\DeclareMathOperator{\id}{id}

\DeclareMathOperator{\Frac}{Frac}

\DeclareMathOperator{\End}{End}

\DeclareMathOperator{\ch}{ch}
\begin{document}
\title[MC, Hirzebruch, and CSM classes of Schubert cells]{From motivic Chern classes of Schubert cells to their Hirzebruch and CSM classes}

\date{December 22, 2022}

\author{Paolo Aluffi}
\address{
Mathematics Department, 
Florida State University,
Tallahassee FL 32306
}
\email{aluffi@math.fsu.edu}

\author{Leonardo C.~Mihalcea}
\address{
Department of Mathematics, 
Virginia Tech University, 
Blacksburg, VA 24061
}
\email{lmihalce@vt.edu}

\author{J\"org Sch\"urmann}
\address{Mathematisches Institut, Universit\"at M\"unster, Germany}
\email{jschuerm@uni-muenster.de}

\author{Changjian Su}
\address{
Yau Mathematical Sciences Center,
Tsinghua University, 
100084, Beijing, China}
\email{changjiansu@mail.tsinghua.edu.cn}


\begin{abstract} The equivariant motivic Chern class of a Schubert cell in a `complete' 
flag manifold $X=G/B$ 
is an element in the equivariant K theory ring of $X$ to which one adjoins a formal parameter $y$.
In this paper we prove several `folklore results' about the motivic Chern classes, 
including finding specializations at $y=-1$ and $y=0$; the coefficient of the top power of $y$;
how to obtain Chern-Schwartz-MacPherson (CSM) classes as leading terms of motivic classes; divisibility 
properties of the Schubert expansion of motivic Chern classes. We collect several conjectures
about the positivity, unimodality, and log concavity of CSM and motivic Chern classes of Schubert cells, including a conjectural positivity of structure constants
of the multiplication of Poincar{\'e} duals of CSM classes.
In addition, we prove a `star duality'
for the motivic Chern classes. We utilize the motivic Chern transformation 
to define two equivariant variants of the Hirzebruch
transformation, which appear naturally in the Grothendieck-Hirzebruch-Riemann-Roch formalism.
We utilize the Demazure-Lusztig recursions from the motivic Chern class theory to find similar recursions giving the 
Hirzebruch classes of Schubert cells, their Poincar{\'e} duals, and their Segre
versions. We explain the functoriality properties needed to extend the results to `partial' flag manifolds $G/P$. 
\end{abstract}

\maketitle

\tableofcontents

\section{Introduction} Let $X$ be a quasi-projective complex algebraic variety and denote by $\K_0(var/X)$
the Grothendieck motivic group consisting of equivalence classes of morphisms $[f:Z \to X]$
modulo the usual additivity relations. Denote also by $\K(X)$ the Grothendieck ring of vector bundles
on $X$. The {\em motivic Chern transformation\/} defined by Brasselet, Sch{\"u}rmann and Yokura 
\cite{brasselet.schurmann.yokura:hirzebruch}
is the assignment for every such $X$ of a group homomorphism
\[ 
MC_y:  \K_0(var/X) \to \K(X)[y] \/, 
\]
uniquely determined by the fact that it commutes with proper push-forwards and that it satisfies 
the normalization condition 
\[ MC_y[\id_X: X \to X] = \lambda_y(T^*_X) = \sum y^i [\wedge^i T^*X]_T \in \K_T(X)[y] \]
if $X$ is nonsingular.
If $Z \hookrightarrow X$ is a locally closed subset, the {\em motivic Chern class} of $Z$ (regarded in $X$) 
is defined by 
\[ MC_y(Z):= MC_y[Z \hookrightarrow X] \in \K(X)[y] \/;\]
here $y$ is a formal indeterminate.
If $X$ admits a torus action, there is an equivariant version 
$MC_y:  \K_0^T(var/X) \to \K_T(X)[y]$
defined in \cite{feher2018motivic,AMSS:motivic}; we will work in this context.

In this paper we study the classes $MC_y(X(w)^\circ)$,
the (torus equivariant) 
motivic Chern classes of Schubert cells $X(w)^\circ$ in the flag manifolds $G/B$,
for $G$ a complex, semisimple, Lie group, and $B \subseteq G$ a Borel subgroup. By functoriality,
these determine the motivic Chern classes in the `partial' flag manifolds $G/P$,
with $P \supset B$ a standard parabolic subgroup.

The motivic classes $MC_y(X(w)^\circ)$ are closely related to the study of 
representation theory 
of the Hecke algebra of $G$, and through this connection they play a 
prominent r{\^o}le in several related topics:
($\K$-theoretic) stable envelopes and integrable systems \cite{RTV:Kstable,AMSS:motivic,feher2018motivic}, 
Whittaker functions from $p$-adic representation theory \cite{mihalcea2019whittaker},
characteristic classes of singular varieties \cite{feher2018characteristic}. In interesting situations they 
recover point counting 
over finite fields \cite{mihalcea2019whittaker} (see also \S \ref{ssec:pointc} below),
and are closely related to the study of the 
intersection (co)homology and the Riemann-Hilbert correspondence for arbitrary complex 
projective manifolds \cite{schurmann2009characteristic,AMSS:MCcot}. In Schubert Calculus, the motivic Chern classes,
and their cohomological counterparts, the {\em Chern-Schwartz-MacPherson (CSM)} classes,
provide deformations of the usual Schubert classes, which, provably or conjecturally,
satisfy remarkable positivity, unimodality, and log-concavity properties; see~\S\ref{ss:pos} below.

Among the main goals of this paper is to gather in a single place several `folklore results'
about properties of motivic classes $MC_y(X(w)^\circ)$. 
These include results about the specializations at $y=-1$ and $y=0$ of the parameter $y$; the coefficient
of $y^{\dim X(w)}$ in~$MC_y(X(w)^\circ)$; how to recover the CSM classes as the initial terms of the motivic 
Chern classes; divisibility properties of Schubert expansions.
As mentioned above, we also record several conjectural properties of  
CSM and motivic classes. In addition, we prove a new 
`star duality' for motivic Chern classes, generalizing to the motivic situation a known relation between
ideal sheaves and duals of structure sheaves of Schubert varieties; cf.~\S \ref{sec:stardual}. 

Our main new contribution is to utilize the properties of motivic Chern classes to
study the {\em Hirzebruch transformation} $Td_{y,*}^T$, 
and the {\em Hirzebruch classes} $Td_{y,*}^T(X(w)^\circ)$ of Schubert cells. 
Similar to the motivic Chern transformation, the Hirzebruch transformation 
${Td}^{T}_{y,*}: \K_0^{T}(var/X) \to \widehat{H}_*^{T}(X; \mathbb{Q}[y]) $
is a functorial transformation defined uniquely by a normalization property,
with values in a completed (Borel-Moore or Chow) homology group; see \S \ref{s:eqH}.
In the non-equivariant context this transformation was defined in \cite{brasselet.schurmann.yokura:hirzebruch}, 
and it arises naturally in 
the context of the Grothendieck-Hirzebruch-Riemann-Roch (GHRR) formalism. The `unnormalized' variant
of this transformation was studied by Weber \cite{weber1,weber2}. 

In this paper we extend the definition of the Hirzebruch transformation to the equivariant context, for arbitrary
quasi-projective complex algebraic varieties $X$ with a torus action. 
As hinted above, there are two variants of the Hirzebruch transformation.
The `unnormalized' variant is defined as the composition
\[ \widetilde{Td}^{T}_{y,*}:= td^{T}_*\circ MC_y: \K_0^{T}(var/X) \to \widehat{H}_*^{T}(X; \mathbb{Q}[y]) \]
of the (equivariant) Todd transformation $td_*^T$ constructed by Edidin and Graham \cite{edidin2000} 
with the motivic Chern transformation. The `normalized' version is the
composition
\[ {Td}^{T}_{y,*}:=\psi_*^{1+y} \circ  \widetilde{Td}^{T}_{y,*}:  \K_0^{T}(var/X) \to 
\widehat{H}_*^{T}(X; \mathbb{Q}[y]) \subseteq \widehat{H}_*^{T}(X; \mathbb{Q}[y,(1+y)^{-1}]) \/.\]
of a certain Adams operator with the unnormalized transformation. The Adams operator acts by 
multiplying by powers of $1+y$ (see \S \ref{s:eqH}). A technical subtlety is that {\em a priori} 
${Td}^{T}_{y,*}$ requires coefficients in $\mathbb{Q}[y,(1+y)^{-1}]$, but one proves that 
$\mathbb{Q}[y]$ suffices. In fact, an important property of 
${Td}^{T}_{y,*}$ is that the specialization at $y=-1$ is well defined. This specialization 
recovers the (equivariant) MacPherson's transformation 
\cite{macpherson:chern,ohmoto:eqcsm} $c_*:\mathcal{F}^T(X) \to H_*^T(X)$ 
from the group of (equivariant) constructible functions to homology; see \Cref{cor:special-Ty}.

Once the general theory is developed, we turn to the situation where $X$ is a flag manifold. Utilizing the Demazure-Lusztig (DL) 
operators which determine recursively the motivic Chern classes $MC_y(X(w)^\circ)$, we obtain operators calculating (recursively) the
Hirzebruch classes $\widetilde{Td}^{T}_{y,*}(X(w)^\circ)$ and ${Td}^{T}_{y,*}(X(w)^\circ)$, their Poincar{\'e} duals, and the 
Segre-Hirzebruch 
classes. Perhaps not surprisingly, the theory we find is essentially equivalent to that of motivic Chern classes. For instance, the Hirzebruch 
operators are images of the DL operators in $\K$-theory via a Todd transformation. In particular, the 
Hirzebruch operators satisfy the same relations as the DL operators in $\K$ theory, implying that 
they give an action of the Hecke algebra on the equivariant (co)homology of $G/B$. 

For the convenience of the reader we will now illustrate some of these results; we refer 
to \S \ref{s:eqH} for all the statements. Let $P_i$ be the minimal parabolic group associated to the $i$-th simple root, 
and denote by $p_i:G/B \to G/P_i$ the natural projection. The BGG operator $\parcoh_i$ is defined as $(p_i)^* (p_i)_*$. 
Define the unnormalized and the normalized variants of the Hirzebruch operators 
\[ \widetilde{\mathcal{T}}_i^H,{\mathcal{T}}_i^{H} :\widehat{H}^*_{T}(X,\mathbb{Q})[y]  \to 
\widehat{H}^*_{T}(X,\mathbb{Q})[y]\]
by 
\[  \widetilde{{\mathcal{T}}}_i^H := 
\widetilde{{Td}}^{T}_{y}(T_{p_i})\parcoh_i  - \id \/; \quad {\mathcal{T}}_i^H := 
{Td}^{T}_{y}(T_{p_i})\parcoh_i  - \id \/.\]
Here $\widetilde{{Td}}^{T}_{y}(T_{p_i})$ and ${Td}^{T}_{y}(T_{p_i})$ are the 
unnormalized, respectively the normalized, Hirzebruch class associated to the 
relative tangent bundle $T_{p_i}$. If $y=0$, each of these is equal to the Todd class of $T_{p_i}$, but if $y=-1$,
\[ \widetilde{{Td}}^{T}_{y}(T_{p_i})_{y=-1} = c_1^T(T_{p_i}) \textrm{ and }  {{Td}}^{T}_{y}(T_{p_i})_{y=-1} = 1+c_1^T(T_{p_i}) \/. \]
\begin{theorem} (cf.~ \Cref{prop:HDL-rec}) Let $W$ be the Weyl group of $G$, $w \in W$, and $s_i$ a simple reflection
such that $ws_i >w$ in the Bruhat ordering. Consider the Schubert 
cell $X(w)^\circ \subseteq G/B$, of dimension $\ell(w)$. Then the Hirzebruch classes are determined by the 
following recursions: 
\[  \widetilde{{\mathcal{T}}}_i^H(\widetilde{Td}^{T}_{y,*}(X(w)^\circ)) = \widetilde{Td}^{T}_{y,*}(X(ws_i)^\circ) \textrm{ and } \/
{{\mathcal{T}}}_i^H({Td}^{T}_{y,*}(X(w)^\circ)) = {Td}^{T}_{y,*}(X(ws_i)^\circ) \/.\]
\end{theorem}
The specialization at $y=0$ of either of the two variants coincides with the Todd transformation
of the class of the ideal sheaves of the boundary of Schubert varieties:
\[ \widetilde{Td}^{T}_{y,*}(X(w)^\circ)_{y=0} = {Td}^{T}_{y,*}(X(w)^\circ)_{y=0}= td_*^T(\mathcal{I}_w^T) \/. \]
The specialization at $y=0$ of the Poincar{\'e} dual classes (see \Cref{thm:Hdual}) gives the Chern
character $\ch_T(\cO^{w,T})$ of the structure sheaves of the (opposite) Schubert varieties.
In particular, in \S \ref{s:specHclasses} we obtain recursions calculating each of the classes (see also the paragraph around
the equation \eqref{equ:BGGChern}), and we observe the orthogonality with respect 
to the Poincar{\'e} pairing: 
\[ \langle td_*^T(\mathcal{I}_u^T), \ch_T(\cO^{v,T}) \rangle = \delta_{u,v} \/. \]
(This also follows directly from the GHRR theorem and the duality $\langle \mathcal{I}_u^T, \cO^{v,T} \rangle = \delta_{u,v}$ in
 $\K$-theory.) The specialization at $y=-1$ recovers
 the CSM and the Segre-MacPherson classes of Schubert cells, and the recursions from \cite{aluffi.mihalcea:eqcsm,AMSS:shadows} 
 in terms of degenerate Hecke algebra operators.

 {\em Acknowledgments.} We would like to thank Mahir Can and J{\"o}rg Feldvoss
 for organizing the AMS Special Session on `Combinatorial and Geometric Representation Theory',
 and for their interest in this work.  
This material is partly based upon work supported by the National Science Foundation 
under Grant No.~DMS-1929284 while LM was in residence at the Institute for Computational and 
Experimental Research in Mathematics in Providence, RI, during the 
Combinatorial Algebraic Geometry program. LM was also supported in part by the NSF grant 
DMS-2152294 and a Simons Collaboration Grant. PA is supported in part by Simons 
Collaboration Grant~\#625561. J. Sch\"urmann
is funded by the Deutsche Forschungsgemeinschaft (DFG, German Research Foundation) Project-ID 427320536 -- SFB~1442, as well as under Germany's Excellence Strategy EXC 2044 390685587, Mathematics M\"unster: Dynamics -- Geometry -- Structure.

Throughout this project we used the Maple package
\texttt{Equivariant Schubert Calculator}, written by Anders Buch.\begin{footnote}
{The program is available at \texttt{https://sites.math.rutgers.edu/$\sim$asbuch/equivcalc/}}\end{footnote}

\section{Preliminaries}
\subsection{Equivariant (co)homology}\label{ss:ECSM}
Let $X$ be a quasi-projective complex algebraic variety. In this paper we will deal with
the Borel-Moore homology group $H_*(X)$ of $X$ and the cohomology ring $H^*(X)$,
with rational coefficients. As an alternative, one could use Chow (co)homology; there is
a homology degree-doubling cycle map from Chow to Borel-Moore, and our constructions 
are compatible with this map. This map is an isomorphism in some important
situations, such as the complex flag manifolds studied later in
this note. We refer to \cite[\S 19.1]{fulton:IT} and \cite[\S 2.6]{ginzburg:methods} 
for more details about Borel-Moore homology
and its relation to the Chow group.  In case we speak of (co)dimension
we always assume that our spaces are pure dimensional; in addition,
by (co)dimension
we will mean the {\em complex} (co)dimension. Any
subvariety $Y\subseteq X$ of (complex) dimension $k$ has a fundamental
class $[Y] \in H_{2k}(X)$.  Whenever $X$ is smooth, we can and will
identify the Borel-Moore homology and cohomology via Poincar{\'e}
duality.

Let $T$ be a torus and let $X$ be a variety with a $T$-action. Then
the equivariant cohomology~$H^*_T(X)$ is the ordinary cohomology of
the Borel mixing space $X_T:= (ET \times X)/T$, where $ET$ is the
universal $T$-bundle and $T$ acts by $t \cdot (e,x) = 
(e t^{-1}, t x)$. The ring $H^*_T(X)$ is an algebra over $H^*_T(pt)$, the polynomial ring 
$Sym_{\mathbb{Q}} \mathfrak{X}(T) \simeq \mathbb{Q}[t_1,\ldots , t_s]$ 
in the character group $\mathfrak{X}(T)$ (written additively)
and with $t_i \in H^2_T(pt)$; see e.g.,~\cite[\S 11.3.5]{kumar:KMbook}.
One may also define $T$-equivariant Borel-Moore homology and Chow
groups, related by an equivariant cycle map; see e.g.,
\cite{edidin.graham:eqchow}. Every $k$-dimensional subvariety
$Y\subseteq X$ that is stable under the $T$ action determines
an equivariant fundamental class $[Y]_T$ in
$H_{2k}^T(X)$. 

As in the non-equivariant case, whenever
  $X$ is smooth, we will identify $H_*^T(X)$ and~$H^*_T(X)$.  In
  particular, when $X=pt$, the identification sends $a \in H^*_T(pt)$
  to $a \cap [pt]_T$. If $X$ is smooth and proper, then there is an $H^*_T(pt)$-bilinear 
  Poincar{\'e} (or intersection) pairing $H^*_T(X) \otimes H^*_T(X) \to H^*_T(pt)$ defined
 by 
  \[ \langle a , b \rangle = \int_X a \cdot b \quad \/,\]
 where the integral stands for the push-forward to a point.
Equivariant vector bundles have equivariant Chern classes
$c^T_i(-)$, such that $c_j^T(E)\cap -$ is an operator
  $H^T_i(X)\to H^T_{i-2j}(X)$; see \cite[\S1.3]{MR2976939},
\cite[\S2.4]{edidin.graham:eqchow}.

We address the reader to \cite{MR2976939, knutson:noncomplex, AndersonFulton}
for background on equivariant cohomology and
homology.

\subsection{Equivariant K theory} For any $T$-variety $X$, the equivariant $\K$ theory ring $\K_T(X)$ 
is the Grothendieck ring generated by symbols $[E]$, where $E \to X$ is a $T$-equivariant 
vector bundle, modulo the relations $[E]=[E_1]+[E_2]$ for any short exact sequence 
$0 \to E_1 \to E \to E_2 \to 0$ of equivariant vector bundles. The ring $\K_T(X)$ is 
an algebra over $R(T)$, the 
representation ring of $T$. This may be identified with the Laurent polynomial ring 
$\Z[e^{\pm t_1}, \ldots , e^{\pm t_r}]$ where $e^{t_i}$ are characters corresponding to a basis of the
character lattice in the Lie algebra of $T$. 
An introduction to equivariant $\K$ theory may be found in \cite{chriss2009representation}. 
The equivariant $\K$ ring of $X$ admits a `vector bundle duality' involution $^{\vee}: \K_T(X) \to \K_T(X)$ 
mapping the class $[E]$ of a vector bundle to the class $[E^\vee]$ of its dual. This is not an involution of 
$\K_T(pt)$-algebras; it satisfies
 $(e^\lambda)^{\vee} = e^{-\lambda}$. Under mild hypotheses (e.g., $X$~projective)
 there is also a `Serre duality' involution 
 $\calD: \K_T(X) \to \K_T(X)$ inherited from 
(equivariant) Grothendieck-Serre duality and defined by
\[
\calD[F]:=[RHom(F,\omega^\bullet_{X})]
= [\omega_{X}^\bullet] \otimes [F]^\vee \in \K_T(X)
\] 
for $[F]\in \K_T(X)$,
where $\omega^\bullet_{X}\simeq \omega_{X}[\dim X]$ is the
(equivariant) dualizing complex of $X$.
Thus if $X$ is nonsingular, $[\omega^\bullet_{X}]=(-1)^{\dim X}[ \omega_{X}]$ with $\omega_{X}$ the equivariant canonical bundle of $X$.
Observe the multiplicativity
\[ \calD ([E] \otimes [F])
 = \calD ([E]) \otimes[F]^\vee \/.
\]
In later sections of this paper we will primarily be concerned with flag manifolds $X=G/B$, 
with $T$ acting on $X$ by left multiplication. 
In this case $X$ is a smooth projective variety
and the ring $\K_T(X)$ is naturally isomorphic to the Grothendieck group 
$ K_0(\mathfrak{coh}^T(\cO_X))$
of $T$-linearized coherent sheaves on $X$. This follows from the fact that every such 
coherent sheaf has a finite resolution by $T$-equivariant vector bundles. 
There is a $\K_T(pt)$-bilinear pairing
\[ \langle - , - \rangle: \K_T(X) \otimes \K_T(X) \to \K_T(pt) = R(T); \quad \langle [E], [F] \rangle := \int_X E \otimes F = \chi(X; E \otimes F) \/, \] 
where $\chi(X; E)$ is the (equivariant) Euler characteristic, i.e., the virtual representation
\[ \chi(X;E) = \int_X [E] = \sum_i (-1)^i H^i(X; E)  \/. \]

Note that
$$ \langle \calD[E], [F]^{\vee} \rangle =\int_X \calD( [E\otimes F])= \chi(X; E \otimes F)^{\vee}= (\langle [E], [F] \rangle)^{\vee} \:,$$
by equivariant Grothendieck-Serre duality \cite[Theorem 7.6]{hartshorne:AGbook}.
%see also \cite[Lemma~8.10]{AMSS:motivic} for a direct proof based on equivariant localization.

Let $y$ be an indeterminate. The {\em Hirzebruch $\lambda_y$-class\/} of an equivariant 
vector bundle~$E$ is the class 
\[ \lambda_y(E):= \sum_k [\wedge^k E] y^k \in \K_T(X)[y] \/.\] 
The $\lambda_y$-class is multiplicative, i.e., for any short exact sequence 
$0 \to E_1 \to E \to E_2 \to 0$ of equivariant vector bundles there is an equality $\lambda_y(E) = \lambda_y(E_1) \lambda_y(E_2)$ in $\K_T(X)[y]$. We refer to the books \cite{chriss2009representation} or \cite{hirzebruch:topological} for details. 

\subsection{The Chern character}
For a pure-dimensional $T$-variety $X$, Edidin and Graham \cite{edidin2000} 
defined an equivariant Chern character 
\[ \ch_{T} : \K_{T}(X) \to \widehat{H}_*^{T}(X) := \prod_{i \le \dim X} H_{2i}^{T}(X) \] 
such that: 
\begin{itemize} \item If $V \subseteq X$ is a 
$T$-invariant subvariety, then $\ch_{T} [\cO_V]  = [V]_{T} + $l.o.t. (lower order terms). 
(Non-equivariantly, see \cite[Theorem 18.3(5)]{fulton:IT}.) 
\item If $\mathcal{L}$ is an equivariant line bundle with first Chern class $c_1^{T}(\mathcal{L})$, then $\ch_{T} [\mathcal{L}] = e^{c_1^{T}(\mathcal{L})} \cap [X]_{T}$. 

\item $\ch_T$ commutes with pull-backs.

\item If $X$ is smooth, then after identifying $H_*^T(X) \simeq H^*_T(X)$ via Poincar{\'e} duality, $\ch_T$ is a ring homomorphism.

\end{itemize}
A fundamental result is the Grothendieck-Hirzebruch-Riemann-Roch (GHRR) theorem. In the equivariant case, this was proved in 
\cite{edidin2000}. For now we state the following particular form; in \S \ref{s:eqH} below we 
will need more general versions. Let $f:X \to Y$ be a smooth proper $T$-equivariant morphism of 
smooth $T$-varieties, and let $a \in \K_T(X)$. Then 
\begin{equation}\label{E:GHRR} \ch_T f_*(a) = f_* (\ch_T(a) \cdot Td(T_f)) \/. \end{equation}
where $Td(T_f)$ is the equivariant Todd class of the relative tangent bundle of $f$. Recall that 
if $E \to X$ is an equivariant vector bundle with Chern roots $x_1, \ldots, x_e$, then 
\[ Td(E) = \prod_{i=1}^e \frac{x_i}{1-e^{-x_i}} = \prod_{i=1}^e (1+ \frac{1}{2} x_i+ \ldots ) \/;\]
see e.g., \cite[Example~3.2.4]{fulton:IT}.
\section{Operators in equivariant cohomology and $\K$ theory of flag manifolds}

\subsection{Schubert data} Let $G$ be a complex semisimple, simply connected, Lie group, and fix a Borel subgroup $B$ with a maximal torus $T\subseteq B$. Let $B^-$ denote the opposite Borel subgroup. Let $W:=N_G(T)/T$ be the Weyl group, and $\ell:W \to \mathbb{N}$ the associated length function. Denote by $w_0$ the longest element in $W$; then $B^- = w_0 B w_0$. 
Let also $\Delta := \{ \alpha_1, \ldots , \alpha_r \} \subseteq R^+$ denote the set of simple roots included in the set of positive roots for $(G,B)$. Let $\rho$ denote the half sum of the positive root. The simple reflection for the root $\alpha_i \in \Delta$ is denoted by $s_i$ and the {\em minimal} parabolic subgroup is denoted by $P_i$, containing the Borel subgroup $B$. 

Let $X:=G/B$ be the flag variety. It has a stratification by Schubert cells $X(w)^\circ:= BwB/B$ and opposite Schubert cells $Y(w)^\circ:=B^- w B/B$. The closures $X(w):= \overline{X(w)^\circ}$ and $Y(w):=\overline{Y(w)^\circ}$ are the Schubert varieties. With these definitions, $\dim_{\C} X(w) = \mathrm{codim}_{\C} Y(w) = \ell(w)$. The Weyl group $W$ admits a partial ordering, called the Bruhat ordering, defined by $u \le v$ if and only if $X(u) \subseteq X(v)$.   

Let $\cO_w^T:=[\cO_{X(w)}]$ be the Grothedieck class determined by the structure sheaf of $X(w)$ (a coherent sheaf), and similarly $\cO^{w,T}:= [\cO_{Y(w)}]$. The equivariant $K$-theory ring has $K_T(pt)$-bases $\{ \cO_w^T \}_{w \in W}$ and $\{ \cO^{w,T} \}_{w \in W}$ for $w \in W$. Let $\partial X(w) := X(w) \setminus X(w)^\circ$ be the boundary of the Schubert variety $X(w)$, and similarly $\partial Y(w)$ the boundary of $Y(w)$. It is known that the dual bases of $\{ \cO_w^T \} $ and $\{\cO^{w,T} \}$ are given by the classes of the ideal sheaves $\mathcal{I}^{w,T}:= [\cO_{Y(w)}(- \partial Y(w))]$ respectively $
\mathcal{I}_w^T:= [\cO_{X(w)}(- \partial X(w))]$. I.e., 
\begin{equation}\label{E:dualst} \langle \cO_u^T , \mathcal{I}^{v,T} \rangle = \langle \cO^{u,T}, \mathcal{I}_v^T \rangle = \delta_{u,v} \/. \end{equation}
See e.g., \cite[Proposition 4.3.2]{brion:flagv} for the non-equivariant case - the same proof works equivariantly; 
see also \cite{graham.kumar:positivity,anderson.griffeth.miller:positivity}. 
It is also shown in \cite{brion:flagv} that 
\begin{equation}\label{equ:idealstru}
 \cO_w^T= \sum_{v \le w} \mathcal{I}_v^T  \textrm{\quad and }  \quad 
\mathcal{I}_w^T= \sum_{v \le w} (-1)^{\ell(w) - \ell(v)} \cO_v^T \/. 
\end{equation}
For any weight (character) $\lambda$ of $T$, 
we denote by $\mathcal{L}_\lambda$ the $G$-homogeneous line bundle
\[ \mathcal{L}_\lambda = G \times^B \C_{\lambda} \/. \]

Let $P$ be a standard parabolic subgroup of $G$, i.e., $B \subseteq P \subseteq G$. It
is determined by a subset $\Delta_P \subseteq \Delta$; for instance, $\Delta_B = \emptyset$. 
Denote by $W_P$ the subgroup of $W$ generated by the simple reflections $s_i$ such that 
$\alpha_i \in \Delta \setminus \Delta_P$. Let $W^P$ denote the subset of minimal length
representatives of $W/W_P$. By definition, $\ell(wW_P) = \ell(w)$ for $w \in W^P$. 
Similarly to $G/B$, the partial flag manifold $G/P$ has finitely $B$ and $B^-$ 
many orbits - the Schubert cells - indexed by the elements $w \in W^P$:
\[ X(wW_P)^{\circ} = Bw P/P \/; \quad Y(wW_P)^{\circ} = B^-w P/P \/;\]
as before $\dim X(wW_P)^\circ = \mathrm{codim}~Y(wW_P)^\circ = \ell(w W_P)$. 
The Schubert varieties $X(w), Y(w)$ in $G/P$ are the closures of the corresponding Schubert cells.

\subsection{BGG, Demazure, and Demazure-Lusztig operators} Fix a simple root $\alpha_i \in \Delta$ and denote by 
$P_i \subseteq G$ the corresponding minimal parabolic group. Consider the fiber diagram: 
\begin{equation}\label{E:fibrediag} 
\xymatrix@C=50pt{
FP:= G/B \times_{G/P_{i}}
G/B \ar[r]^-{pr_1}\ar[d]^{pr_2} & G/B \ar[d]^{p_{i}} \\ 
G/B \ar[r]^{p_{i}} & G/P_{i}
} 
\end{equation} 
The {\bf Bernstein-Gelfand-Gelfand} (BGG) operator \cite{BGG} is the 
$H^*_T(pt)$-linear morphism $\parcoh_i: H^i_T(X) \to H^{i+2}_T(X)$ defined by 
$\parcoh_i:= (p_i)^* (p_i)_*$. The same geometric definition gives the {\bf Demazure} operator
$\partial_i: \K_T(X) \to \K_T(X)$ in the (equivariant) $\K$-theory, linear over $\K_T(pt)$; see
\cite{demazure:desingularisations}. These operators satisfy
\begin{equation}\label{equ:BGGonstru}
\parcoh_i [X(w)]_T = \begin{cases} [X(ws_i)]_T & \textrm{ if } ws_i>w \/; \\ 0 & \textrm{ otherwise } \/. \end{cases} \quad 
 \partial_i(\cO_w^T) = \begin{cases} \cO_{ws_i}^T & \textrm{ if } ws_i>w \/; \\ \cO_w^T & \textrm{ otherwise } \/. \end{cases} 
\end{equation}
From this, one deduces that both operators satisfy the same commutation and braid relations as those 
for the elements of $W$. In cohomology, $(\parcoh_i)^2 = 0$, while in $\K$-theory
$\partial_i^2 = \partial_i$. 

The relative cotangent bundle of the projection $p_i$ is $T^*_{p_i}= \mathcal{L}_{\alpha_i}$. Define the $H^*_T(pt)$-algebra automorphism $s_i: H^*_T(G/B) \to H^*_T(pt)$ by 
\begin{equation}\label{E:sidef} s_i  = \id + c_1^T(T^*_{p_i})\parcoh_i = \id + c_1^T(\mathcal{L}_{\alpha_i}) \parcoh_i \/. \end{equation}
It was proved by Knutson \cite{knutson:noncomplex} that this is an automorphism
induced by the right Weyl group action on $G/B$; see \cite{aluffi.mihalcea:eqcsm} and also \cite{MNS}, 
where both left and right actions are studied. Using this automorphism, 
the {\bf cohomological Demazure-Lusztig (DL) operators}
are $H^*_T(pt)$-linear endomorphism of $H^*_T(G/B)$ defined by 
\begin{equation}\label{E:Ticohdef} \mathcal{T}_i^\mathrm{coh} = \parcoh_i - s_i \/; \quad  \mathcal{T}_i^{coh,\vee} = \parcoh_i + s_i  \/. \end{equation}
These operators satisfy the same braid and commutation relations as the BGG operators, and, in addition
$(\mathcal{T}_i^{coh})^2 = (\mathcal{T}_i^{coh,\vee})^2 = \id$; see \cite{aluffi.mihalcea:eqcsm}. 
In other words, these give a twisted representation
of the Weyl group $W$ on $H^*_T(G/B)$. This representation was studied earlier by Lascoux, Leclerc and Thibon \cite{LLT:twisted},
and by Ginzburg \cite{ginzburg:methods} in relation to the degenerate Hecke algebra. The operators are adjoint to each other, in the sense that for any $a,b \in H^*_T(G/B)$,  
\[ \langle \mathcal{T}_i^{coh}(a), b \rangle = \langle a, \mathcal{T}_i^{coh,\vee}(b) \rangle \/. \]
It is convenient to consider a homogenized version of this operator. Add a formal variable $\hbar$ of cohomological 
complex degree $1$. Then the homogenized operators are
 \[ \mathcal{T}_i^{coh,\hbar} = \hbar \parcoh_i - s_i \/; \quad \mathcal{T}_i^{coh,\vee,\hbar} = \hbar \parcoh_i  + s_i \/. \]
The variable $\hbar$ will arise geometrically from the $\C^*$-action by dilation on $T^*(G/B)$. The restriction of this action
to the zero-section $G/B \hookrightarrow T^*(G/B)$ is trivial, and $H^*_{T \times \C^*}(G/B) = H^*_T(G/B)[\hbar]$,
where $\hbar$ is interpreted as a generator of $H^2_{\C^*}(pt)$. 

We define next the $\K$-theoretic version of the DL operator.
Fix an indeterminate $y$; later, we will set $y = - e^{-\hbar}$. 
%\begin{definition}\label{def:hecke} Let $\alpha_i \in \Delta$ be a simple root. 
Define the {\bf $\K$-theoretic Demazure-Lusztig (DL) operators}
\begin{equation}\label{E:TiKdef} \mathcal{T}_i: = \lambda_y(T^*_{p_i}) \partial_i - \id; \quad \mathcal{T}_i^\vee: = \partial_i \lambda_y(T^*_{p_i}) - \id \/. \end{equation} %\end{definition}
The operators $\mathcal{T}_i$ and ${\mathcal{T}_i}^\vee$ are $\K_T(pt)[y]$-module endomorphisms of $\K_T(X)[y]$.

\begin{remark}\label{rmk:convos} The operator ${\mathcal{T}_i}^\vee$ was defined 
by Lusztig \cite[Eq.~(4.2)]{lusztig:eqK} in relation to affine Hecke algebras and 
equivariant $\K$ theory of flag varieties. The `dual' operator $\mathcal{T}_i$ 
arises naturally in the study of motivic Chern classes of Schubert cells 
\cite{AMSS:motivic}. In an algebraic form, 
$\mathcal{T}_i$ appeared recently in \cite{brubaker.bump.licata, lee.lenart.liu:whittaker,mihalcea2019whittaker}, 
in relation to Whittaker functions. The left versions of these operators are studied in \cite{MNS}.\end{remark}
As in the cohomological case, the Demazure and Demazure-Lusztig (DL) operators are adjoint to each other and $ \partial_i$ is self-adjoint: for any $a, b \in K_T(X)$, \[ \langle \mathcal{T}_i(a), b \rangle = \langle a , {\mathcal{T}_i}^\vee(b) \rangle 
\quad \text{and} \quad  \langle \partial_i(a), b \rangle = \langle a , \partial_i(b) \rangle \/. \] 
See \cite[Lemma 3.3]{AMSS:motivic} for a proof. 
\begin{prop}[\cite{lusztig:eqK}]\label{prop:hecke-relations} 
	The operators $\mathcal{T}_i$ and $\mathcal{T}_i^\vee$ satisfy the usual commutation and braid relations for the group $W$. For each simple root $\alpha_i \in \Delta$ the following quadratic formula holds:\[ (\mathcal{T}_i + \id) (\mathcal{T}_i + y) = ({\mathcal{T}_i}^\vee + \id) ({\mathcal{T}_i}^\vee + y) = 0 \/. \] 
\end{prop}

An immediate corollary of the quadratic formula is that for $y \neq 0$, the operators $\mathcal{T}_i$ and ${\mathcal{T}_i}^\vee$ are invertible. In fact,
\begin{equation}\label{E:invTi} \mathcal{T}_i^{-1} = - \frac{1}{y} \mathcal{T}_i - \frac{1+y}{y} \id 
\end{equation} 
as operators on $\K_T(X)[y,y^{-1}]$. The same formula holds when $\mathcal{T}_i$ is exchanged with $\mathcal{T}_i^\vee$. 

Consider next the localized equivariant K theory ring 
$$\K_T(G/B)_{\textit{loc}} := \K_T(G/B) \otimes_{\K_T(pt)} \Frac( \K_T(pt))$$ where $\Frac$ denotes the fraction field. The Weyl group elements $w\in W$ are in bijection with the torus fixed points $e_w \in G/B$. Let $\iota_w:=[\cO_{e_w}] \in K_T(G/B)_{loc}$ be the class of the structure sheaf of $e_w$. By the localization theorem, the classes $\iota_w$ form a basis for the localized equivariant K theory ring; we call this the {\em fixed point basis}. 

We need the following lemma, whose proof can be found e.g., in \cite[Lemma 3.7]{AMSS:motivic}.
\begin{lemma}\label{lem:actiononfixedpoint} The following formulas hold in $\K_T(G/B)_{loc}$:

(a) For any weight $\lambda$, $\mathcal{L}_\lambda \cdot \iota_w = e^{w \lambda} \iota_w$;

(b) For any simple root $\alpha_i$, \[ \partial_i (\iota_w) =\frac{1}{1-e^{w\alpha_i}}\iota_{w}+\frac{1}{1-e^{-w\alpha_i}}\iota_{ws_{i}} \/; \]

(c) The action of the operator $\calT_i$ on the fixed point basis is given by the following formula
\[\calT_i(\iota_{w})=-\frac{1+y}{1-e^{-w\alpha_i}}\iota_{w}+\frac{1+ye^{-w\alpha_i}}{1-e^{-w\alpha_i}}\iota_{ws_{i}}.\]

(d) The action of the adjoint operator $\calT_i^\vee$ is given by

\[\calT^\vee_i(\iota_{w})=-\frac{1+y}{1-e^{-w\alpha_i}}\iota_{w}+\frac{1+ye^{w\alpha_i}}{1-e^{-w\alpha_i}}\iota_{ws_{i}} \/.\]

(e) The action of the inverse operator $(\calT^\vee_i)^{-1}$is given by
\[(\calT^\vee_i)^{-1}(\iota_{w})=- \frac{1+y^{-1}}{1- e^{w\alpha_i}}\iota_{w}-\frac{y^{-1}+e^{w\alpha_i}}{1-e^{-w\alpha_i}}\iota_{ws_{\alpha_i}}\/.\]

\end{lemma}

We also record the action of several specializations of the Demazure-Lusztig operators, see \cite[Lemma 3.8]{AMSS:motivic}.
\begin{lemma}\label{lemma:yspec} 
(a) The specializations \[ (\calT_i)_{y=0} = \partial_i - \id \/; \quad (\mathcal{T}_i^\vee)_{y=0} = \partial_i-\id \/; \] Further, for any $w \in W$, the following hold:
\[ (\partial_i - \id) (\mathcal{I}_w^T )= \begin{cases} \mathcal{I}_{ws_i}^T & \textrm{ if } ws_i > w; \\ - \mathcal{I}_{w}^T & \textrm{ if } ws_i < w \/. \end{cases} \quad \partial_i (\mathcal{O}_w^T) = \begin{cases} \mathcal{O}_{ws_i}^T & \textrm{ if } ws_i > w; \\  \mathcal{O}_{w}^T & \textrm{ if } ws_i < w \/. \end{cases} \]

(b) Let $w \in W$. Then the specializations at $y=-1$ satisfy \[  (\mathcal{T}_i)_{y=-1} (\iota_w) = \iota_{ws_i} \/. \]
In other words, this specialization is compatible with the right Weyl group multiplication.\end{lemma}

\subsection{Leading terms of DL operators}
Next we utilize the grading 
%dimension filtration
induced by the equivariant Chern character in order to identify the 
`initial terms' of the Demazure-Lusztig operators
as certain operators on equivariant (co)homology related to the degenerate Hecke algebra. These
operators appeared as convolution operators in \cite{ginzburg:methods}, and determine
the Chern-Schwartz-MacPherson classes of Schubert cells \cite{aluffi.mihalcea:eqcsm}. 

As usual, $X=G/B$ but we consider the extended torus ${{A}}:= T \times \C^*$ where $\C^*$ acts trivially. 
(This is the restriction of the action of ${A}$ on $T^*(X)$, where $\C^*$ acts by dilation. The
CSM and motivic Chern classes considered later in this paper are naturally $\C^*$-equivariant; this justifies the use of 
the extended torus.) Assume that $y=-e^{-\hbar}$, i.e., more precisely,
\[ \ch_{\mathbb{C}^*}(y)=- e^{- \hbar}=  - 1+ \hbar + O(\hbar^2) \in \widehat{H}^*_{\C^*}(pt)\/. \]
We analyze the relation between the cohomological and $\K$-theoretic DL operators. 
\begin{prop}\label{prop:initial} Let $w \in W$ and consider the Grothendieck class $\cO_w^{{A}} \in \K_A(X)$ 
for the Schubert variety $X(w)$. 
Then \[ \ch_{A}( \mathcal{T}_i (\cO_w^{{A}}))  = \mathcal{T}_i^{coh,\hbar} [X(w)]_{{{A}}} + l.o.t. \] 
and 
\[ \ch_{A}( \mathcal{T}^\vee_i (\cO_w^{{A}}))  = \mathcal{T}_i^{coh,\vee,\hbar} [X(w)]_{{{A}}} + l.o.t., \]
where l.o.t. are terms in $\prod_{i < \ell(w)} H_{2i}^{{A}}(X)$. \end{prop}
%In particular, the Chern character of the motivic Chern class $\ch_{\mathbb{A}} (MC_y[ X(w)^\circ \hookrightarrow X])$ is an element in $\prod_{i \le 0} H_i^{\mathbb{A}}(X,\mathbb{Q})$, and the degree $0$ term equals the homogenized CSM class 
%$\csmTh(X(w)^\circ) \in H_0^{\mathbb{A}}(X,\mathbb{Z}) \subseteq H_0^{\mathbb{A}}(X,\mathbb{Q}) $. 

\begin{proof} Since $X=G/B$ is non-singular, the Chern character is a ring homomorphism, thus for any invariant subvariety $Z \subseteq X$ and any equivariant line bundle $\mathcal{L}$, 
\begin{equation}\label{eq:lb}  \ch_{{A}}([\cO_Z]_A \cdot \mathcal{L}) = [Z]_{{A}} + c_1^{{A}}(\mathcal{L}) \cdot [Z]_{{A}} + l.o.t. \end{equation} 
We take $Z= X(w)$, and we have two cases: either $w < ws_i$ or $w > ws_i$. If $w < ws_i$ then 
$\partial_i(\cO_w^{{A}}) = \cO_{ws_i}^{{A}}$. Using this, we obtain 
\[ \begin{split} \ch_{{A}}( \mathcal{T}_i (\cO_w^{{A}})) = & \ch_{{A}}( \cO_{ws_i}^{{A}} +y \mathcal{L}_{\alpha_i} \cdot \cO_{ws_i}^{{A}} - \cO_w^{{A}})   \\ 
= & \ch_{{A}}(\cO_{ws_i}^{{A}}) - e^{-\hbar} e^{c_1^{{A}}(\mathcal{L}_{\alpha_i})}
 \ch_{{A}}(\cO_{w s_i}^{{A}}) - [X(w)]_{{A}} + l.o.t. \\ = & \ch_A(\cO_{ws_i}^{{A}}) - (1 - \hbar)(1 + c_1^A(\mathcal{L}_{\alpha_i})) \ch_A(\cO_{ws_i}^{{A}}) - [X(w)]_{{A}} + l.o.t. \\ = & \hbar [X(ws_i)]_A -  c_1^A(\mathcal{L}_{\alpha_i}) [X(ws_i)]_A - [X(w)]_{{A}} + l.o.t. \\ = & \hbar \parcoh_i [X(w)]_A - (\id +c_1^{{A}}(\mathcal{L}_{{\alpha_i}}) \parcoh_i)[X(w)]_{{A}}+ l.o.t. \end{split} \]
where $l.o.t. \in \prod_{i < \ell(w)} H_{i}^{{A}}(X)$. By \cite[(3)]{aluffi.mihalcea:eqcsm} (which
uses a different sign convention) the last expression equals 
$$(\hbar \parcoh_i - s_i) [X(w)]_{{A}} + l.o.t. =\mathcal{T}_i^{coh,\hbar} [X(w)]_{{A}} + l.o.t.$$ and we are done in this case. If $w > ws_i$ then $\partial_i(\cO_w) = \cO_w$, $\partial_i[X(w)]_{{A}} = 0$ and $s_i[X(w)]_{{A}} = [X(w)]_{{A}}$ by \cite[(4)]{aluffi.mihalcea:eqcsm}. Then a similar, but simpler calculation, proves the first part of the proposition. 

For the second statement we start by observing that for $a \in \K_T(X)$, by the Grothendieck-Hirzebruch-Riemann-Roch (GHRR) 
\begin{equation}\label{E:BGG-ch} \ch_A \partial_i (a) = \ch_A p_i^* (p_i)_* (a) = p_i^* \ch_A((p_i)_*(a)) = \parcoh_i (\ch_A(a) Td(T_{p_i})) \/.\end{equation}
Then the second statement can be proved as follows. By the GHRR theorem,
\[ \begin{split} \ch_{{A}}( \mathcal{T}^\vee_i (\cO_w^{{A}})) = & \ch_{{A}}\partial_i( \cO_{w}^{A} +y \mathcal{L}_{\alpha_i} \cdot \cO_{w}^{{A}}) - \ch_A(\cO_w^{{A}})   \\ 
=&\parcoh_i\bigg(\ch_A (\cO_{w}^{A}+y \mathcal{L}_{\alpha_i} \cdot \cO_{w}^{{A}})Td(T_{p_i})\bigg)-\ch_A(\cO_w^{{A}})\\
= & \parcoh_i \bigg( \ch_A (\cO_w^A)Td(T_{p_i})(1- e^{-\hbar + c_1^A(\mathcal{L}_{\alpha_i})})\bigg) -\ch_A(\cO_w^{{A}}) 
\end{split} \]
Observe that $Td(T_{p_i})(1- e^{-\hbar + c_1^A(\mathcal{L}_{\alpha_i})}) = \hbar + c_1^A(\mathcal{L}_{-\alpha_i}) + $
(terms of degree $\ge 2$) in {\em cohomology}. Then the last expression equals
\[ \begin{split} \parcoh_i & \bigg((\hbar +c_1(\mathcal{L}_{-\alpha_i})[X(w)]_A\bigg)-[X(w)]_A+l.o.t\\
 = & \hbar \parcoh_i [X(w)]_{{A}} + (\id +c_1^{{A}}(\mathcal{L}_{{\alpha_i}}) \parcoh_i)[X(w)]_{{A}}+ l.o.t.\\
 =& (\hbar \parcoh_i + s_i) [X(w)]_{{A}} + l.o.t. \\
 =&\mathcal{T}_i^{coh,\vee,\hbar} [X(w)]_{{A}} + l.o.t. \end{split} \]
%
%\[ \begin{split} \ch_{{A}}( \mathcal{T}^\vee_i (\cO_w^{{A}})) = & \ch_{{A}}\partial_i( \cO_{w}^{A} +y \mathcal{L}_{\alpha_i} \cdot \cO_{w}^{{A}}) - \ch_A(\cO_w^{{A}})   \\ 
%=&\parcoh_i\bigg(\ch_A (\cO_{w}^{A}+y \mathcal{L}_{\alpha_i} \cdot \cO_{w}^{{A}})Td(T_{p_i})\bigg)-\ch_A(\cO_w^{{A}})\\
% = & \parcoh_i \bigg( \ch_A (\cO_w^A)Td(T_{p_i})(1- e^{-\hbar + c_1^A(\mathcal{L}_{\alpha_i})})\bigg) -\ch_A(\cO_w^{{A}}) \\ 
%=&\parcoh_i(-c_1(\mathcal{L}_{\alpha_i})[X(w)]_A+\hbar [X(w)]_A)-[X(w)]_A+l.o.t\\
% = & \hbar \parcoh_i [X(w)]_{\mathbb{A}} + (\id +c_1^{{A}}(\mathcal{L}_{{\alpha_i}}) \parcoh_i)[X(w)]_{{A}}+ l.o.t.\\
% =& (\hbar \parcoh_i + s_i) [X(w)]_{{A}} + l.o.t. \\
% =&\mathcal{T}_i^{coh,\vee,\hbar} [X(w)]_{{A}} + l.o.t. .\end{split} \]
Here the third equation follows from the definition of $s_i$ from \eqref{E:sidef}, and the second from the general fact that for every weight $\lambda$, $c_1(\mathcal{L}_\lambda)\parcoh_i=\parcoh_ic_1(\mathcal{L}_{s_i\lambda})-\langle\lambda,\alpha_i^\vee\rangle$. 
This can be proved by e.g., adapting parts (a) and (b) of \Cref{lem:actiononfixedpoint} to the cohomological context. 
\end{proof}

\section{Equivariant motivic Chern  classes} 
\subsection{Preliminaries about motivic Chern classes}
We recall the definition of the motivic Chern classes, following 
\cite{brasselet.schurmann.yokura:hirzebruch}. For now let $X$ 
be a quasi-projective, complex algebraic variety, with an action of $T$. 
First we recall the definition of the (relative) motivic Grothendieck group 
${\K}_0^T(var/X)$ of varieties over $X$, mostly following Looijenga's notes 
\cite{looijenga:motivic}; see also Bittner \cite{bittner:universal}. For simplicity, 
we only consider the $T$-equivariant quasi-projective context (replacing the `goodness' assumption 
in~\cite{bittner:universal}), which is enough for all applications in this paper. 
The group ${\K}_0^T(var/X)$ is the quotient of the 
free abelian group generated by symbols $[f: Z \to X]$ where $Z$ is a quasi-projective 
$T$-variety and $f: Z \to X$ is a $T$-equivariant morphism modulo the additivity relations 
$$[f: Z \to X] = [f: U \to X] + [f:Z \setminus U \to X]$$ for $U \subseteq Z$ an open invariant subvariety. 
For any equivariant morphism $g:X \to Y$ of quasi-projective $T$-varieties there are 
well defined push-forwards $g_!: \K_0^T(var/X) \to \K_0^T(var/Y)$ (given by composition) 
and pull-backs $g^*:\K_0^T(var/Y) \to \K_0^T(var/X)$ (given by fiber product); see \cite[\S 6]{bittner:universal}.
There are also external products
$$\times: \K_0^T(var/X)\times \K_0^T(var/X') \to \K_0^T(var/X\times X'); 
\quad [f]\times [f']\mapsto [f\times f'],$$
which are  $\K_0^T(var/pt)$-bilinear and
commute with push-forward and pull-back.
If $X=pt$, then ${\K}_0^T(var/pt)$ is a ring with this external product, and the groups ${\K}_0^T(var/X)$ also acquire by the external product a module structure over $\K_0^T(var/pt)$ such that push-forward $g_!$ and pull-back $g^*$ are $\K_0^T(var/pt)$-linear.

\begin{remark} For any variety $X$, similar functors can be defined on the ring of constructible functions 
$\mathcal{F}(X)$, and the Grothendieck group $\K_0(var/X)$ may be regarded as a motivic version of 
$\mathcal{F}(X)$. In fact, there is a map $e: \K_0(var/X) \to \mathcal{F}(X)$ sending $[f:Y \to X] \mapsto f_!(\one_Y)$, where $f_!(\one_Y)$ is defined using compactly supported Euler characteristic of the fibers. The map $e$ is a group homomorphism, and if $X=pt$ then $e$ is a ring homomorphism. The constructions extend equivariantly, with $\mathcal{F}^T(X)\subseteq \mathcal{F}(X)$ the subgroup of $T$-invariant constructible functions.
 \end{remark}
The following theorem was proved in the non-equivariant case by Brasselet, Sch{\"u}rmann and Yokura \cite[Theorem 2.1]{brasselet.schurmann.yokura:hirzebruch}. Minor changes in the argument are needed in the equivariant case - see  also \cite{feher2018motivic,AMSS:motivic}. In the upcoming paper \cite{AMSS:MCcot} we will reprove the theorem below and relate equivariant motivic Chern classes to certain classes in the equivariant K-theory of the cotangent bundle as defined by Tanisaki \cite{tanisaki:hodge} with the help of equivariant mixed Hodge modules.

\begin{theorem}\label{thm:existence}\cite[Theorem 4.2]{AMSS:motivic} Let $X$ be a quasi-projective, non-singular, complex algebraic variety with an action of the torus $T$. There exists a unique natural transformation $MC_y: \K_0^T(var/X) \to \K_T(X)[y]$ satisfying the following properties:
\begin{enumerate} \item[(1)] It is functorial with respect to push-forwards via $T$-equivariant proper morphisms of non-singular, quasi-projective varieties. 

\item[(2)] It satisfies the normalization condition \[ MC_y[\id_X: X \to X] = \lambda_y(T^*X) = \sum y^i [\wedge^i T^*X]_T \in \K_T(X)[y] \/. \]
\end{enumerate}
The transformation $MC_y$ satisfies the following properties:
\begin{enumerate}

\item[(3)] It is determined by its image on classes $[f:Z \to X]=f_![\id_Z]$ where $Z$ is a non-singular,  irreducible,
quasi-projective algebraic variety and $f$ is a $T$-equivariant proper morphism.

\item[(4)] It satisfies a Verdier-Riemann-Roch (VRR) formula: for any smooth, $T$-equivariant morphism $\pi: X \to Y$ of quasi-projective and non-singular algebraic varieties, and any $[f: Z \to Y] \in \K_0^T(var/Y)$, the following holds:
\[ MC_y[\pi^* f:Z \times_Y X \to X] = \lambda_y(T^*_\pi) \cap \pi^* MC_y[f:Z \to Y]\/. \]

\end{enumerate} \end{theorem}

If one forgets the $T$-action, then the equivariant motivic Chern class above recovers the non-equivariant motivic Chern class from
\cite{brasselet.schurmann.yokura:hirzebruch} (either by its construction, or by the properties (1)-(3) from Theorem~\ref{thm:existence}
and the corresponding results from \cite{brasselet.schurmann.yokura:hirzebruch}). 

\begin{remark} \Cref{thm:existence} and its proof work more generally for a possibly singular, quasi-projective $T$-equivariant base variety
 $X$, if one works with the Grothendieck group of $T$-equivariant coherent $\cO_X$-sheaves in the target
(cf.~\cite[Remark 2.2]{feher2018motivic}), i.e.,
$$ MC_y: \K_0^T(var/X) \to \K_0(\mathfrak{coh}^T(\cO_X))[y] \:.$$
Moreover, $MC_y$ commutes with exterior products:
\begin{equation}\label{MC-product}
MC_y[f\times f': Z\times Z'\to X\times X']=MC_y[f: Z\to X] \boxtimes MC_y[f': Z'\to X'] \:.
\end{equation}
This follows as in the non-equivariant context~\cite[Corollary~2.1]{brasselet.schurmann.yokura:hirzebruch} from part (3) of Theorem~\ref{thm:existence} 
and the multiplicativity of the equivariant $\lambda_y$-class for smooth and quasi-projective $T$-varieties $X,X'$:
$$\lambda_y(T^*(X\times X')) = \lambda_y(T^*X)\boxtimes \lambda_y(T^*X') \in \K_T(X\times X')[y]\:.$$
\end{remark}

\begin{remark} \label{rem:rigid}
The equivariant $\chi_y$-genus of a $T$-variety $Z$ is by definition
$$\chi_y(Z):=MC_y([Z\to pt]) \in \K_T(pt)[y]\:.$$
By \emph{rigidity} of the $\chi_y$-genus (see~\cite[\S 2.5]{feher2018motivic} and 
\cite[Theorem~ 7.2]{weber1}), it contains no information about the action of $T$; it is equal to 
the non-equivariant $\chi_y$-genus under the embedding 
$\mathbb{Z}[y]\to \K^T(pt)[y]$.
\end{remark}

In what follows, the variety $X$ will usually be understood from the context. 
If $Y \subseteq X$ is a $T$-invariant subvariety, not necessarily closed, 
denote by \[ MC_y(Y) := MC_y[Y \hookrightarrow X] \/. \] If $i: Y \subseteq X$ 
is closed nonsingular subvariety and $Y' \subseteq Y$, then by functoriality 
$$MC_y[Y' \hookrightarrow X] = i_* MC_y [Y' \hookrightarrow Y]$$
 (K-theoretic push-forward). For instance if $Y'=Y$  then 
 $$MC_y[i:Y \to X]= i_*(\lambda_y(T^*Y) \otimes [\cO_Y])$$ as an element in 
 $\K_T(X)$. We will often abuse notation and suppress the push-forward notation. 
 Similarly for the other transformations discussed in later sections.

\subsection{Motivic Chern classes of Schubert cells}
Assume now that $X=G/B$. The following theorem, proved in \cite[Corollary 5.2 and Theorem 6.2]{AMSS:motivic}, 
allows us to calculate recursively the motivic Chern 
classes of Schubert cells. 
\begin{theorem}\label{thm:MC+dual} Let $w \in W$. Then the motivic Chern class $MC_y(X(w)^\circ)$ is given by \[ MC_y(X(w)^\circ) =  \mathcal{T}_{w^{-1}}(\cO_{\id}^T) \/.\] 
\end{theorem}

Following \cite[Remark~6.4]{AMSS:motivic}, we introduce an operator which will yield the (Poincar{\'e}) duals of motivic Chern classes.
For each simple root $\alpha_i \in \Delta$, define
\begin{equation}\label{equ:defofL} \operL_i = \partial_i + y (\partial_i \mathcal{L}_{\alpha_i} + \id) = - y (\mathcal{T}_i^\vee)^{-1}
= \mathcal{T}_i^\vee + (1+y)\id\/.
\end{equation} 
%Note that $(\mathcal{L}_i)_{y=0}=\partial_i$.
Since $\mathcal{T}_i^\vee$ satisfy the usual braid and commutativity relations, so do these operators. 
Hence, $\operL_w$ is well defined for any $w\in W$. Define the following elements 
in the equivariant $\K$ theory ring:
%\[ MC (X(\id)^\circ) = \cO_{\id}; \quad MC(X(w)^\circ):= \mathcal{T}_{w^{-1}}(\cO_{\id}) \/.\]
% \widetilde{MC}(Y(w_0)^\circ) = \cO^{w_0}; \quad
\[ \widetilde{MC}_y(Y(w)^\circ)):=\operL_{w^{-1}w_0}(\cO^{w_0, T});\quad  \widetilde{MC}_y(X(w)^\circ)=\operL_{w^{-1}}(\cO_{\id}^T)\/. \]
By definition $\widetilde{MC}_y(Y(w)^\circ))$ and $\widetilde{MC}_y(X(w)^\circ)$ are elements in $\K_T(X)[y]$.\begin{footnote}
{If $w_0^L$ denotes the left Weyl group action by $w_0$, as in \cite{MNS}, then 
$w_0^L . \widetilde{MC}_y(X(w)^\circ) = \widetilde{MC}_y(Y(w_0w)^\circ)$. 
This generalizes the more familiar formula from Schubert calculus: $w_0^L.[X(w)]_T = [Y(w_0w)]_T$.}\end{footnote}
\begin{theorem}\cite[Theorem~6.2]{AMSS:motivic}\label{thm:MCtilde}  The classes $\widetilde{MC}_y(Y(w)^\circ))$ are orthogonal to the motivic Chern 
Chern classes: for any $u, v \in W$, 
\begin{equation}\label{E:duality} \langle MC_y(X(u)^\circ), \widetilde{MC}_y(Y(v)^\circ) \rangle = \delta_{u,v} \prod_{\alpha > 0} (1+ y e^{-\alpha}) \/. \end{equation}
\end{theorem}

Note that $\prod_{\alpha > 0} (1+ y e^{-\alpha})=\lambda_y(T^*_{w_0}X)$.

\begin{remark}\label{rem:dualandnormalized} Another family of classes dual to motivic Chern
classes is given by a certain Serre dual variant of Segre motivic classes. 
Combining \cite[Theorem~8.11]{AMSS:motivic} and \cite[Theorem~7.1]{MNS}) (see also 
\cite{feher2018motivic}), one obtains the remarkable equality:
\begin{equation}
\frac{\widetilde{MC}_y(Y(v)^\circ)}{\prod_{\alpha > 0} (1+ y e^{-\alpha})}= (-y)^{\dim G/B-\ell(v)}\frac{\calD (MC_y(Y(v)^\circ))}{\lambda_y(T^*X)}
\in K_T(X)[y]_S[y^{-1}] \/.
\end{equation}
Here $\calD$ denotes the (Grothendieck-Serre) duality, extended to the parameter $y$ via $y^n\mapsto y^{-n}$, and
$\K_T(X)[y]_S$ is appropriately localized so that 
$\lambda_y(T^*X)$ is invertible 
(see~\cite[Remark 8.9]{AMSS:motivic}). We note that one may define these 
`Serre-Segre' motivic classes for any partial flag manifold $G/P$, and they are always dual to motivic Chern 
classes of Schubert cells; see \cite[Theorem~7.2]{MNS}.
Geometrically, the duality above is expected to arise from a transversality formula, generalizing to 
$\K$ theory the results from \cite{schurmann:transversality}. In cohomology, this was explained in \cite[\S 7]{AMSS:shadows}.
\end{remark}
Consider the expansions of the equivariant motivic Chern classes, 
\begin{equation}\label{E:schub} MC_y (X(w)^\circ) = \sum_{ u \le w} c_{u,w}(y;e^t) \cO_u^{T}\/.
\end{equation} 
The equivariant K Chevalley formula, used to multiply a Schubert class by the line bundle~$\mathcal{L}_{\alpha_i}$ \cite{lenart2007affine, pittie1999pieri},
and \Cref{thm:MC+dual}, give a recursive procedure to calculate the the motivic Chern classes of Schubert cells. The recursion also implies that 
coefficients $c_{u,w}(y;e^t)$ are polynomials $c_{u,w}(y,e^t) \in \Z[e^{\pm \alpha_1}, \ldots , e^{ \pm \alpha_{r}}][y]$ in the characters associated to the simple roots. We provide next a few calculations for the motivic Chern classes of the flag manifolds~$\Pbb^1$ and $\Fl(3)$. 

\begin{example}[Equivariant motivic Chern classes for $\Pbb^1$]\label{ex:P1} 
The equivariant motivic Chern classes for $\Pbb^1$ are : 
\[ MC_y(X(\id)) = \cO_{\id}^{T}; \quad MC_y(X(s)^\circ) = (1+ e^{-\alpha_1} y)\cO_{\Pbb^1}^{T} - (1 + (1+ e^{-\alpha_1})y) \cO_{\id}^{T} \/. \]

\end{example}
%The (homogenized) CSM classes are 
%\[ \csm (X(\id)) = [X(\id)]_{T}; \quad \csm(X(s)^\circ) = (\hbar + \alpha_1) [\bbP^1]_{T} + [e_{\id}]_{T}\in H_0^{T\times \bbC^*}(\bbP^1) \/, \]
%where the extra torus $\bbC^*$ acts trivially on $\bbP^1$, and the equivariant parameter $\hbar\in H_\bbC^*(pt)$ is determined by $y=-e^{-\hbar}$.

%\begin{example}[Motivic classes for $\Fl(3)$]\label{ex:Fl3} The non-equivariant motivic Chern classes for Schubert cells in $\Fl(3)$ are : 
%\[ \begin{split} MC_y(X(\id)) = & \cO_{\id}; \\ MC_y(X(s_1)^\circ) = & (1+y) \cO_{s_1} - (1+2y) \cO_{\id}\/; \\ MC_y(X(s_2)^\circ) = & (1+y) \cO_{s_2} - (1+2y) \cO_{\id} \/; \\ MC_y(X(s_1 s_2)^\circ) = & (1+y)^2 \cO_{s_1 s_2} -  (1+y) (1+2y) \cO_{s_1} - (1+y)(1+3y) \cO_{s_2} + (5y^2+ 5y+1) \cO_{\id} \/;  \\ MC_y(X(s_2 s_1)^\circ) = & (1+y)^2 \cO_{s_2 s_1} -  (1+y) (1+2y) \cO_{s_2} - (1+y)(1+3y) \cO_{s_1} + (5y^2 + 5y+1) \cO_{\id} \/; \\ MC_y(X(s_1 s_2 s_1)^\circ) = & (1+y)^3 \cO_{s_1 s_2 s_1} -  (1+y)^2 (1+2y) (\cO_{s_1s_2} + \cO_{s_2 s_1}) + \\ & (1+y)(5 y^2 + 4 y +1)( \cO_{s_1 } + \cO_{s_2})    - (8y^3+11y^2+ 5y+1) \cO_{\id} \/. \end{split} \] \end{example}

\begin{example}[The Schubert cell $X(s_1 s_2)^\circ$]\label{ex:FL3} The equivariant motivic Chern classes for larger flag manifolds are much more complicated. For instance, the equivariant motivic Chern class of the Schubert cell $X(s_1 s_2)^\circ \subseteq \mathrm{Fl}(3)$ is \[ \begin{split} MC_y(X(s_1s_2)^\circ) = & (1+ e^{-\alpha_1}y)(1+ e^{-(\alpha_1 + \alpha_2)}y) \cO_{s_1 s_2}^{T} - \\ & (1+ e^{-\alpha_1}y)(1+(1+ e^{-(\alpha_1+ \alpha_2)})y) \cO_{s_1}^{T}  - \\ & (1 + (1 + e^{-\alpha_1})(1+e^{-\alpha_2})y + e^{-\alpha_2}(1+ e^{-\alpha_1}+ e^{-2\alpha_1})y^2)\cO_{s_2}^{T} + \\ & (1 + (2+ e^{-\alpha_1} + e^{-\alpha_2} + e^{-(\alpha_1+\alpha_2)})y) \cO_{\id}^{T}+ \\ & (1+ e^{-\alpha_1} + e^{-\alpha_2} + e^{-(\alpha_1+\alpha_2)} + e^{-(2 \alpha_1 + \alpha_2)})y^2\cO_{\id}^{T} \/. \end{split} \]
\end{example}
The expressions above encode a remarkable amount of information. 
For instance, a simple verification in \Cref{ex:FL3} shows that the expression for the {\em sum} of 
the coefficients simplifies dramatically and equals $y^2$, reflecting the geometric fact 
that we deal with a cell of dimension $2$. We will prove this and more 
in~\Cref{prop:sum-st} and \Cref{prop:specialize} below. 
For now, we 
also provide some examples of non-equivariant motivic Chern classes. These are obtained 
from the equivariant ones by making the substitution $e^\lambda \mapsto 1$ 
for each weight $\lambda$. 
  
\begin{example}\label{ex:motcells} The following are the non-equivariant motivic Chern classes of Schubert cells in $\Fl(3)$: \[ \begin{split} MC_y(X(\id)) = & \cO_{\id}; \\ MC_y(X(s_1)^\circ) = & (1+y) \cO_{s_1} - (1+2y) \cO_{\id}\/; \\ MC_y(X(s_2)^\circ) = & (1+y) \cO_{s_2} - (1+2y) \cO_{\id} \/; \\ MC_y(X(s_1 s_2)^\circ) = & (1+y)^2 \cO_{s_1 s_2} -  (1+y) (1+2y) \cO_{s_1} - (1+y)(1+3y) \cO_{s_2} + (5y^2+ 5y+1) \cO_{\id} \/;  \\ MC_y(X(s_2 s_1)^\circ) = & (1+y)^2 \cO_{s_2 s_1} -  (1+y) (1+2y) \cO_{s_2} - (1+y)(1+3y) \cO_{s_1} + (5y^2 + 5y+1) \cO_{\id} \/; \\ MC_y(X(w_0)^\circ) = & (1+y)^3 \cO_{w_0} -  (1+y)^2 (1+2y) (\cO_{s_1s_2} + \cO_{s_2 s_1}) + \\ & (1+y)(5 y^2 + 4 y +1)( \cO_{s_1 } + \cO_{s_2})    - (8y^3+11y^2+ 5y+1) \cO_{\id} \/. \end{split} \] 
%The normalized non-equivariant dual motivic 
The classes $\widetilde{MC}_y(Y(w)^\circ)$ for the Schubert cells in $\Fl(3)$ are: 
\[ \begin{split} \widetilde{MC}_y(Y(w_0)) = & \cO^{w_0}; \\ \widetilde{MC}_y(Y(s_1 s_2)^\circ) = & (1+y) \cO_{s_1s_2} + y \cO^{w_0}\/; \\ \widetilde{MC}_y(Y(s_2 s_1)^\circ) = & (1+y) \cO_{s_2 s_1} + y \cO^{w_0} \/; \\ \widetilde{MC}_y(Y(s_1)^\circ) = & (1+y)^2 \cO^{s_1} +  y(1+y) \cO^{s_1 s_2} + 2y (1+y)\cO_{s_2 s_1} + y^2 \cO_{w_0} \/;  \\ \widetilde{MC}_y(Y(s_2)^\circ) = & (1+y)^2 \cO^{s_2} +  2y(1+y) \cO^{s_1 s_2} + y (1+y)\cO^{s_2 s_1} + y^2 \cO^{w_0} \/; \\ \widetilde{MC}_y(Y(\id)^\circ) = & (1+y)^3 \cO^{\id} + y (1+y)^2 (\cO^{s_1} + \cO^{s_2}) + \\ & 2 y^2 (1+y)( \cO^{s_1 s_2} + \cO^{s_2 s_1})  + y^3 \cO^{w_0} \/. \end{split} \] 

An algebra check together with fact that $\langle \cO_u, \cO^v \rangle = 1$ if and only if $u \ge v$,  shows that \[ \langle MC_y(X(u)^\circ), \widetilde{MC}_y(Y(v)^\circ) \rangle = (1+y)^3 \delta_{u,v} \] as expected from \Cref{thm:MCtilde}. 

These examples suggest that the motivic Chern classes and their duals satisfy a certain positivity property. This is discussed in section \ref{ss:pos} below. At this time we note that the positivity of the dual classes does {\em not} extend beyond small cases. For instance the coefficient of $\cO^{s_3 s_1 s_2}$ in the expansion of $\widetilde{MC}_y(Y(\id)^\circ) \in K(\Fl(4))$ equals $y^2 (4y-1)(1+y)^3$.
\end{example}
The `top' Schubert coefficient is calculated in the following result. For the convenience of the 
reader, we provide a proof.
\begin{lemma}\label{lemma:cww} The coefficient $c_{w,w}(y;e^t)$ is given by 
\[ c_{w,w}(y,e^t) = \prod_{\alpha > 0, w\alpha< 0} (1+ e^{w \alpha} y) = \lambda_y(T^*_{e_w} X(w)) \/. \]
\end{lemma}
\begin{proof} The localization $MC_y(X(w)^\circ){|_w}$ equals $c_{w,w}(y,e^t) (\cO_w^{T})|_w$.
By Lemma~\ref{lem:actiononfixedpoint}(c) and \Cref{thm:MC+dual}, we get 
\[ MC_y(X(w)^\circ){|_w}=\prod_{\alpha>0,w\alpha<0}\frac{1+ye^{w\alpha}}{1-e^{w\alpha}}\iota_w|_w.\] 
However, $(\cO_w^{T})|_w=\frac{\iota_w|_w}{\lambda_{-1}(T^*_w X(w))}=\frac{\iota_w|_w}{\prod_{\alpha>0,w\alpha<0}(1-e^{w\alpha})}$. 
The claim follows from this.\end{proof}

We end this section with an analogue of \Cref{thm:MC+dual} for the Segre motivic classes 
\[ SMC_y(X(w)^\circ) :=  \frac{MC_y(X(w)^\circ)}{\lambda_y(T^*X) } \/. \]
These classes live in a localization of $\K_T(X)_S$ where the element $
\prod_{\alpha > 0} (1+ye^\alpha)(1+ye^{-\alpha}) \in \K_T(pt)$
is invertible; see \cite[Remark 8.9]{AMSS:motivic}. 
We recall \cite[Theorem~4.2]{mihalcea2019whittaker}, which will be used below when 
discussing the Hirzebruch version of the Segre classes.
\begin{theorem}\label{thm.msegre}
For any $w\in W$ one has in $ K_T(X)[y]_S$:
\begin{equation}
\frac{MC_y(X(w)^\circ)}{\lambda_y(T^*X)}= \frac{\mathcal{T}_{w^{-1}}^\vee(\cO_{\id}^T)}{\prod_{\alpha > 0} (1+ y e^{\alpha})}
= \mathcal{T}_{w^{-1}}^\vee\left( \frac{\cO_{\id}^T}{\prod_{\alpha > 0} (1+ y e^{\alpha})}\right)
\end{equation}
and
\begin{equation}\label{eq:SegreMC}
\frac{MC_y(Y(w)^\circ)}{\lambda_y(T^*X)}= \frac{\mathcal{T}_{(w_0w)^{-1}}^\vee(\cO^{w_0, T})}{\prod_{\alpha > 0} (1+ y e^{-\alpha})}
= \mathcal{T}_{(w_0w)^{-1}}^\vee\left( \frac{\cO^{w_0, T}}{\prod_{\alpha > 0} (1+ y e^{-\alpha})}\right) \:.
\end{equation}
\end{theorem}

Note that $\prod_{\alpha > 0} (1+ y e^{\alpha})=\lambda_y(T^*_{e_{\id}}X)$ and $\prod_{\alpha > 0} (1+ y e^{-\alpha})=\lambda_y(T^*_{e_{w_0}}X)$.

\subsection{Integrals of motivic Chern classes and point counting}\label{ssec:pointc}
By functoriality of motivic Chern classes the integral of motivic Chern class of a Schubert cell equals 
\[  \int_{G/B} MC_y(X(w)^\circ)=MC_y[X(w)^\circ \to pt] = MC_y[\mathbb{A}^{\ell(w)} \to pt] = MC_y[\mathbb{A}^1 \to pt]^{\ell (w)} \/. \]
The third equality uses the fact the the map $MC_y: \K_0^{T}(var/pt) \to \K_0^{T}(pt)$ is a {\em ring} homomorphism, with the product given by exterior product of varieties; see e.g., \cite{brasselet.schurmann.yokura:hirzebruch} or \cite[Theorem 4.2]{AMSS:motivic}. One can calculate 
$MC_y[\mathbb{A}^1 \to pt]$ directly from \Cref{ex:P1}:
\[ MC_y[\mathbb{A}^1 \to pt] = \int_{\mathbb{P}^1} MC_y(\mathbb{A}^1) = -y \/. \]
Combining all these, we deduce:
\begin{prop}\label{prop:sum-st} Recall the Schubert expansion \eqref{E:schub}. Then the following hold:

(a) $\int_{G/B} MC(X(w)^\circ) = \sum c_{w,u} (y,e^t) = (-y)^{\ell(w)}$. 

(b) The $\chi_y$-genus of $G/B$ equals
\[ \chi_y(G/B) = \int_{G/B} \lambda_y(T^*(G/B)) = \sum_{w \in W} (-y)^{\ell(w)} \/. \]
\end{prop}
\begin{proof} Part (a) follows from the considerations above and because $\int_{G/B} \cO_w^T =1$, since
$H^i(X(w),\cO_{G/B}) = 0$ for $i>0$, as Schubert varieties are rational with rational singularities. 
Part (b) follows from (a), using the fact that
$ \lambda_y(T^*(G/B)) \otimes \cO_{G/B}= MC_y(G/B) = \sum_{w \in W} MC_y(X(w)^\circ)$.
\end{proof}
If one specializes $y=-q$, this proposition shows that the $\chi_y$ genus of a Schubert variety
$X(w)$ is equal to the number of points of $X(w)$ over $\mathbb{F}_q$, the field with $q$ elements. 
This type of arguments extend more generally to any $G/P$, or to $T$ varieties with finitely many fixed points; 
see e.g.,~ \cite{mihalcea2019whittaker}. 

Utilizing again the specialization $y=-q$ and taking $G/B=\Fl(n)$, one recovers in a natural way $q$-analogues 
of classical formulae. In this case, $W=S_n$ (the symmetric group) and 
\[ 
\chi_{-q}(\Fl(n)) = \sum_{w \in S_n} q^{\ell (w)}\/.
\]
It is known that this sum equals the $q$-analogue of the factorial,
\[
\sum_{w \in S_n} q^{\ell (w)} = [n]_q ! = (1+q)(1+q+q^2)\cdot \ldots \cdot (1+q + \ldots + q^{n-1}) \/. 
\]
In fact, it is fun to work out a geometric interpretation of this formula. 
~The natural projection $p_n: \Fl(n) \to \Gr(n-1, n)$ sending a flag 
$(F_1 \subseteq \ldots \subseteq F_n= \mathbb{C}^n)$ to $F_{n-1} \subseteq \mathbb{C}^n$ is a 
$G$-equivariant Zariski locally trivial fibration, with fiber isomorphic to $\Fl(n-1)$. 
By additivity and multiplicativity of motivic Chern classes over a point, it follows that 
\[ \chi_{-q}(\Fl(n)) = \chi_{-q}(\Fl(n-1)) \cdot \chi_{-q}(\Gr(n-1,n)) = \chi_{-q}(\Fl(n-1)) \cdot (1 +q + q^2 + q^3 + \ldots + q^{n-1}) \/. \] 
The last equality follows from the fact that the (dual) projective space $\Gr(n-1,n)$ is the disjoint union of Schubert cells, each of which with contribution $q^{\dim (\textrm{cell})}$ to the $\chi_{-q}$ genus. Then the equality follows by induction on $n$. A more detailed analysis of this geometric argument is done in the section \ref{ss:parabolic} below.

\subsection{Divisibility properties} Consider now the expansions of the {\em nonequivariant} motivic Chern classes, 
\begin{equation}\label{E:noneqschub} MC_y (X(w)^\circ) = \sum_{ u \le w} c_{u,w}(y) \cO_u \quad \in \K(X)[y] \/.
\end{equation} 
The coefficients $c_{u,w}(y)$ from \eqref{E:noneqschub} are polynomials in $\Z[y]$. \Cref{ex:motcells} suggests a divisibility property 
of these coefficients.  
\begin{prop}\label{prop:divisibility} The coefficient $c_{u,w}(y)$ is divisible by $(1+y)^{\ell(u)}$.
\end{prop}
\begin{proof} We prove the statement by induction on $\ell(u)$. If $u=w$, \Cref{lemma:cww} gives 
$c_{w,w}(y)=(1+y)^{\ell(w)}$. Now assume $u<w$ and that the statement is true for any $v\leq w$ with $\ell(v)>\ell(u)$.
Suppose that $(1+y)^\ell\mid c_{u,w}(y)$ and $(1+y)^{\ell+1}\nmid c_{u,w}(y)$ for some $\ell<\ell(u)$.
We use the hypotheses of \Cref{prop:initial} where we only keep the action of $\C^*$. In particular $y = - e^{-\hbar}$,
therefore the initial term of $1+y$ is $\hbar$. Consider the expansion:
\begin{align}\label{equ:ch}
\ch_{\bbC^*}&(MC_{-e^{-\hbar}}(X(w)^\circ))=\sum_{z\leq w, \ell(z)<\ell(u)} c_{z,w}(-e^{-\hbar})\ch_{\bbC^*}(\calO_z)\\
&+\sum_{z\leq w, \ell(z)>\ell(u)} c_{z,w}(-e^{-\hbar})\ch_{\bbC^*}(\calO_z)+\sum_{z\leq w, \ell(z)=\ell(u)} c_{z,w}(-e^{-\hbar})\ch_{\bbC^*}(\calO_z).\nonumber
\end{align}
\Cref{thm:MC+dual} and \Cref{prop:initial} imply that the part with the highest homological degree in $\ch_{\bbC^*}(MC_{-e^{-\hbar}}(X(w)^\circ))$ lies in 
$H_0^{\bbC^*}(X)$. (Later, we will show this is the homogenized Chern-Schwartz-MacPherson class $\csmh(X(w)^\circ)$, in 
particular it is nonzero.) Since $\ell < \ell(u)$, it follows that the coefficient of $\hbar^{\ell}[X(u)] \in H_{2\ell(u) - 2\ell }^{\C^*}(X)$ in this class is
equal to $0$. We analyze this coefficient on the right hand side. The first summand has no contribution because $\ell(z)<\ell(u)$. 
By induction, every term in the second summand is divisible by $\hbar^{\ell(u)+1}$, thus again it does not contribute to the 
coefficient of $\hbar^{\ell}[X(u)]$. In the last summand, only the term with $z=u$ can contribute. Its contribution
equals the coefficient of $\hbar^\ell$ in $c_{u,w}(-e^{-\hbar})$. This coefficient is non-zero, by the hypothesis on $\ell$, and it 
gives a non-zero coefficient of $\hbar^\ell [X(u)]$, which is is impossible.\end{proof}

We end with the following corollary.

\begin{corol} Consider the non-equivariant motivic Chern class
\[ MC_y(X(w)^\circ) = \sum c_{u,w}(y) \cO_u \/. \]
Then $c_{\id,w}(-1) =1$.
\end{corol}
\begin{proof} By the divisibility property, $c_{u,w}(-1) = 0$ for $\ell(u) >0$. Then
by \Cref{prop:sum-st},
\[ 1=\int_{G/B} MC_{-1}(X(w)^\circ) = c_{\id, w}(-1) \/, \]
which finishes the proof.\end{proof}
We invite the reader to verify this corollary for the motivic Chern classes in $\Fl(3)$ 
from \Cref{ex:motcells} above.
\subsection{The parabolic case}\label{ss:parabolic} Consider the (generalized) partial
flag manifold $G/P$, and let $\pi:G/B \to G/P$ be the natural projection. This is 
a $G$-equivariant locally trivial fibration in Zariski topology, with fiber
$F:=\pi^{-1}(1.P) = P/B$. This fiber is the flag manifold $L/(B \cap L)$, where $L$ is the Levi subgroup
of $P$. The Schubert varieties in $F$ are indexed by 
the elements in $W_P$. Furthermore, the image $\pi(X(w)^\circ)$ equals $X(wW_P)^\circ$, and 
the restriction of $\pi$ to $X(w)^\circ$ is a trivial fibration, showing that 
$X(w)^\circ \simeq X(wW_P)^\circ \times (w.F \cap X(w)^\circ)$. It follows that in the 
Grothedieck group $\K_0(var/(G/P))$,
\begin{equation}\label{E:G0prod} [X(w)^\circ \to G/P] = [w.F \cap X(w)^\circ \to wP]\boxtimes [X(wW_P)^\circ \to G/P]  \/. \end{equation}
The intersection $w.F \cap X(w)^\circ$ is the Schubert cell in $w.F$ indexed by $w_2\in W_P$,
where $w= w_1 w_2$ is the parabolic factorization of $w$ with respect to $P$; cf.~\cite[Theorem~2.8]{BCMP:QKpos}. 
This argument allows us calculate the push-forwards of motivic Chern classes. 
\begin{prop}\label{prop:pf} The following hold:

(a) $\pi_* MC_y(X(w)^\circ) = (-y)^{\ell(w) - \ell(w W_P)} MC_y(X(wW_P)^\circ)$ in $\K_T(G/P)$.

(b) More generally, let $P \subseteq Q$ be two standard parabolic subgroups, and $\pi':G/P \to G/Q$ the natural projection. 
Then
\[ \pi'_* MC_y(X(wW_P)^\circ)= (-y)^{\ell(wW_P) - \ell(wW_Q)} MC_y(X(wW_Q)^\circ) \/. \]
\end{prop}

\begin{proof} Equation \eqref{E:G0prod} and functoriality of motivic Chern classes imply that
\[\pi _* MC_y(X(w)^\circ) = MC_y[w.F \cap X(w)^\circ \to wP] \cdot MC_y(X(wW_P)^\circ)  \/, \]
Then the claim follows from \Cref{prop:sum-st}(a). Part (b) may be obtained by applying (a)
to the composition of projections $G/B \to G/Q \to G/P$.
\end{proof}

We end this section by pointing out that the argument from \Cref{prop:sum-st} extends {\em verbatim} to the parabolic case. 
One obtains:
\begin{prop}\label{prop:sum-st-parab} Let $w \in W^P$ and consider the Schubert expansion 
$MC_y(X(wW_P)^\circ) = \sum_{u \in W^P, u \le w} c_{u,w}(y,e^t) MC_y(X(uW_P)^\circ)$. Then the following hold:

(a) $\int_{G/P} MC(X(wW_P)^\circ) = \sum c_{w,u} (y,e^t) = (-y)^{\ell(w)}$. 

(b) The $\chi_y$-genus of $G/P$ equals $\chi_y(G/P) = \sum_{w \in W^P} (-y)^{\ell(w)}$.
\end{prop}
To illustrate, let $G/P= \Gr(k,n)$ be the Grassmann manifold of subspaces of dimension $k$ in $\C^n$. The set $W^P$ corresponds 
to partitions $\lambda=(\lambda_1, \ldots, \lambda_k)$ included in the $k \times (n-k)$ rectangle such that 
$\dim X(\lambda) = |\lambda| = \lambda_1 + \ldots + \lambda_k$. For $q=-y$, the $\chi_{-q}$ genus is
\[ \chi_{-q}(\Gr(k,n))= \sum_{\lambda} q^{|\lambda|} = {n \choose k}_q  \/, \]
the $q$-analogue of the binomial coefficient. 

\section{The parameter $y$ in motivic Chern classes} 
In this section we discuss some key combinatorial properties of the Schubert expansion of the motivic Chern classes, including specializations of the parameter $y$, and their geometric interpretation.
\begin{theorem}\label{prop:specialize} Let $X=G/P$ and $w\in W^P$. Then the following hold:

(a) The specialization $y=-1$ gives $MC_{-1}(X(wW_P)^\circ)= \iota_{wW_P}$ (the equivariant class of the unique torus fixed point in $X(wW_P)^\circ$).

(b) The specialization $y=0$ gives $MC_0(X(wW_P)^\circ)=\calI_{wW_P}^T$ (the class of the ideal sheaf $\cO_{X(wW_P)}(- \partial X(wW_P))$).

(c) The degree of $MC_y(X(wW_P)^\circ)$ with respect to $y$ is
equal to $\ell (w)$,
and the coefficient of~$y^{\ell (w)}$ in $MC_y(X(wW_P)^\circ)$ is the class of $\omega_{X(wW_P)}$, the dualizing sheaf on $X(wW_P)$. 
\end{theorem}

\begin{proof} We first prove all the statements for $P=B$. Parts (a) and (b) follow from \Cref{thm:MC+dual} and the specializations
at $y=-1$ and $y=0$ of the operator $\mathcal{T}_i$, from \Cref{lemma:yspec}.
The fact that the $y$-degree in (c) is less than $\ell(w)$ follows again from  \Cref{thm:MC+dual}.
Let $\theta: Z \to X(w)$ denote the Bott-Samelson 
resolution of the Schubert variety $X(w)$ utilized, e.g., in \cite{aluffi.mihalcea:eqcsm}. 
By functoriality,
\[ MC_y[X(w)^\circ \to G/B] = \theta_* MC_y [X(w)^\circ \to Z].\] 
The restriction of $\theta$ to $\theta^{-1}(X(w)^\circ)$ is an isomorphism, and it is known that 
the boundary of $\theta^{-1}(X(w)^\circ)$ in $Z$ is a simple normal crossing divisor.
Since $Z$ is smooth, by inclusion-exclusion it follows that the term with
highest power of $y$ in $MC_y(Z\setminus \theta^{-1}(X(w)^\circ))$ is the same
as the one in $MC_y(Z) = \lambda_y(T^*(Z))$, namely $y^{\ell(w)} \wedge^{\dim Z} T^*(Z)= 
y^{\ell(w)} \omega_Z$, a multiple of the canonical bundle of $Z$. 
This finishes the proof of (c), since $\theta_* (\omega_Z) = \omega_{X(w)}$ as 
$X(w)$ has rational singularities; see e.g., \cite{brion.kumar:frobenius}.

We now turn to the general $G/P$ situation. Since $w \in W^P$, the projection
$\pi:G/B \to G/P$ restricts to a birational map $\pi:X(w) \to X(w W_P)$ which is an isomorphism
over the Schubert cell $X(wW_P)^\circ$. By \Cref{prop:pf}, 
$\pi_* MC_y(X(w)^\circ) = MC_y(X(wW_P)^\circ)$. Then each claim follows from the corresponding
statement for $G/B$, taking into consideration that $\pi_*(\iota_w) = \iota_{wW_P}$,
$\pi_*\calI_w^T = \calI_{wW_P}^T$ (cf.~\cite[Proof of Lemma 4]{brion:Kpos}), and 
$\pi_* \omega_{X(w)} = \omega_{X(wW_P)}$ (since  $X(wW_P)$ has rational singularities;
see \cite{brion.kumar:frobenius}). 
\end{proof}
\begin{remark} In the non-equivariant case, the result in (b) can also be proved using that 
for the Schubert {\em variety} $MC_0(X(w)) = \cO_w^{T}$, since $X(w)$ has rational singularities, 
hence DuBois \cite[Example 3.2]{brasselet.schurmann.yokura:hirzebruch}. The equivariant generalization uses either an equivariant version of the DuBois complex, which is not available at this time in the literature, or just a corresponding virtual equality in $ K_0(\mathfrak{coh}^{T}(\cO_X))$.\end{remark}

The duality in \Cref{thm:MCtilde} allows us to calculate the specializations of $\widetilde{MC}_y(Y(w)^\circ)$.
\begin{corol}\label{cor:sp-MC} Let $w\in W$. Then the following hold:

(a) The specialization $y=-1$ gives 
\[\widetilde{MC}_{-1}(Y(w)^\circ)={\frac{\lambda_{-1} (T^*_{w_0}(G/B))}{\lambda_{ -1}(T^*_{w}(G/B))}\iota_w =
\frac{\prod_{\alpha > 0} (1- e^{-\alpha})}{\prod_{\alpha > 0} (1- e^{w \alpha})} \iota_w }
 \/. \]

(b) The specialization $y=0$ gives $\widetilde{MC}_0(Y(w)^\circ)=\cO^{w,T}$.
\end{corol}
\begin{proof} \Cref{thm:MCtilde} and \Cref{prop:specialize} imply that $\kappa:=\widetilde{MC}_{-1}(Y(w)^\circ)$
is a class in $\K_T(G/B)[y]$ with localizations
\[ \kappa|_u = \delta_{u,w} {\lambda_{-1} (T^*_{w_0} (G/B))} \/. \]
Then (a) follows from the injectivity of the localization map (see e.g., \cite{chriss2009representation})
and the fact that $\langle \iota_w,\iota_u \rangle = \delta_{u,w} {\lambda_{-1} (T^*_{w}(G/B))}$.
Part (b) is a consequence of the duality between the ideal sheaves and structure sheaves from
\Cref{E:dualst}, combined with \Cref{prop:specialize}(b).
\end{proof}
\section{Equivariant Hirzebruch  classes}\label{s:eqH} 
\subsection{The equivariant Hirzebruch transformation}\label{ssec:eqH}
In the non-equivariant case, going from motivic Chern classes to Chern-Schwartz-MacPherson (CSM) classes uses a 
certain renormalization of the Todd transformation, together with Grothendieck-Riemann-Roch; the procedure 
is explained in \cite[\S0 and \S 3]{brasselet.schurmann.yokura:hirzebruch}. 
In this paper we use  the equivariant Riemann-Roch theorem proved by Edidin and Graham \cite[Theorem~3.1]{edidin2000} to explain the equivariant counterparts
studied first by Weber~\cite{weber1, weber2} for the unnormalized Hirzebruch classes.

Since many of the results explained here are not available (in this generality) in the literature, 
we find it useful to recall the precise hypotheses we utilize. We hope this section may be used 
as a reference in the future.

By $X$ we denote a 
complex algebraic variety. For a commutative ring $R$, the completions
$\widehat{A}_*^{T}(X,R),\widehat{H}_*^{T}(X,R)$ denote the product of the equivariant 
Chow groups \cite{edidin.graham:eqchow}, respectively the equivariant Borel-Moore
homology groups (where the degree is doubled), with coefficients in $R$:
$$\widehat{A}_*^{T}(X,R):= \prod_{ i \le \dim X} A_i^{T}(X)\otimes R\quad \text{or}\quad 
 \widehat{H}_*^{T}(X,R):= \prod_{ i \le \dim X} H_{2i}^{T}(X)\otimes R \quad \/.$$
The Chow and Borel-Moore homology are related by a cycle map 
$cl:A_i^T(X) \to H_{2i}^T(X)$, and one may 
work directly in the Chow context, or in the image under this 
map; cf.~ \cite[\S2.8]{edidin.graham:eqchow}. In what follows we have chosen 
to work in the Borel-Moore context.
~The coefficients ring $R$
will be mostly $\mathbb{Z},\mathbb{Q}, \mathbb{Q}[y]$ or $\mathbb{Q}[y,(1+y)^{-1}]$.
~In case no coefficients are mentioned, we are using $R=\mathbb{Q}$
(as before).

For a $T$-variety $X$ let 
$$td_*^{T}: K_0(\mathfrak{coh}^{T}(\cO_X))\to \widehat{H}_*^{T}(X)$$ 
be the equivariant Todd class transformation to the completion
$\widehat{H}_*^{T}(X)$, constructed in \cite[\S3.2]{edidin2000}. 
Then $td_*^{T}$ is covariant for proper $T$-equivariant morphisms.
Also note that 
$$td_*^{T}(\cO_X))=Td^T(TX)\cap [X]_T$$
for $X$ smooth by \cite[Theorem~3.1(d)(i)]{edidin2000} and~\cite[Remark~6.10 and Lemma~A.1]{EG:nonabelian},
since $T$ is abelian so that the adjoint action of $T$ on its Lie algebra  is trivial
(see also~\cite[pp.~2218-2219]{MS:toric} and compare with ~\cite[Proof of Theorem~3.3]{AMSS:shadows} for the counterpart of Ohmoto's equivariant Chern class transformation).
Recall the equivariant Chern character 
$$\ch_{T}: \K_{T}(X)\to \widehat{H}^*_{T}(X) \/.$$
This is a contravariant ring homomorphism for $T$-equivariant morphisms. Then the Todd transformation 
satisfies the module property
\begin{equation} \label{Ch-td}
td_*^{T}([E]_{T}\otimes [\mathcal{F}]_{T})= \ch_{T}([E]_{T})\cap td_*^{T}([\mathcal{F}]_{T})
\end{equation}
for $[E]_{T}\in \K_{T}(X)$ and $[\mathcal{F}]_{T}\in  \K_0(\mathfrak{coh}^{T}(\cO_X))$
(\cite[Theorem~3.1]{edidin2000}). Recall that
for a $T$-equivariant vector bundle 
$E$, the cohomological Todd class $Td^T(E):=Td^{T}([E]_{T})$ is multiplicative in short exact sequences, and for an equivariant line bundle $\mathcal{L}$ 
it is defined by
$$Td^{T}(\mathcal{L}):=\frac{c_1^T(\mathcal{L})}{1-e^{-c_1^T(\mathcal{L})}} = 1+ \frac{1}{2} c_1^T(\mathcal{L})+ \ldots$$

Define the {\bf unnormalized} (respectively {\bf normalized}) 
{\bf cohomological Hirzebruch class} 
$\widetilde{Td}^{T}_y$ (resp. $Td^{T}_y$) by
$$\widetilde{Td}^{T}_y(\mathcal{L}):=\ch_{T}(\lambda_y(\mathcal{L}^{\vee}))  Td^{T}(\mathcal{L}) =   \frac{c_1^T(\mathcal{L})(1+ye^{-c_1^T(\mathcal{L})})}{1-e^{-c_1^T(\mathcal{L})}}$$%=:\tilde{p}(c_1^T(\mathcal{L}))$$
and 
$$Td^{T}_y(\mathcal{L}):=\frac{\widetilde{Td}^{T}_y((1+y)[\mathcal{L}]_{T})}{1+y}= \frac{c_1^T(\mathcal{L})(1+ye^{-c_1^T(\mathcal{L})(1+y)})}{1-e^{-c_1^T(\mathcal{L})(1+y)}}$$%=:p(c_1^T(\mathcal{L})) \/.$$
Then extend these definitions to any equivariant vector bundle by requiring that they are multiplicative
on short exact sequences.
Note the specializations:
\begin{equation}\label{E:Tdy=0} \widetilde{Td}^{T}_{y=0}(\mathcal{L})=Td^{T}_{y=0}(\mathcal{L}) =Td^{T}(\mathcal{L}) \/,\end{equation}
and
\begin{equation}\label{E:Tdy=-1} \widetilde{Td}^{T}_{y=-1}(\mathcal{L})=c_1^{T}(\mathcal{L})\quad ,
\quad Td^{T}_{y=-1}(\mathcal{L})=c^T(\mathcal{L}) = 1+c_1^{T}(\mathcal{L}) \/. \end{equation}
The power series ${Td}^{T}_y(\calL)= 1 +\ldots$ has constant coefficient $1$,
whereas $\widetilde{Td}^{T}_y(\calL)=1+y +\ldots$ has constant coefficient 
$1+y$, explaining the name `unnormalized'. 

Combining \Cref{thm:existence} with the equivariant Riemann-Roch theorem proved by Edidin and Graham 
\cite[Theorem~3.1]{edidin2000}
one obtains the following results about the unnormalized equivariant Hirzebruch class transformation.

\begin{theorem}\label{thm:Hirzebruch} Let $X$ be a quasi-projective, 
non-singular, complex algebraic variety with an action of the torus $T$. 
The unnormalized Hirzebruch transformation
$$\widetilde{Td}^{T}_{y,*}:=td^{T}_*\circ MC_y: \K_0^{T}(var/X) \to \widehat{H}_*^{T}(X)[y] 
\subseteq \widehat{H}_*^{T}(X; \mathbb{Q}[y,(1+y)^{-1}])$$
is the unique natural transformation satisfying the following properties:
\begin{enumerate} \item[(a)] It is functorial with respect to $T$-equivariant proper morphisms of non-singular, quasi-projective varieties. 

\item[(b)] It satisfies the normalization condition \[ \widetilde{Td}^{T}_{y,*}([\id_X])=\widetilde{Td}^{T}_{y}(TX)\cap [X]_{T}\/. \]

\item[(c)] It is determined by its image on classes $[f:Z \to X]=f_![\id_Z]$ where $Z$ is a non-singular,  irreducible,
quasi-projective algebraic variety and $f$ is a $T$-equivariant proper morphism.

\item[(d)] It satisfies a Verdier-Riemann-Roch (VRR) formula: for any smooth, $T$-equivariant morphism $\theta: X \to Y$ of quasi-projective and non-singular algebraic varieties, and any $[f: Z \to Y] \in \K_0^{T}(var/Y)$,
\[ \widetilde{Td}^{T}_{y}(T^*_\theta) \cap \theta^*\widetilde{Td}^{T}_{y,*}[f:Z \to Y] = \widetilde{Td}^{T}_{y,*}[\theta^* f:Z \times_Y X \to X] \/. \]

\end{enumerate} \end{theorem}

If one forgets the $T$-action, then the unnormalized equivariant Hirzebruch transformation above 
recovers the corresponding non-equivariant transformation
$\widetilde{T}_{y,*}$ from
\cite{brasselet.schurmann.yokura:hirzebruch} (either by its construction, or by the properties (a)-(c) from \Cref{thm:Hirzebruch}
and the corresponding results from \cite{brasselet.schurmann.yokura:hirzebruch}.) 

\begin{remark}\label{rem:Hirzebruch} \Cref{thm:Hirzebruch} and its proof work 
more generally for a possibly singular, quasi-projective $T$-equivariant base variety
 $X$. Moreover, $\widetilde{Td}^{T}_{y,*}$ commutes with exterior products:
\begin{equation}\label{Hirze-product}
\widetilde{Td}^{T}_{y,*}[f\times f': Z\times Z'\to X\times X']=\widetilde{Td}^{T}_{y,*}[f: Z\to X] 
\boxtimes \widetilde{Td}^{T}_{y,*}[f': Z'\to X'] \:.
\end{equation}
This follows as in the non-equivariant context~\cite[Corollary 3.1]{brasselet.schurmann.yokura:hirzebruch} 
from part (3) of \Cref{thm:Hirzebruch} 
and the multiplicativity of the corresponding  equivariant cohomological class 
for smooth and quasi-projective $T$-varieties $X,X'$:
$$ \widetilde{Td}^{T}_{y}(T^*(X\times X')) =  \widetilde{Td}^{T}_{y}(T^*X)\boxtimes \widetilde{Td}^{T}_{y}(T^*X') \in \widehat{H}^*_{T}(X\times X')[y]\:.$$
\end{remark}
One can also define a {\em normalized} equivariant Hirzebruch class transformation. With this aim, we 
introduce the following functorial {\bf (co)homological Adams operations}:
$$\psi_{1+y}^*: \widehat{H}^*_{T}(X,\mathbb{Q}[y])\to \widehat{H}^*_{T}(X,\mathbb{Q}[y]) \quad \text{and} \quad
\psi^{1+y}_*: \widehat{H}_*^{T}(X,\mathbb{Q}[y])\to \widehat{H}_*^{T}(X,\mathbb{Q}[y,(1+y)^{-1}]) $$ 
given by multiplication with $(1+y)^i$ on $H^{2i}_{T}(X,\mathbb{Q}[y])$ respectively $(1+y)^{-j}$ on $H_{2j}^{T}(X,\mathbb{Q}[y])$,
and satisfying the module and ring  properties
\begin{equation}\label{eq:psi-module}
\psi^{1+y}_*(-\cap -)=\psi_{1+y}^*(-)\cap \psi^{1+y}_*(-) \quad \text{and} \quad \psi_{1+y}^*(-\cup -)=\psi_{1+y}^*(-)\cup \psi_{1+y}^*(-) \:.
\end{equation}
(In the Chow context, the (co)homological grading will not be doubled.)
\begin{remark}\label{rem:duality}
These module and ring  properties also hold for the functorial (co)homological duality transformations
$$\psi_{-1}^*: \widehat{H}^*_{T}(X,\mathbb{Q})\to \widehat{H}^*_{T}(X,\mathbb{Q}) \quad \text{and} \quad
  \psi^{-1}_*: \widehat{H}_*^{T}(X,\mathbb{Q})\to \widehat{H}_*^{T}(X,\mathbb{Q}) $$ 
given by multiplication with $(-1)^i$ on $H^{2i}_{T}(X,\mathbb{Q})$ resp. $(-1)^{j}$ on $H_{2j}^{T}(X,\mathbb{Q})$.
If $X$ is smooth, these are consistent with the corresponding duality involutions in K-theory:
$$\ch_{T}\circ (-)^{\vee}= \psi_{-1}^*\circ \ch_{T}(-): \K_{T}(X)\to \widehat{H}^*_{T}(X,\mathbb{Q})\:,$$
since $c^{T}_1(\mathcal{L}^{\vee})=-c^{T}_1(\mathcal{L})$ for a $T$-equivariant line bundle $\mathcal{L}$.
Similarly, for the Grothendieck-Serre duality $\mathcal{D}$, 
$$td^{T}_*\circ \calD(-)= \psi^{-1}_*\circ td^{T}_*(-): \K_{T}(X)\to \widehat{H}_*^{T}(X,\mathbb{Q})\:,$$
since
 $$td^{T}_*( \omega_{X})=\ch_{T}(\omega_{X})Td^{T}(TX)\cap [X]_{T}
=Td^{T}(T^*X)\cap [X]_{T}=
(-1)^{\dim X}\psi^{-1}_*(Td^{T}(TX)\cap [X]_{T}) \:.$$
\end{remark}
By definition, the (co)homological Adams operations satisfy
 $$\psi_{1+y}^*(c_1^{T}(\mathcal{L}))=(1+y)c_1^{T}(\mathcal{L});\quad 
{\psi_{1+y}^*(\widetilde{Td}^{T}_{y}(\mathcal{L}))=(1+y)Td^{T}_{y}(\mathcal{L}) }$$
 and 
$\psi^{1+y}_*([X]_{T})=(1+y)^{-d}[X]_{T}$ for $X$ of pure dimension $d$. It follows that
\begin{equation}\label{eq:Psi-smooth}
\psi^{1+y}_*(\widetilde{Td}^{T}_{y}(TX)\cap [X]_{T}) =\psi_{1+y}^*(\widetilde{Td}^{T}_{y}(TX))\cap \psi^{1+y}_*([X]_{T})= Td^{T}_{y}(TX)\cap [X]_{T}
\end{equation}
for $X$ smooth and pure dimensional. In particular 
$$
\psi^{1+y}_*({\widetilde{Td}_{y}^{T}(TX)}\cap [X]_{T})\in \widehat{H}_*^{T}(X,\mathbb{Q}[y])\subseteq 
\widehat{H}_*^{T}(X,\mathbb{Q}[y,(1+y)^{-1}]) \/.
$$
This motivates the following:
\begin{definition} The {\bf normalized 
equivariant Hirzebruch transformation} $Td^{T}_{y,*}$ is defined by: 
\begin{equation}\label{eq:nHir}
Td^{T}_{y,*}:= \psi^{1+y}_*\circ \widetilde{Td}^{T}_{y,*}: \K_0^{T}(var/X) \to 
\widehat{H}_*^{T}(X,\mathbb{Q}[y]) \subseteq \widehat{H}_*^{T}(X,\mathbb{Q}[y,(1+y)^{-1}])\:.
\end{equation}
\end{definition}
{After replacing the unnormalized classes by the normalized ones, this transformation satisfies the properties from \Cref{thm:Hirzebruch} and \Cref{rem:Hirzebruch} above.~Furthermore, the normalized transformation
has values in $\widehat{H}_*^{T}(X,\mathbb{Q}[y])$,}
by \Cref{thm:Hirzebruch}(c) and 
\Cref{eq:Psi-smooth} above. 

{\begin{remark} \label{rem:rigid2}
The equivariant $\chi_y$-genus of a $T$-variety $Z$ may be calculated by
$$\chi_y(Z)=\widetilde{Td}_{y}^{T}([Z\to pt]) \in \widehat{H}^*_{T}(pt)[y] \quad \text{and} \quad
\chi_y(Z)=Td_{y}^{T}([Z\to pt]) \in \widehat{H}^*_{T}(pt)[y]\:.$$
By \emph{rigidity} of the $\chi_y$-genus (see~\cite[Theorem~7.2]{weber1}), these 
quantities
contain no information about the action of $T$, i.e.
both are equal to the non-equivariant $\chi_y$-genus under the embedding 
$\mathbb{Z}[y]\to \mathbb{Q}[y]\to  \widehat{H}^*_{T}(pt)[y]$.
\end{remark}
}

The definition of the Hirzebruch classes implies that  
$$Td^{T}_{y=-1}(TX)\cap [X]_{T}= c^{T}(TX)\cap [X]_{T} \quad \text{and} \quad Td^{T}_{y=0}(TX)\cap [X]_{T}= Td^{T}(TX)\cap [X]_{T} $$
for $X$ smooth and pure dimensional. As in the non-equivariant context of
~\cite{brasselet.schurmann.yokura:hirzebruch} (for the non-equivariant normalized Hirzebruch class $T_{y,*}$)
this implies the following:
\begin{corol}\label{cor:special-Ty}
The equivariant Hirzebruch transformation $Td^{T}_{y,*}$ fits into the following commutative diagram of natural transformations:
$$\begin{CD}
\mathcal{F}^{T}(X) @< e << \K_0^{T}(var/X) @> MC_{y=0} >> \K_0(\mathfrak{coh}^{T}(\cO_X))\\
@V c_*^{T}\otimes \mathbb{Q} VV @V Td^{T}_{y,*} VV @VV td^{T}_* V\\
\widehat{H}_*^{T}(X) @< y=-1 << \widehat{H}_*^{T}(X,\mathbb{Q}[y])@> y=0 >>
\widehat{H}_*^{T}(X) \:.
\end{CD}$$
\end{corol}
Here $c_*^{T}$ is the \textit{equivariant Chern class} transformation defined by Ohmoto~\cite{ohmoto:eqcsm};
see more about it in section~\ref{sec:CSM}. We note that $c_*$ has values in the {\em integral} homology, and also 
the completion is not needed, i.e., 
$c_*^{T}: \mathcal{F}^{T}(X) \to H_*^{T}(X,\mathbb{Z}) \subseteq \widehat{H}_*^{T}(X,\mathbb{Z})$.
The rational coefficients are needed due to the use of the Chern character.
Finally $e([\id_X])=\one_X$ even for $X$ singular, so that the equivariant normalized Hirzebruch class  $Td^{T}_{y,*}(X):=Td^{T}_{y,*}([\id_X])$ specializes for $y=-1$ also for a singular $X$
to the (rationalized) equivariant Chern-Schwartz-MacPherson (CSM) class $c_*^{T}(X):=c_*^{T}(\one_X)$ of $X$, and similarly for a locally closed $T$-invariant subvariety $Z\subseteq X$:
\begin{equation}\label{eq:T=c}
Td^{T}_{y=-1,*}[Z\hookrightarrow  X]= c^{T}_*(\one_Z)\otimes \mathbb{Q}\:.
\end{equation}
Note that if $H_*^{T}(X,\mathbb{Z})$ is torsion free, no information about $c^{T}_*(\one_Z)$ 
is lost by switching to rational coefficients; this is the case for the flag manifolds $X=G/P$.

{If $Z \subseteq X$ we will refer to the classes $\widetilde{Td}_{y,*}[Z \hookrightarrow X]$ and ${Td}_{y,*}[Z \hookrightarrow X]$ in $\widehat{H}_*^T(X, \mathbb{Q}[y])$ as the {\bf unnormalized}, respectively the {\bf normalized Hirzebruch class} of $Z$ in $X$.}

\subsection{Hirzebruch classes for flag manifolds} We turn now to the study 
of the Hirzebruch transformation in the case 
when $X=G/B$. There is the following commutative diagram:
\begin{equation}\label{diagr-pairing}
\begin{CD}
%\langle -, - \rangle:   
\K_T(X) \times \K_0(\mathfrak{coh}^{T}(\cO_X)) @> \otimes >>\K_0(\mathfrak{coh}^{T}(\cO_X)) @> \int_X >> \K_{T}(pt)\\
@V \ch_T \otimes td^{T}_* VV @V td^{T}_* VV @VV \ch_{T}=td^{T}_* V\\
%\langle -, - \rangle_H: 
\widehat{H}^*_{T}(X) \times \widehat{H}_*^{T}(X) @> \cap >> \widehat{H}_*^{T}(X)  @> \int_X >> \widehat{H}^*_{T}(pt) \:.
 \end{CD}
\end{equation}
%\begin{equation}\label{diagr-pairing}
%\begin{CD}
%\langle -, - \rangle:  K_0(\mathfrak{coh}^{T}(\cO_X))\times K_{T}(X) @> \otimes >>K_0(\mathfrak{coh}^{T}(\cO_X)) @> \int_X >> K_{T}(pt)\\
%@V td^{T}_*\otimes \ch_{T} VV @V td^{T}_* VV @VV \ch_{T}=td^{T}_* V\\
%\langle -, - \rangle_H: \widehat{H}_*^{T}(X,\mathbb{Q})\times \widehat{H}^*_{T}(X,\mathbb{Q})@> \cap >> \widehat{H}_*^{T}(X,\mathbb{Q})  @> \int_X >> \widehat{H}^*_{T}(pt,\mathbb{Q}) \:.
% \end{CD}
%\end{equation}
Since $X$ is smooth, from \eqref{Ch-td} we have that $td^{T}_*(-)=\ch_{T}(-)Td^{T}(TX)$. The functoriality of $td^{T}_*$ gives the equivariant Grothendieck-Hirzebruch-Riemann-Roch theorem
for an equivariant morphism of smooth $T$-varieties; cf.~\cite[Theorem~3.1]{edidin2000}. In particular, the GHRR theorem and \eqref{E:dualst} imply that
\begin{equation}\label{equ:todd-duals}
\langle td_*^{T}(\cO_{u}^{T}), \ch_{T}( \mathcal{I}^{v,T}) \rangle=\delta_{u,v}= \langle td_*^{T}( \mathcal{I}_u^T), \ch_{T}(\cO^{v,T}) \rangle \:.
\end{equation}
Utilizing that $\{ [X(w)]_T \}_{w \in W}$ is a $H^*_T(pt)$-basis of $H^*_T(X)$, it is not difficult to show that 
the {Todd} classes $td^{T}_*(X(w)):=td_*^{T}(\cO_{w}^{T})$ of the Schubert varieties $X(w)$, respectively 
the Todd classes of the corresponding ideal sheaves $td_*^{T}( \mathcal{I}_w)$, give two bases of 
$\widehat{H}_*^{T}(X)$ as a $\widehat{H}^*_{T}(pt)$-module. The key point is that 
the corresponding coefficient matrix is triangular with respect to the Bruhat ordering, with units on the diagonal.
This follows by 
the functoriality of $td^{T}_*$ for a closed inclusion $X(w)\hookrightarrow X$. 
The corresponding dual bases are given by the Chern characters $ \ch_{T}( \mathcal{I}^{v,T})$ 
respectively $\ch_{T}(\cO^{v,T})$ of the opposite Schubert varieties
for $v\in W$.
If one specializes all the equivariant parameters to $0$ one recovers the natural map  
$\widehat{H}_*^{T}(X)\to H_*(X)$ forgetting the $T$-action, and mapping $td_*^{T}$ to $td_*$. 
In this case ${td_*(\calO_w)}=[X(w)]+ l.o.t.$ and $td_*( \mathcal{I}_w)=[X(w)]+ l.o.t.$ (lower order terms).

Recall from \Cref{E:BGG-ch} that the Demazure and BGG operators are related by
\begin{equation}\label{E:BGG-ch2}
\ch_{T} (\partial_i (-)) = \parcoh_i (Td^{T}(T_{p_i})\ch_{T}(-)): \K_{T}(X) \to 
\widehat{H}^*_{T}(X)\:.
\end{equation}
Similarly, the equivariant Verdier-Riemann-Roch theorem (VRR) of 
\cite[Theorem~3.1(d)]{edidin2000} implies
\begin{equation}\label{eq:BGG-td}
td^{T}_* (\partial_i (-)) = Td^{T}(T_{p_i})\parcoh_i (td^{T}_*(-)): \K_0(\mathfrak{coh}^{T}(\cO_X)) \to \widehat{H}_*^{T}(X)\:.
\end{equation}
Then identities \eqref{E:BGG-ch2} and \eqref{eq:BGG-td}, combined with \eqref{equ:BGGonstru}, translate into
\begin{equation}\label{equ:BGGChern}
\parcoh_i ( Td^{T}(T_{p_i}) \ch_{T}(\cO_{w}^{T})) = \begin{cases} \ch_{T}(\cO_{ws_i}^{T}) & \textrm{ if } ws_i>w \/; \\
\ch_{T}(\cO_{w}^{T}) & \textrm{ otherwise. }  \end{cases} 
\end{equation}
and
\begin{equation}\label{equ:BGGTodd}
 Td^{T}(T_{p_i})\parcoh_i (td^{T}_*(\cO_{w}^{T})) = \begin{cases} td^{T}_*(\cO_{ws_i}^{T}) & \textrm{ if } ws_i>w \/; \\
 td^{T}_*(\cO_{w}^{T}) & \textrm{ otherwise }  \end{cases} 
\end{equation}
Similarly, Lemma~\ref{lemma:yspec}(a) translates into
\begin{equation}\label{equ:BGGChern-I}
(\parcoh_i Td^{T}(T_{p_i})-\id) ( \ch_{T} (\mathcal{I}_w^T)) = \begin{cases} \ch_{T}( \mathcal{I}_{ws_i}^T) & \textrm{ if } ws_i>w \/; \\
-\ch_{T}( \mathcal{I}_w^T) & \textrm{ otherwise. }  \end{cases} 
\end{equation}
and
\begin{equation}\label{equ:BGGTodd-I}
 (Td^{T}(T_{p_i})\parcoh_i -\id)(td^{T}_*( \mathcal{I}_w^T)) = \begin{cases} td^{T}_*( \mathcal{I}_{ws_i}^T) & \textrm{ if } ws_i>w \/; \\
- td^{T}_*( \mathcal{I}_w^T) & \textrm{ otherwise }  \end{cases} 
\end{equation}
{The specializations from \Cref{cor:special-Ty} and the last two equations motivate the definition of the 
{\bf unnormalized (ordinary and dual) Hirzebruch operators}
$\widetilde{\mathcal{T}}_i^H,\widetilde{\mathcal{T}}_i^{H,\vee} :\widehat{H}^*_{T}(X,\mathbb{Q})[y]  \to 
\widehat{H}^*_{T}(X,\mathbb{Q})[y]$ by 
\begin{equation}\label{eq:unnDL-H} 
\widetilde{\mathcal{T}}_i^H := 
\widetilde{Td}^{T}_{y}(T_{p_i})\parcoh_i  - \id \/; \quad 
\widetilde{\mathcal{T}}_i^{H,\vee} = \parcoh_i \circ(\widetilde{Td}^{T}_{y}(T_{p_i})\cup (-) ) - \id \/.
\end{equation}
Similarly we define
the {\bf normalized Hirzebruch operators} 
${\mathcal{T}}_i^H,{\mathcal{T}}_i^{H,\vee} :\widehat{H}^*_{T}(X,\mathbb{Q})[y]  \to 
\widehat{H}^*_{T}(X,\mathbb{Q})[y]$ by   
\begin{equation}\label{eq:normDL-H} 
{\mathcal{T}}_i^H := 
{Td}^{T}_{y}(T_{p_i})\parcoh_i  - \id \/; \quad 
{\mathcal{T}}_i^{H,\vee} = \parcoh_i \circ({Td}^{T}_{y}(T_{p_i})\cup (-) ) - \id \/.
\end{equation}
All these operators are $\widehat{H}^*_T(pt)$-linear. 
\begin{lemma}\label{lemma:Hcomm} The Hirzebruch operators satisfy the following commutativity relations:

(a) As operators $\K_{T}(X)[y] \to \widehat{H}^*_{T}(X,\mathbb{Q})[y]$,
\begin{equation}\label{E:comm-tildeTH} td^{T}_*  \mathcal{T}_i = \widetilde{\mathcal{T}}_i^H td^{T}_* \quad \textrm{ and } \quad  \ch_{T} \mathcal{T}_i^\vee = \widetilde{\mathcal{T}}_i^{H,\vee} \ch_{T} \/.\end{equation}

(b) As operators $\K_{T}(X)[y] \to \widehat{H}^*_{T}(X,\mathbb{Q}[y,(1+y)^{-1}])$,
\begin{equation}\label{E:comm-TH} \psi_*^{1+y}td^{T}_* \mathcal{T}_i = \mathcal{T}_i^H \psi_*^{1+y}td^{T}_* \/ \quad \textrm{ and } \quad 
\psi^*_{1+y}\ch_{T} \mathcal{T}_i^\vee = \mathcal{T}_i^{H,\vee} \psi^*_{1+y} \ch_{T} \/.\end{equation}
\end{lemma}
\begin{proof} The first commutativity relation in part (a) follows from the following sequence of equalities:
\[ \begin{split} \widetilde{\mathcal{T}}_i^H td^{T}_* & = \widetilde{Td}_y^T (T_{p_i})\parcoh_i td_* - td_*  \\
& = \ch_T(\lambda_y(T^*_{p_i})) Td(T_{p_i}) \parcoh_i td_* - td_* \\
& = \ch_T(\lambda_y(T^*_{p_i}))td_* \partial_i -td_* \\
& = td_* (\lambda_y(T^*_{p_i}) \partial_i - \id) \\
& = td_* \mathcal{T}_i \/. \end{split}\]
Here the third equality follows from \eqref{eq:BGG-td}, and the fourth from the module property \eqref{Ch-td}
of the Todd transformation. The second commutativity relation in (a) follows from 
\[ \begin{split} \widetilde{\mathcal{T}}_i^{H,\vee} \ch_{T} & = \parcoh_i (\widetilde{Td}_y(T_{p_i}) \ch_T) - \ch_T \\
& =  \parcoh_i ({Td(T_{p_i})} \ch_T(\lambda_y(T^*_{p_i}) \cup (-) )) - \ch_T \\
& = \ch_T \partial_i (\lambda_y(T^*_{p_i}) \cup (-)) - \ch_T \\
& = \ch_T (\partial_i \lambda_y(T^*_{p_i}) - \id) \\
& = \ch_T \mathcal{T}_i^\vee \/.
\end{split}\]
In this case, the second equality follows since $\ch_T$ is a ring homomorphism, the third from \eqref{E:BGG-ch2}, and the rest by the definitions. 
Part (b) follows from (a) and the identities in \Cref{lemma:commHops} below.
\end{proof} 

\begin{lemma}\label{lemma:commHops} As operators $\widehat{H}^*_{T}(X)[y] \to \widehat{H}^*_{T}(X,\mathbb{Q}[y,(1+y)^{-1}])$,
\begin{equation}\label{E:comm-TH1}
 \psi_*^{1+y}\widetilde{\mathcal{T}}_i^H=  \mathcal{T}_i^H \psi_*^{1+y} \/
 \quad \textrm{ and } \quad 
\psi^*_{1+y}\widetilde{\mathcal{T}}_i^{H,\vee} = \mathcal{T}_i^{H,\vee} \psi^*_{1+y}  \/.
\end{equation}
\end{lemma}
\begin{proof}
The first commutativity relation follows similarly to the proof of 
\Cref{lemma:Hcomm}(a)
above from the following sequence of equalities:
\[ \begin{split} \psi_*^{1+y}\widetilde{\mathcal{T}}_i^H & = \psi_*^{1+y}( \widetilde{Td}_y^T (T_{p_i})\parcoh_i - \id)  \\
& = \psi^*_{1+y}( \widetilde{Td}_y^T(T^*_{p_i})) \psi_*^{1+y} \parcoh_i  - \psi_*^{1+y} \\
& = \frac{ \psi^*_{1+y}( \widetilde{Td}_y^T(T^*_{p_i}))}{1+y} \parcoh_i\psi_*^{1+y} -\psi_*^{1+y} \\
& = Td_y^T(T^*_{p_i}))\parcoh_i \psi_*^{1+y} -\psi_*^{1+y} \\
& = \mathcal{T}_i^H \psi_*^{1+y} \/. \end{split}\]
Here the second equality follows from  the module property \eqref{eq:psi-module}
of the Adams transformation, and the third uses the property
$$\psi_*^{1+y}\circ \parcoh_i = (1+y)^{-1}\parcoh_i\circ \psi_*^{1+y}\:,$$
since  $\parcoh_i$ shifts the  complex homological  degree by one.
 The second commutativity relation follows similarly from: 
\[ \begin{split} \psi^*_{1+y}\widetilde{\mathcal{T}}_i^{H,\vee} & =\psi^*_{1+y}( \parcoh_i \widetilde{Td}_y(T_{p_i})  - \id) \\
& =  \parcoh_i\left(\frac{\psi^*_{1+y}}{1+y} ( \widetilde{Td}_y(T_{p_i}) \cup (-) )\right) - \psi^*_{1+y}\\
& = \parcoh_i\left(\frac{\psi^*_{1+y}  \widetilde{Td}_y(T_{p_i})}{1+y}\psi^*_{1+y}(-) \right ) - \psi^*_{1+y}\\
& = \parcoh_i( Td_y^T(T^*_{p_i}) \psi^*_{1+y}(-)) -\psi^*_{1+y} \\
& =  \mathcal{T}_i^{H,\vee} \psi^*_{1+y}  \/.
\end{split}\]
In this case, the second equality uses the property
$$\psi_{1+y}^*\circ \parcoh_i = \parcoh_i\circ ((1+y)^{-1}\psi^*_{1+y})\:,$$
since  $\parcoh_i$ shifts the  complex cohomological  degree by minus one.
The third equality follows since $\psi^*_{1+y}$ is a ring homomorphism.
\end{proof}

\begin{lemma}\label{lemma:adjoint} (a) The ordinary (normalized/unnormalized) Hirzebruch 
operators are adjoint to the dual operators, i.e.~for any $a,b \in \widehat{H}^*_{T}(X,\mathbb{Q})[y]$,
\[ \langle {\mathcal{T}}_i^H (a), b \rangle = \langle a, {\mathcal{T}}_i^{H,\vee}(b) \rangle \/; \quad 
\langle \widetilde{\mathcal{T}}_i^H (a), b \rangle = \langle a, \widetilde{\mathcal{T}}_i^{H,\vee}(b) \rangle \/, \]
where the pairing is extended by $\mathbb{Q}[y]$-linearity.\begin{footnote} {Later in the context of Segre classes, we will tacitly
extend this pairing by linearity from coefficients in $\mathbb{Q}[y]$ to $\mathbb{Q}[y,(1+y)^{-1}]$.}\end{footnote}

(b) Each family of the Hirzebruch operators (ordinary / dual / (un)normalized) satisfies the same relations as the
K-theoretic Demazure-Lusztig operators; cf.~\Cref{prop:hecke-relations}.
\end{lemma}
\begin{proof} 
Part (a) can be proved by using the self-adjointness of $\parcoh_i$ and the operators of multiplication by
${Td}^{T}_{y}(T_{p_i})$, as follows. 
By definition,
\begin{align*}
\langle {\mathcal{T}}_i^H (a), b \rangle =&\langle ({Td}^{T}_{y}(T_{p_i})\parcoh_i  - \id)  (a), b \rangle\\
=&\langle a,(\parcoh_i{Td}^{T}_{y}(T_{p_i})-\id) b \rangle\\
=& \langle a, {\mathcal{T}}_i^{H,\vee}(b) \rangle.
\end{align*}
A similar proof works for the unnormalized operators.
 
 We turn to the relations in (b). First, we deduce from (a) that it suffices to prove the statements for the
 ordinary operators. Then we utilize \Cref{lemma:Hcomm}(a) again to show that for $a' \in \K_T(G/B)$, 
 \[ \widetilde{\mathcal{T}}_{i_1}^H \widetilde{\mathcal{T}}_{i_2}^H \ldots \widetilde{\mathcal{T}}_{i_k}^H (td_*(a')) = td_* {\mathcal{T}}_{i_1} {\mathcal{T}}_{i_2} \ldots {\mathcal{T}}_{i_k} (a') \/. \]
Since $td^T_*$ is surjective (after appropriate extending the coefficients via $\ch_T: \K^T(pt)\to \widehat{H}^*_{T}(pt)$), this implies the claim for the unnormalized Hirzebruch operators. The same proof works for the normalized operators, using classes of the form $\psi_*^{1+y}td_*(a')$,  \Cref{lemma:Hcomm}(b) and \Cref{lemma:commHops}.
We also note that in order to prove the statement for the dual operators, instead of adjointness, one may alternatively
work with classes of type $\ch_T(a')$  and $\psi^*_{1+y}\ch_T(a')$. 
\end{proof}

\begin{remark}
The results of \Cref{lemma:adjoint}(b) also follow from the following argument. 
Regard $\widehat{H}^*_{T}(pt)[y]$ as a $\K_T(pt)[y]$-algebra via the (injective) equivariant 
Chern character map.
%(Note that the character map is injective, but not surjective in this case.)  
Then the transformations $td_*^T$ and $\ch_T$ induce injective homomorphisms of $\K_T(pt)[y]$-algebras
\begin{equation}
\overline{td_*^T} , \overline{\ch_T} : End_{\K^T(pt)[y]}(\K^T(X)[y]) \to End_{\widehat{H}^*_{T}(pt)[y] }(\widehat{H}^*_{T}(X)[y] ) 
\/.
\end{equation}
The injectivity part follows by using suitable bases, such as images of Schubert classes
$\cO_w$ under the Todd or Chern character maps. Then \Cref{lemma:Hcomm} 
may be interpreted as
giving the identities
\[ \overline{td_*^T}(\mathcal{T}_i) = \widetilde{\mathcal{T}}_i^H \quad \text{and} \quad
\overline{\ch_T}(\mathcal{T}_i^{\vee}) = \widetilde{\mathcal{T}}_i^{H,\vee} \:. \]
Since $\overline{td_*^T} , \overline{\ch_T}$ are algebra homomorphisms, they will preserve
relations satisfied by $\mathcal{T}_i,\mathcal{T}_i^{\vee}$, proving the claim for the unnormalized operators.

One may argue similarly in the case of normalized operators. 
{Start by changing the coefficient ring using the algebra isomorphism
$$\begin{CD}
 \widehat{H}^*_{T}(X)[y]\otimes_{\widehat{H}^*_{T}(pt)[y]}
\widehat{H}^*_{T}(pt,\mathbb{Q}[y,(1+y)^{-1}])
@> \sim >> \widehat{H}^*_{T}(X,\mathbb{Q}[y,(1+y)^{-1}])
\end{CD} \/.$$
Then the Adams operations 
$$\psi^*_{1+y}: \widehat{H}^*_{T}(pt)[y] \to  
\widehat{H}^*_{T}(pt)[y] \textrm { and } \psi_*^{1+y}: \widehat{H}_*^{T}(pt)[y] \to  
\widehat{H}_*^{T}(pt,\mathbb{Q}[y,(1+y)^{-1}])$$ 
induce algebra isomorphisms
%$\widehat{H}^*_{T}(pt,\mathbb{Q}[y,(1+y)^{-1}])$-algebra isomorphisms 
%$$\begin{CD}
% \widehat{H}^*_{T}(X)[y]\otimes_{\widehat{H}^*_{T}(pt)[y]}
%\widehat{H}^*_{T}(pt,\mathbb{Q}[y,(1+y)^{-1}])
%@> \sim >> \widehat{H}^*_{T}(X,\mathbb{Q}[y,(1+y)^{-1}])
%\end{CD} \/.$$
%(Again, the isomorphism part may be checked using suitable Schubert-type bases.)
%Therefore the Adams operations also induce algebra isomorphisms
\[ \overline{\psi^*_{1+y}} : End_{\widehat{H}^*_{T}(pt,\mathbb{Q}[y,(1+y)^{-1}]}(\widehat{H}^*_{T}(X,\mathbb{Q}[y,(1+y)^{-1}] ) \to
End_{\widehat{H}^*_{T}(pt,\mathbb{Q}[y,(1+y)^{-1}]}(\widehat{H}^*_{T}(X,\mathbb{Q}[y,(1+y)^{-1}] ) \]
}
%\[ \overline{\psi^*_{1+y}} : End_{\widehat{H}^*_{T}(pt)[y] }(\widehat{H}^*_{T}(X)[y] ) \to
%End_{\widehat{H}^*_{T}(pt)[y] }(\widehat{H}^*_{T}(X)[y] ) \]
and 
%$\widehat{H}^*_{T}(pt,\mathbb{Q}[y,(1+y)^{-1}])$-algebra isomomorphisms
\[ \overline{\psi_*^{1+y}}: End_{\widehat{H}^*_{T}(pt,\mathbb{Q}[y,(1+y)^{-1}]) }(\widehat{H}^*_{T}(X,\mathbb{Q}[y,(1+y)^{-1}]) ) \to
End_{\widehat{H}^*_{T}(pt,\mathbb{Q}[y,(1+y)^{-1}]) }(\widehat{H}^*_{T}(X,\mathbb{Q}[y,(1+y)^{-1}]) ) \]
which by \Cref{lemma:commHops} satisfy
\[ \overline{\psi_*^{1+y}}(\widetilde{\mathcal{T}}_i^H) = \mathcal{T}_i^H \quad \text{and} \quad
\overline{\psi^*_{1+y}}( \widetilde{\mathcal{T}}_i^{H,\vee} ) = \mathcal{T}_i^{H,\vee} \/.\]
Therefore the relations satisfied by the unnormalized operators are the same as those
for the normalized ones.
\end{remark}

\begin{theorem}\label{prop:HDL-rec} 
Let $w \in W$ and $s_i$ a simple reflection such that $ws_i >w$. 

(a) The unnormalized Hirzebruch operators satisfy:
\[ \widetilde{\mathcal{T}}_i^H (\widetilde{Td}_{y,*}^T (X(w)^\circ))= \widetilde{Td}_{y,*}^T (X(ws_i)^\circ)\/;\]

(b) The normalized Hirzebruch operators satisfy:
\[ {\mathcal{T}}_i^H ({Td}_{y,*}^T (X(w)^\circ))= {Td}_{y,*}^T (X(ws_i)^\circ)\/.\]

(c) In particular, for $w \in W$, the Hirzebruch classes of Schubert cells are given by:
\[ \widetilde{Td}_{y,*}^T (X(w)^\circ) = \widetilde{\mathcal{T}}_{w^{-1}}^H [e_{\id}]_T \textrm { and } 
{Td}_{y,*}^T (X(w)^\circ) = {\mathcal{T}}_{w^{-1}}^H [e_{\id}]_T \/. \]
 
\end{theorem}
\begin{proof} Part (a) follows from \Cref{thm:MC+dual} together with the commutation relations \eqref{E:comm-tildeTH}.
The same argument applies to (b), using the definition of the normalized Hirzebruch transformation from \eqref{eq:nHir},
\Cref{thm:MC+dual} again, and the commutation relations \eqref{E:comm-TH}. Part (c) is a consequence of (a)
and (b), taking into account that, by functoriality, $\widetilde{Td}_{y,*}^T (e_{\id})={Td_{y,*}^T} (e_{\id})=[e_{\id}]_T$. 
\end{proof} 

In analogy with the definition of the operators $\operL_i$ giving 
the dual classes of the motivic Chern classes, 
define the operators $ \operL^H_i$ and $\widetilde{ \operL}^H_i$ by
$$ \operL^H_i  := - y (\mathcal{T}_i^{H,\vee})^{-1}= \mathcal{T}_i^{H,\vee} + (1+y)\id
\quad \text{and} \quad 
\widetilde{\operL}^H_i  := - y (\widetilde{\mathcal{T}}_i^{H,\vee})^{-1}= \widetilde{\mathcal{T}}_i^{H,\vee} + (1+y)\id \:.$$

\begin{theorem}\label{thm:Hdual} For any $u, v \in W$, 
 \begin{equation}\label{equ:Hirze-dual1} 
 \langle \widetilde{Td}^{T}_{y*}(X(u)^\circ), \widetilde{\operL}^H_{v^{-1}w_0} ([e_{w_0}]_{T}) \rangle_H =
\delta_{u,v} \widetilde{Td}_y^{T}(T_{{w_0}}X) \end{equation}
and
 \begin{equation}\label{equ:Hirze-dual2} 
 \langle Td^{T}_{y*}(X(u)^\circ), \operL^H_{v^{-1}w_0} ([e_{w_0}]_{T}) \rangle_H =
\delta_{u,v} Td_y^{T}(T_{{w_0}}X) \:. \end{equation}
\end{theorem}
\begin{proof} 
%Recall that $\ch_T = td_*^T$ when restricted to $\K_T(pt)$. 
We
apply the equivariant Todd class transformation to both sides in the expression from \Cref{thm:MCtilde} to obtain
\[ \begin{split} \delta_{u,v} td_*^T \left(\prod_{\alpha > 0} (1+ y e^{-\alpha}) \right) 
=& td^T_*\left(\langle MC_y(X(u)^\circ), \widetilde{MC}_y(Y(v)^\circ) \rangle_K \right) \\
= & td^T_*\left( \int_X^K MC_y(X(u)^\circ)\cdot \widetilde{MC}_y(Y(v)^\circ) \right)\\
=&\int_X^{H^*} td^T_*(MC_y(X(u)^\circ))\cdot \ch_T(\widetilde{MC}_y(Y(v)^\circ))\\
=&\int_X^{H^*} \widetilde{Td}^{T}_{y,*}(X(u)^\circ) \cdot \ch_T(\operL_{v^{-1}w_0}(\calO^{w_0,T}))\\
=&\int_X^{H^*} \widetilde{Td}^{T}_{y,*}(X(u)^\circ)\cdot \widetilde{\operL}^H_{v^{-1}w_0}(\ch_T(\calO^{w_0,T}))\\
=&\langle \widetilde{Td}^{T}_{y,*}(X(u)^\circ), \widetilde{\operL}^H_{v^{-1}w_0} (\ch_T(\calO^{w_0,T})) \rangle_{H} \/.
\end{split} \]
%\begin{align*}
%\delta_{u,v} td_*^T \left(\prod_{\alpha > 0} (1+ y e^{-\alpha})\right \cap [\mathcal{O}_{pt}]_T =
%& td^T_*\left(
%\langle MC_y(X(u)^\circ), \widetilde{MC}_y(Y(v)^\circ) \rangle_K \right) \\
%=&  td^T_*\left( \int_X^K MC_y(X(u)^\circ)\cdot \widetilde{MC}_y(Y(v)^\circ) \right)\\
%=&\int_X^{H^*} td^T_*(MC_y(X(u)^\circ))\cdot \ch_T(\widetilde{MC}_y(Y(v)^\circ))\\
%=&\int_X^{H^*} \widetilde{Td}^{T}_{y,*}(X(u)^\circ) \cdot \ch_T(\calL_{v^{-1}w_0}(\calO^{w_0,T}))\\
%=&\int_X^{H^*} \widetilde{Td}^{T}_{y,*}(X(u)^\circ)\cdot \widetilde{\operL}^H_{v^{-1}w_0}(\ch_T(\calO^{w_0,T}))\\
%=&\langle \widetilde{Td}^{T}_{y,*}(X(u)^\circ), \widetilde{\operL}^H_{v^{-1}w_0} (\ch_T(\calO^{w_0,T})) \rangle_{H}.
%\end{align*}
Here we use $K$ and $H^*$ to indicate where the operation is taken, and the fifth equality 
follows from
\Cref{lemma:Hcomm}. Given this, the claim in \eqref{equ:Hirze-dual1} follows because
\[\ch_T(\calO^{w_0,T})=\frac{[e_{w_0}]_T}{Td^T(TX)}=\frac{[e_{w_0}]_T}{Td^T(T_{{w_0}}X])},\]
\[\widetilde{Td}_y^{T}(T_{{w_0}}X)=\ch_T\left(\prod_{\alpha > 0} (1+ y e^{-\alpha})\right) Td^T(T_{{w_0}}X),\]
and the fact that $\tilde{\operL}^H_i$ is $\widehat{H}^*_{T}(pt)[y]$-linear.

The equality from \eqref{equ:Hirze-dual2} follows from \eqref{equ:Hirze-dual1} 
by application of the Adams operation $\psi_*^{1+y}$:
\begin{align*}
\delta_{u,v}\psi^*_{1+y}\left( \widetilde{Td}_y^{T}(T_{{w_0}}X)\right) =& \psi_*^{1+y}\left( 
 \langle \widetilde{Td}^{T}_{y*}(X(u)^\circ), \widetilde{\operL}^H_{v^{-1}w_0} ([e_{w_0}]_{T}) \rangle_H \right) \\
=&  \langle \psi_*^{1+y}\left(\widetilde{Td}^{T}_{y*}(X(u)^\circ)\right), \psi^*_{1+y}\left(\widetilde{\operL}^H_{v^{-1}w_0} ([e_{w_0}]_{T}) \right) \rangle_H \\
=&  \langle Td^{T}_{y*}(X(u)^\circ), \operL^H_{v^{-1}w_0} \psi^*_{1+y} ([e_{w_0}]_{T}) \rangle_H\:.
\end{align*}
Then the claim follows because
\[ \psi^*_{1+y}\left( \widetilde{Td}_y^{T}(T_{e_{w_0}}X)\right) =(1+y)^{\dim X}Td_y^{T}(T_{e_{w_0}}X) \textrm{
 and }\psi^*_{1+y} ([e_{w_0}]_{T}) = (1+y)^{\dim X}[e_{w_0}]_{T} \/, \]
with $[e_{w_0}]_{T}$ viewed as an equivariant cohomology class of complex  degree $\dim X$ (by equivariant Poincar\'e duality).
\end{proof}

We finish this section with the counterpart of \Cref{thm.msegre},
using now the (un)normalized Segre version of the Hirzebruch classes:
\[ \frac{\widetilde{Td}^{T}_{y*}(X(w)^\circ)}{\widetilde{Td}_y^{T}(TX)} \/ \quad \textrm{ and } \quad 
\frac{Td^{T}_{y*}(X(w)^\circ)}{Td_y^{T}(TX)} \/. \]
Observe that for any smooth $X$, the class $\widetilde{Td}_y^{T}(X)$ is invertible in 
the completed ring
$\widehat{H}^*_{T}(X,\mathbb{Q}[y,(1+y)^{-1}])$, since its leading term is the invertible element $1+y$. 
Similarly $Td_y^{T}(X)$ is invertible in 
$ \widehat{H}^*_{T}(X,\mathbb{Q}[y])$, as its leading term is $1$.
\begin{theorem}\label{thm:Hirze-segre}
For any $w\in W$ one has in $ \widehat{H}^*_{T}(X,\mathbb{Q}[y,(1+y)^{-1}])$ resp. 
$ \widehat{H}^*_{T}(X,\mathbb{Q}[y])$:
$$
\frac{\widetilde{Td}^{T}_{y*}(X(w)^\circ)}{\widetilde{Td}_y^{T}(TX)}
= \widetilde{\mathcal{T}}_{w^{-1}}^{\vee,H}\left(\frac{ [e_{\id}]_{T}}{\widetilde{Td}_y^{T}(T_{e_{\id}}X)}\right)
\quad \text{and} \quad
\frac{Td^{T}_{y*}(X(w)^\circ)}{Td_y^{T}(TX)}
= \mathcal{T}_{w^{-1}}^{\vee,H}\left(\frac{ [e_{\id}]_{T}}{Td_y^{T}(T_{e_{\id}}X)}\right)\:,
$$
as well as 
$$
\frac{\widetilde{Td}^{T}_{y*}(Y(w)^\circ)}{\widetilde{Td}_y^{T}(TX)}
= \widetilde{\mathcal{T}}_{(w_0w)^{-1}}^{\vee,H}\left(\frac{ [e_{w_0}]_{T}}{\widetilde{Td}_y^{T}(T_{e_{w_0}}X)}\right)
\quad \text{and} \quad
\frac{Td^{T}_{y*}(Y(w)^\circ)}{Td_y^{T}(TX)}
= \mathcal{T}_{(w_0w)^{-1}}^{\vee,H}\left(\frac{ [e_{w_0}]_{T}}{Td_y^{T}(T_{e_{w_0}}X)}\right)\:.
$$
\end{theorem}

\begin{proof}
We only explain the proof for the opposite Schubert cells $Y(w)^\circ$, since  the result for the Schubert cells $X(w)^\circ$ are shown in exactly the same way. (Alternatively, one may apply the automorphism $w_0^L$ obtained 
by left multiplication by the longest element $w_0 \in W$; see \cite[\S 5.2]{AMSS:shadows} or \cite[\S 3.1]{MNS}.)
We start with the unnormalized classes. 
The application of $\ch_T$ to~\eqref{eq:SegreMC} together with \Cref{lemma:Hcomm}(a)
%$\ch_T\mathcal{T}_{(w_0w)^{-1}}^\vee=\widetilde{\mathcal{T}}_{(w_0w)^{-1}}^{H,\vee} \ch_T$
imply:
\begin{align*}
\frac{\ch_T(MC_y(Y(w)^\circ))}{\ch_T(\lambda_yT^*X)}=& 
\widetilde{ \mathcal{T}}_{(w_0w)^{-1}}^{H,\vee}\left( \frac{\ch_T(\cO^{w_0, T})}{\ch_T\left(\prod_{\alpha > 0} (1+ y e^{-\alpha})\right)}\right) \\
=&\widetilde{ \mathcal{T}}_{(w_0w)^{-1}}^{H,\vee}\left( \frac{[e_{w_0}]_{T}}{\widetilde{Td}_y^{T}(T_{e_{w_0}}X)}\right) \:,
\end{align*}
with the last equality as in the proof of \Cref{thm:Hdual}.
Then the result follows from
$$\widetilde{Td}^{T}_{y*}(Y(w)^\circ)=  td^{T}_*(MC_y(Y(w)^\circ))=\ch_T(MC_y(Y(w)^\circ)) Td^{T}(TX)$$
and
$$\widetilde{Td}_y^{T}(TX)=\ch_T(\lambda_yT^*X)Td^{T}(TX)\:.$$
To deduce the result for the normalized classes, we further apply the Adams transformation $\psi^*_{1+y}$, with $td^{T}_*(MC_y(Y(w)^\circ))$ and $[ e_{e_0}]_T$ viewed as an equivariant cohomology class as before (by equivariant Poincar\'e duality, with $[e_{w_0}]_{T}$  of complex  degree $\dim X$):
\[ \begin{split} 
\psi^*_{1+y} \bigg(\frac{\widetilde{Td}^{T}_{y,*}(Y(w)^\circ)}{\widetilde{Td}_y^{T}(TX)} \bigg)
= &\left( \psi^*_{1+y}
\widetilde{ \mathcal{T}}_{(w_0w)^{-1}}^{H,\vee}\right) \left( \frac{[e_{w_0}]_{T}}{\widetilde{Td}_y^{T}(T_{e_{w_0}}X)}\right)
\\=&
\left( \mathcal{T}_{(w_0w)^{-1}}^{H,\vee}\psi^*_{1+y}\right)\left( \frac{[e_{w_0}]_{T}}{\widetilde{Td}_y^{T}(T_{e_{w_0}}X)}\right)
\\=& 
\mathcal{T}_{(w_0w)^{-1}}^{H,\vee}\left( \frac{(1+y)^{\dim X}[e_{w_0}]_{T}}{(1+y)^{\dim X}Td_y^{T}(T_{e_{w_0}}X)}\right) \/.
\end{split}\]
%
%\begin{align*}
%\psi^*_{1+y} \bigl(\frac{\widetilde{Td}^{T}_{y*}(Y(w)^\circ)\right)}{\widetilde{Td}_y^{T}(TX)} \bigr)=&
%%\frac{\psi^*_{1+y}\left(\widetilde{Td}^{T}_{y*}(Y(w)^\circ)\right)}{\psi^*_{1+y}\left(\widetilde{Td}_y^{T}(TX)\right)} =
%\left( \psi^*_{1+y}
%\widetilde{ \mathcal{T}}_{(w_0w)^{-1}}^{H,\vee}\right) \left( \frac{[e_{w_0}]_{T}}{\widetilde{Td}_y^{T}(T_{e_{w_0}}X)}\right)\\
%=&\left( \mathcal{T}_{(w_0w)^{-1}}^{H,\vee}\psi^*_{1+y}\right)\left( \frac{[e_{w_0}]_{T}}{\widetilde{Td}_y^{T}(T_{e_{w_0}}X)}\right)\\
%=& \mathcal{T}_{(w_0w)^{-1}}^{H,\vee}\left( \frac{(1+y)^{\dim X}[e_{w_0}]_{T}}{(1+y)^{\dim X}Td_y^{T}(T_{e_{w_0}}X)}\right) \:.
%\end{align*}
Then the result follows from
$$\psi^*_{1+y}\left(\widetilde{Td}_y^{T}(TX)\right)= (1+y)^{\dim X} Td_y^{T}(TX)$$
and
$$\psi^*_{1+y}\left(\widetilde{Td}^{T}_{y*}(Y(w)^\circ)\right)=(1+y)^{\dim X}Td^{T}_{y*}(Y(w)^\circ)\:,$$
since by the module property \eqref{eq:psi-module},
$$\psi^*_{1+y}(-)\cap [X]_T=\psi^*_{1+y}(-)\cap\psi^{1+y}_*((1+y)^{\dim X}[X]_T)=(1+y)^{\dim X}\psi^{1+y}_*( - \cap [X]_T)\:.$$ \end{proof}
}

\subsection{Specializations of (dual) Hirzebruch operators and Hirzerbruch classes of Schubert cells}\label{s:specHclasses}
We take the opportunity to record the specializations at $y=-1$ and $y=0$ for the 
(un)normalized Hirzebruch operators $\widetilde{\mathcal{T}}_i^H$ and $\mathcal{T}_i^H$ 
 and  their (shifted)  dual operators $\widetilde{\mathcal{T}}_i^{H,\vee}, \mathcal{T}_i^{H,\vee}$ and
$\widetilde{\operL}_i^H, \operL_i^H$. These  follow from the definitions of the objects involved, 
utilizing the corresponding specializations of the  Hirzebruch classes from \eqref{E:Tdy=0} and \eqref{E:Tdy=-1},
and are stated in the next proposition. 

\begin{prop} The following hold:

(a) The specializations at $y=0$ of the (un)normalized Hirzebruch operators are given by
\[ (\widetilde{\mathcal{T}}_i^H)_{y=0} = ({\mathcal{T}}_i^H)_{y=0} = Td^T (T_{p_i}) \parcoh_i -\id \/,\]
and for their duals by
\[ (\widetilde{\mathcal{T}}_i^{H,\vee})_{y=0} = ({\mathcal{T}}_i^{H,\vee})_{y=0} = \parcoh_i Td^T (T_{p_i})  -\id \]
so that
\[ (\widetilde{\operL}_i^H)_{y=0} = ({\operL}_i^H)_{y=0} = \parcoh_i Td^T (T_{p_i})  \:. \]

(b) The specializations at $y=-1$ of the (un)normalized Hirzebruch operators is given by:

\[ (\widetilde{\mathcal{T}}_i^H)_{y=-1} = -s_i \/; \quad  ({\mathcal{T}}_i^H)_{y=-1} = \mathcal{T}_i^\mathrm{coh} \/, \]
 and for their (shifted) duals by
\[ (\widetilde{\mathcal{T}}_i^{H,\vee})_{y=-1} = (\widetilde{\operL}_i^H)_{y=-1} = -s_i^{\vee}=s_i \/; \quad 
({\mathcal{T}}_i^{H,\vee})_{y=-1} = ({\operL}_i^H)_{y=-1} = \mathcal{T}_i^{coh,\vee} \:,\]
where the operators $s_i$ and $\mathcal{T}_i^\mathrm{coh}, \mathcal{T}_i^{\mathrm{coh},\vee}$ are defined in \eqref{E:sidef} respectively \eqref{E:Ticohdef}.
\end{prop}

Using these specializations of the (shifted dual) Hirzebruch operators, we can 
specialize in a similar way the corresponding results 
from the Theorems~\ref{prop:HDL-rec},~\ref{thm:Hdual} and~\ref{thm:Hirze-segre} to $y=0$ and $y=-1$. First we consider the case $y=0$. We obtain: 
\begin{equation}\label{eq:classes-cell1}
 Td^{T}_{y=0,*}(X(w)^\circ)=\widetilde{Td}^{T}_{y=0,*}(X(w)^\circ) = td_*(MC_0(X(w)^\circ)  )=td_*^{T}( \mathcal{I}^T_w) 
\end{equation}
by \Cref{prop:specialize}(b). Then \Cref{prop:HDL-rec} specializes for $y=0$ to the recursion
$$ (Td^{T}(T_{p_i})\parcoh_i -\id)(td^{T}_*( \mathcal{I}^T_w)) =  td^{T}_*( \mathcal{I}^T_{ws_i}) 
 \textrm{ for } ws_i>w $$
from \eqref{equ:BGGTodd-I} for $w\in W$ and $s_i$ a simple reflection. Similarly \Cref{thm:Hdual} specializes for $y=0$ to 
 \begin{equation}\label{equ:Hirze-dual2y=0} 
 \langle Td^{T}_*(\mathcal{I}^T_u), (\operL^H_{v^{-1}w_0})_{y=0} ([e_{w_0}]_{T}) \rangle_H =
\delta_{u,v} Td^T(T_{e_{w_0}}X)  \end{equation}
for $u,v \in W$. Since $\ch_T(\calO^{w_0,T})=\frac{[e_{w_0}]_T}{Td^T(T_{e_{w_0}}X)}$, this translates into 
$$ \langle Td^{T}_*(\mathcal{I}^T_u), (\operL^H_{v^{-1}w_0})_{y=0} (\ch_T(\calO^{w_0,T})) \rangle_H =\delta_{u,v}\:,$$
from which we deduce by \eqref{equ:todd-duals} (for $v\in W$) that:
\begin{equation}\label{equ:todd-duals=L}
  \ch_{T}(\cO^{v,T})  = (\operL^H_{v^{-1}w_0})_{y=0} (\ch_T(\calO^{w_0,T})) \:,
\end{equation}
with $ ({\operL}_i^H)_{y=0} = \parcoh_i Td^T (T_{p_i})$.
This recovers \Cref{equ:BGGChern}.
%is consistent with \Cref{cor:sp-MC}(b) and \Cref{lemma:Hcomm}(a).

Finally, since $td^T_*(-)=\ch_T(-)Td^T(TX)$ for $y=0$, \Cref{thm:Hirze-segre} specializes to 
\begin{equation}\ch_{T} (\mathcal{I}^T_w)) = \left(\mathcal{T}_{w^{-1}}^{\vee,H}\right)_{y=0}(\ch_{T} (\mathcal{I}^T_{\id}))
\quad \text{and} \quad
 \ch_{T} (\mathcal{I}^{w,T})) = \left(\mathcal{T}_{(w_0w)^{-1}}^{\vee,H}\right)_{y=0}(\ch_{T} (\mathcal{I}^{w_0,T}))
\end{equation}
for $w\in W$, with  $({\mathcal{T}}_i^{H,\vee})_{y=0} = \parcoh_i Td (T_{p_i})  -\id$, consistent
with \Cref{equ:BGGChern-I}.

Next we record the specializations of Theorems \ref{prop:HDL-rec},~\ref{thm:Hdual} and~\ref{thm:Hirze-segre} for $y=-1$. For simplicity we only consider the more interesting case of normalized  classes and operators.
Note that
\begin{equation}\label{eq:classes-cell2}
 Td^{T}_{y=-1,*}(X(w)^\circ)=c_*^T(\one_{X(w)^\circ})=:\csmT(X(w)^\circ) \in H^T_*(G/B,\mathbb{Z})
\end{equation}
by  \eqref{eq:T=c}, since $H_*^{T}(G/B,\mathbb{Z})$ is torsion free.~Recall that $\csmT(X(w)^\circ)$ is
\textit{Chern-Schwartz-MacPherson} (CSM) class of the Schubert cell $X(w)^\circ$; this class is discussed 
in more detail in the next section. \Cref{prop:HDL-rec}(b) for  the normalized Hirzebruch classes 
\emph{implies} for $y=-1$ the following important recursion
 of \cite[Theorem~6.4]{aluffi.mihalcea:eqcsm}
and \cite[Theorem~6.1]{AMSS:shadows} 
%\cite[Theorem~6.1]{AMSS:shadows} 
(formulated in \Cref{equ:csm} in the next section, in terms of homogenized classes):
\begin{equation}\label{eq:rec-CSM}
\mathcal{T}_i^\mathrm{coh}(\csmT(X(w)^\circ))=   (c^{T}(T_{p_i})\parcoh_i -\id)(\csmT(X(w)^\circ))=\csmT(X(ws_i)^\circ)
\end{equation}
for $w\in W$ and $s_i$ a simple reflection, with $ws_i>w$.

Similarly, ~\Cref{thm:Hdual} for the normalized Hirzebruch classes \emph{implies} for $y=-1$
the corresponding {\bf Hecke orthogonality} of 
\cite[Theorem~7.2]{AMSS:shadows} (with their equivariant parameter $\hbar\in H^2_{\bbC^*}(pt,\mathbb{Z})$ specialized here to $\hbar=1$):
%\cite[Theorem~7.2]{AMSS:shadows}:
\begin{equation}
 \langle \csmT(X(u)^\circ), (\operL^H_{v^{-1}w_0})_{y=-1} ([e_{w_0}]_{T}) \rangle =
\delta_{u,v} c^{T}(T_{e_{w_0}}X) = \delta_{u,v} \prod_{\alpha> 0} (1+\alpha)
\end{equation}
for $u,v\in W$. Here 
$$(\operL^H_{v^{-1}w_0})_{y=-1} ([e_{w_0}]_{T}) =  \mathcal{T}_{v^{-1}w_0}^{coh,\vee}  ([e_{w_0}]_{T}) =: {\csmTv}(Y(v)^\circ)$$
is the \textit{dual Chern-Schwartz-MacPherson} class from \cite[Eq.~(14)]{AMSS:shadows}.
%\cite[Definition~5.3]{AMSS:shadows},
Note that, in terms of the duality operators from \Cref{rem:duality},
$${\csmTv}(Y(v)^\circ)=(-1)^{\dim X-\ell(v)}\psi_*^{-1}(\csmT(Y(v)^\circ)\:,$$
since the homogenized operators satisfy
 \begin{equation}\label{E:Tihomsign} \mathcal{T}_i^{coh,\vee,\hbar} = \hbar \parcoh_i  + s_i = -(-\hbar \parcoh_i  - s_i)
= -\mathcal{T}_i^{coh,\hbar}|_{h\mapsto -h}  \/. \end{equation}

Finally \Cref{thm:Hirze-segre} for  the normalized Hirzebruch classes {implies} for $y=-1$:
%(since $H_*^{T}(G/B,\mathbb{Z})$ is torsion-free):
\begin{equation}\label{E:ssmrec}
\frac{\csmT(X(w)^\circ)}{c^{T}(TX)}
= \mathcal{T}_{w^{-1}}^{\vee,coh}\left(\frac{ [e_{\id}]_{T}}{\prod_{\alpha >0} (1-\alpha)}\right) 
= \frac{{\csmTv}(X(w)^\circ)}{\prod_{\alpha >0} (1-\alpha)}
\end{equation}
and
\begin{equation}
\frac{\csmT(Y(w)^\circ)}{c^{T}(TX)}
= \mathcal{T}_{(w_0w)^{-1}}^{\vee,coh}\left(\frac{ [e_{w_0}]_{T}}{\prod_{\alpha >0} (1+\alpha)}\right) 
= \frac{{\csmTv}(Y(w)^\circ)}{\prod_{\alpha >0} (1+\alpha)} \:.
\end{equation}
This recovers ~\cite[Theorem~7.5]{AMSS:shadows}, 
%~\cite[Theorem~7.3]{AMSS:shadows},
which is one of the main results of that paper (again with the
equivariant parameter $\hbar\in H^2_{\bbC^*}(pt,\mathbb{Z})$ specialized to $\hbar=1$).

\subsection{Parabolic Hirzebruch classes}\label{subsec-Parab-Hierze} We now
consider the (generalized) partial
flag manifold $G/P$, and we let $\pi:G/B \to G/P$ be the natural projection. 
The Schubert varieties $X(wW_P)^\circ$ in $G/P$ are indexed by 
the elements in $w\in W^P$, with the image $\pi(X(w)^\circ) = X(wW_P)^\circ$ for $w\in W$. 
Applying $td^T_*$ and $\psi_*^{1+y}td_*^T$ to the equalities from~\Cref{prop:pf} implies 
by functoriality the following counterpart for the Hirzebruch classes.
 
\begin{prop}\label{prop:pf2} The following hold for $w\in W$:

(a) $\pi_* \widetilde{Td}^{T}_{y*}(X(w)^\circ) = (-y)^{\ell(w) - \ell(w W_P)} \widetilde{Td}^{T}_{y*}(X(wW_P)^\circ)$ in $\widehat{H}_*^T(G/P)[y]$
and 
 $$\pi_*Td^{T}_{y*}(X(w)^\circ) = (-y)^{\ell(w) - \ell(w W_P)} Td^{T}_{y*}(X(wW_P)^\circ) \in \widehat{H}_*^T(G/P)[y] \:.$$
(b) More generally, let $P \subseteq Q$ be two standard parabolic subgroups, and $\pi':G/P \to G/Q$ the natural projection. 
Then
\[ \pi'_* \widetilde{Td}^{T}_{y*}(X(wW_P)^\circ)= (-y)^{\ell(wW_P) - \ell(wW_Q)} \pi_* \widetilde{Td}^{T}_{y*}(X(wW_Q)^\circ) 
\in \widehat{H}_*^T(G/Q)[y]  \]
and
\[ \pi'_*Td^{T}_{y*}(X(wW_P)^\circ)= (-y)^{\ell(wW_P) - \ell(wW_Q)} \pi_* Td^{T}_{y*}(X(wW_Q)^\circ) \in \widehat{H}_*^T(G/Q)[y] \/. \]
\end{prop}

Specializing to $y=0$ (with the convention $0^0=1$), we get by \Cref{prop:specialize}(b):
\[ \pi_*Td^{T}_{*}(\mathcal{I}^T_w) = \begin{cases} Td^{T}_{*}(\mathcal{I}^T_{wW_P}) & \textrm{ if } \ell(w) = \ell(w W_P) \/; \\
0 & \textrm{ otherwise} \/. \end{cases} \]
Similarly, specializing the normalized classes to $y=-1$, we get for $w\in W$:
\begin{equation}\label{E:pushcsm}
\pi_*\csmT(X(w)^\circ) = \csmT(X(wW_P)^\circ) \textrm{ and } \pi'_*\csmT(X(wW_P)^\circ) =  \csmT(X(wW_Q)^\circ) \/.
\end{equation}
These equalities hold in $H_*^{T}(G/P,\mathbb{Z})$ and $H_*^T(G/Q,\mathbb{Z})$, since these are torsion
free.

\section{The Chern-Schwartz-MacPherson classes as leading terms} We have seen in the \Cref{cor:special-Ty} that the Chern-Schwartz-MacPherson (CSM) classes may be recovered from the Hirzebrch 
transformation by specializing at $y=-1$. In this section we take a different route, and 
recover the CSM classes directly, by identifying them as the leading terms of the motivic 
Chern classes.

\subsection{Chern-Schwartz-MacPherson classes}\label{sec:CSM} 
According to a conjecture attributed to Deligne and Grothendieck, there is a unique natural 
transformation $c_*: \cF(-) \to H_*(-,\mathbb{Z})$ from the functor of constructible functions 
on a complex algebraic variety $X$ to homology (i.e., even degree Borel-Moore homology, 
or Chow groups), where all morphisms are proper, such that if $X$ is smooth then 
$c_*(\one_X)=c(TX)\cap [X]$.  This conjecture was proved by 
MacPherson \cite{macpherson:chern}; the class $c_*(\one_X)$ for possibly singular 
$X$ was shown to coincide with a class defined earlier by 
M.-H.~Schwartz \cite{schwartz:1, schwartz:2, BS81}. For any constructible subset 
$W\subseteq X$, the class 
$c_{SM}(W):=c_*(\one_W)\in H_*(X,\mathbb{Z})$ is called the \textit{Chern-Schwartz-MacPherson} (CSM) class of $W$ in $X$. If $X$ is a $T$-variety, an equivariant version of the group of constructible
functions $\mathcal{F}^T(X)$ and a Chern class transformation 
$c_*^T: \mathcal{F}^T(X) \to H_*(X;\mathbb{Z})$ were defined by Ohmoto \cite{ohmoto:eqcsm}.

\subsection{The homogenized CSM class via the motivic Chern class }
We now consider $X=G/B$. If
\[\csmT(X(w)^\circ)=\sum_i \csmT(X(w)^\circ)_i\in H^{{T}}_*(G/B,\mathbb{Z}),\] 
where $\csmT(X(w)^\circ)_i\in H_{2i}^{{T}}(G/B,\mathbb{Z})$, the homogenized CSM class is 
defined to be
$$\csmTh(X(w)^\circ):=\sum_i\hbar^i\csmT(X(w)^\circ)_i \in H_0^{{T} \times \C^*}(G/B,\mathbb{Z}) \/.$$ 
Here $\bbC^*$ acts trivially on $G/B$ and $\hbar\in H^2_{\bbC^*}(pt,\mathbb{Z})$ is a generator. 
Consider the Schubert expansion of the homogenized CSM class:
\[ \csmTh(X(w)^\circ) = \sum_{u \le w} c'_{u,w}(\hbar,t) [X(u)]_{T} \in H_0^{T \times \C^*}(X) \/,\] 
where $c'_{u,w}(\hbar,t) \in
 H^*_{{T} \times \C^*}(pt,\mathbb{Z}) = \Z[\hbar; \alpha_1, \ldots \alpha_r]$ is a homogeneous 
 polynomial of degree $\ell(u)$. It was proved in \cite{aluffi.mihalcea:eqcsm} (see also \cite{AMSS:shadows}, 
 or \eqref{eq:rec-CSM} above) that  
\begin{equation}\label{equ:csm}
\mathcal{T}_i^{coh,\hbar} (\csmTh(X(w)^\circ) = \csmTh(X(w s_i)^\circ). 
\end{equation}
Combined with \Cref{prop:initial} this implies that the CSM class of the Schubert cell is the `initial term' of the 
motivic Chern class $MC_y(X(w)^\circ)$, where $y=-e^{-\hbar}$. We make this precise next. 

\begin{theorem}\label{prop:csm=initial} Let $w \in W$ and consider the Schubert expansions 
\begin{equation}\label{E:MC-to-schub} MC_y (X(w)^\circ) = \sum_{ u \le w} c_{u,w}(y,e^t) \cO_u^T \in \K_T(G/B)[y] \end{equation}
and
\begin{equation}\label{E:csm-to-schub} \csmTh(X(w)^\circ) = \sum_{ u \le w} c'_{u,w}(\hbar, t) [X(u)]_T \in H_0^{T \times \C^*}(G/B,\mathbb{Z}) \/, \end{equation}
where $ c'_{u,w}(\hbar, t) \in H^{2 \ell(u)}_{T \times \C*}(pt)$ and $\ch_{\C^*}(y)=-e^{-\hbar}$. 
Then the following hold:

(a) The image $\ch_{T \times \C^*} (c_{u,w}(y,e^t))$ of the coefficient $c_{u,w}(y,e^t)$ under 
the Chern character belongs to 
$\prod_{i \ge \ell(u)} H^{2i}_{T \times \C^*}(pt)$.

(b) The coefficient $c'_{u,w}(\hbar, t)$ equals the term of degree $\ell(u)$ in $c_{u,w}(-e^{-\hbar},e^t)$, i.e., 
\[ c'_{u,w}(\hbar, t) =  (\ch_{T \times \C^*} (c_{u,w}(y,e^t)))_{\ell(u)} \/.\] 
Equivalently,
\[ \csmTh(X(w)^\circ) = \textrm {degree $0$ component of } \ch_T(MC_y (X(w)^\circ)) \/. \]
\end{theorem}
\begin{proof} Both parts follow by induction on $\ell(w)$, using the recursion calculating
$MC_y(X(w)^\circ)$ from \Cref{thm:MC+dual} combined with \Cref{prop:initial}; in part (b)
we utilize the recursion for $\csmTh(X(w)^\circ)$ from
\eqref{equ:csm}.\end{proof}
\begin{example}\label{ex:mcp1-to-csm} Consider the equivariant motivic Chern class in $\K_T(\mathbb{P}^1)[y]$:
\[ MC_y(X(s)^\circ) = (1+ e^{-\alpha_1} y)\cO_{s_1}^{T} - (1 + (1+ e^{-\alpha_1})y) \cO_{\id}^{T}  \/. \]
The specialization $y=-e^{-\hbar}$ in the coefficient $c_{s_1,s_1}(y,e^t)$ gives:
\[ c_{s_1,s_1}(-e^{-\hbar},e^t) = 1 - e^{\alpha_1 - \hbar} = \hbar - \alpha_1 + h.o.t. \]
The term of degree $1$ is $c'_{s_1,s_1}= \hbar -\alpha_1$.
Similarly, the specialization of $c_{\id,s_1}(y,e^t)$ gives 
\[ - 1 + 2 e^{-\hbar} - \alpha_1 e^{-\hbar} +h.o.t. = -1 + h.o.t. \] By \Cref{prop:csm=initial},
\[ \csmTh(X(w)^\circ) = (\hbar - \alpha_1) [X(s_1)]_T + [X(\id)]_T \/. \]
\end{example}
Consider now the non-equivariant case, i.e., in the expansions from \Cref{prop:csm=initial}
we set $\alpha= 0$, so $e^\alpha \mapsto 1$. In this case we denote the coefficients
by $c_{u,w}(y)$ and $c'_{u,w}(\hbar)$. Note that $c_{u,w}(y) \in \Z[y]$ and 
$c'_{u,w}(\hbar) \in \Z[\hbar]$. Furthermore, by homogeneity, 
\[ c'_{u,w}(\hbar) = \overline{c}_{u,w} \hbar^{\ell(u)} \/,\] 
where $\overline{c}_{u,w} \in \Z$ is an integer.
Next we give a more direct relation between these 
coefficients. 

Recall from \Cref{prop:divisibility} that the polynomial $c_{u,w}(y)$ is divisible by $(1+y)^{\ell(u)}$.
 
\begin{corol}\label{thm:csm} The coefficient
$\overline{c}_{u,w}$ equals the specialization at $y=-1$ of $\frac{c_{u,w}(y)}{(1+y)^{\ell(u)}}$:
\[ \overline{c}_{u,w}= \bigg(\frac{c_{u,w}(y)}{(1+y)^{\ell(u)}}\bigg)_{y \mapsto -1} \/. \]
\end{corol}

\begin{proof} Let $Q_{u,w}(y) := \sum a_i y^i$ in $\Z[y]$
be the quotient $\frac{c_{u,w}(y)}{(1+y)^{\ell(u)}}$. By \Cref{prop:csm=initial}, the coefficient
$\overline{c}_{u,w}$ equals the term of degree $0$ in the specialization 
$Q_{u,w}(-e^{-\hbar})$. Since $-e^{-\hbar} = -1 + $ higher order terms,   
$\overline{c}_{u,w} = Q_{u,w}(-1)$. This finishes the proof.
\end{proof}
\begin{example} Consider the non-equivariant version of \Cref{ex:mcp1-to-csm} :
\[ MC_y(X(s_1)^\circ)  = (1+y) \cO_{s_1} - (1+2y) \cO_{\id} \/. \]
According to \Cref{thm:csm}, we need to divide each coefficient 
$c_{u,s_1}(y) $ by $(1+y)^{\ell(u)}$ and then specialize at $y=-1$. We obtain:
\[ \csmh(X(s_1)^\circ) = \hbar [X(s_1)] + [X(\id)] \/. \]
The non-homogenized class is obtained by setting $\hbar$ to $1$.
\end{example}

\begin{example} Consider the motivic Chern class $MC_y(X(s_1 s_2)^\circ) \in \K(\Fl(3))[y]$: 
\[ MC_y(X(s_1 s_2)^\circ) = (1+y)^2 \cO_{s_1 s_2} -  (1+y) (1+2y) \cO_{s_1} - (1+y)(1+3y) \cO_{s_2} + (5y^2+ 5y+1) \cO_{\id} \/ \/. \]
As before, we need to divide each coefficient $c_{u,s_1 s_2}(y) $ by $(1+y)^{\ell(u)}$ and then specialize at $y=-1$. We obtain:
\[ \csm(X(s_1 s_2)^\circ) = [X(s_1 s_2)] +[X(s_1)] + 2 [X(s_2)] + [X(\id)] \/. \]
\end{example}

\section{Positivity, unimodality, and log concavity conjectures}\label{ss:pos} 
In this section, we record several conjectures involving Schubert expansions of the motivic 
Chern classes, and of the CSM classes, and about the structure constants of the CSM
classes. Some of these conjectures have been observed by other authors, and our goal 
to collect them in a single place.

We start with the CSM classes, since this is the case when we have the most partial results.

\subsection{Positivity of Schubert expansions of CSM classes} Consider the 
(non-equivariant) CSM class of a Schubert
cells in a generalized flag manifolds $G/P$:
\[ \csm(X(wW_P)^\circ) = \sum_{vW_P \le wW_P} c_{v,w} [X(vW_P)] \/, \]
%\[ \csm(X(wW_P)^\circ) = \sum_{vW_P \le wW_P} c_{v,w}(t) [X(vW_P)] \/, \]
with $c_{v,w} \in \Z$. For $G/P=\Gr(k;n)$, it was conjectured in \cite{aluffi.mihalcea:csm} that the 
coefficients $c_{v;w}$~are nonnegative; this was proved in some special cases in {\em loc.cit.}~and 
in \cite{mihalcea:binomial, jones:csm,stryker:thesis}, and in full generality (for 
Grassmannians) by Huh~\cite{huh:csm}, also see below. 
The recursive algorithm from \cite{aluffi.mihalcea:eqcsm} yielded calculations of CSM 
classes of Schubert 
cells in any $G/P$, and provided supporting evidence that the CSM classes of Schubert cells
in all flag manifolds are effective. This conjecture was recently proved in~\cite{AMSS:shadows}.

Equivariantly, the numerical evidence supports the following conjecture.
\begin{conj}[Equivariant Positivity]\label{conj:eqcsmpos} Let $X(wW_P)^\circ \subseteq G/P$ be any Schubert cell and 
consider the Schubert expansion of the equivariant CSM class:  
\[ \csmT(X(wW_P)^\circ) = \sum_{v \le w} c_{v,w}(\alpha) [X(vW_P)]_T \quad \in H_*^T(G/P) \/. \]
Then $c_{v;w}(\alpha)$ is a polynomial in positive roots $\alpha$ with non-negative coefficients.
\end{conj}
In the non-equivariant case, while we have proved that $c_{v,w}\ge 0$ in~\cite{AMSS:shadows}, 
the evidence suggests a stronger result.
\begin{conj} [Strong positivity]\label{conj:strong-pos} Let $X(wW_P)^\circ \subseteq G/P$ be any Schubert cell and 
consider the Schubert expansion: 
\[ \csm(X(wW_P)^\circ) = \sum_{v \le w} c_{v,w} [X(vW_P)] \quad \in H_*(G/P; \Z) \/. \]
Then $c_{v,w} >0$ for all $v\le w$.
\end{conj} 
Huh's result for Grassmannians in \cite{huh:csm} shows that 
each homogeneous component $\csm(X(wW_P)^\circ)_k$ of the CSM class is 
represented by a non-empty irreducible variety. This is slightly weaker than the requirement 
in \Cref{conj:strong-pos}. On the other hand, {if this variety may be chosen to be $T$-stable, 
then Huh's result and the positivity 
results of Graham \cite{graham:pos} would imply Conjecture~\ref{conj:eqcsmpos}
for Grassmannians.}

Let $\pi: G/B \to G/P$ be the natural projection. Since $\pi_* (\csm(X(w)^\circ)) = \csm(X(wW_P)^\circ)$
(see e.g., \eqref{E:pushcsm})
it follows that if \Cref{conj:eqcsmpos} or \Cref{conj:strong-pos} holds for cells in $G/B$, then it also 
holds in $G/P$.  

\subsection{Positivity of CSM/SM structure constants}
We now turn to structure constants of the CSM multiplication. The 
CSM classes $\csm(X(w)^\circ)$ `are in homology', and it is natural to multiply their Poincar{\'e}
duals, the Segre-MacPherson (SM) classes. 
In general, if $Z$ is a subvariety of a nonsingular variety $X$, we set
\[ 
\ssm(Z) = \frac{\csm(Z)}{c(T(X))} \in H_*(X)
\]
(the ambient variety $X$ is understood in context, so omitted from the notation).
For instance, in $G/B$ we have
\[
\ssm(Y(w)^\circ) = \frac{\csm(Y(w)^\circ)}{c(T(G/B))}
=c(T^*(G/B))\cap \csm(Y(w)^\circ)\,:
\]
we proved in \cite[Lemma 8.2]{AMSS:shadows} that $c(T(G/B))\cdot c(T^*(G/B))=1$.
We also proved that 
\[ \ssm(X(ws_i)^\circ) = \mathcal{T}_i^{coh,\vee} \ssm(X(w)^\circ) \]
(cf.~\eqref{E:ssmrec} above). `Poincar{\'e} duality' states that,
in $G/P$,
\begin{equation}\label{E:Poincarecsm-ssm} \langle \ssm(Y(vW_P)^\circ), \csm(X(wW_P)^\circ) \rangle = \delta_{v,w} \/; \end{equation}
cf.~\cite[Thm.~7.1]{AMSS:shadows}. 
This can be proved using a transversality formula due to 
Sch{\"u}rmann \cite{schurmann:transversality};
see also \cite[Cor.~10.3]{AMSS:shadows}. 

A key consequence of \eqref{E:ssmrec} (cf.~\cite[Eq.~(36)]{AMSS:shadows}) 
is that for $G/B$, the Schubert expansions of $\csm(X(w)^\circ)$ and 
$\ssm(X(w)^\circ)$ are related by changing signs. More precisely, if 
\[ \ssm(X(w)^\circ) = \sum f_{v;w} [X(v)] \quad \in H_*(G/B) \/, \]
then with notation as above $f_{v;w} = (-1)^{\ell(w) - \ell(v)} c_{v;w}$. This follows because
the homogenized 
operator $\mathcal{T}_i^{coh,\hbar}= \hbar \partial_i-s_i$, giving CSM classes, and its adjoint
$\mathcal{T}_i^{coh,\vee,\hbar}= \hbar \partial_i + s_i$, giving SM classes, differ by a sign.
(See also \eqref{E:Tihomsign} above.) Consider now the structure
constants
\begin{equation}\label{E:ssmmult} \ssm(Y(u)^\circ) \cdot \ssm(Y(v)^\circ) = \sum  e_{u,v}^w \ssm(Y(w)^\circ) \/. \end{equation}
The transversality theorem by Sch{\"u}rmann \cite{schurmann:transversality} shows that
\[ e_{u,v}^w = \chi(g_1 Y(u)^\circ \cap g_2 Y(v)^\circ \cap g_3 X(w)^\circ ) \/, \]
the topological Euler characteristic of the intersection of three Schubert cells translated in general 
position via $g_1, g_2, g_3 \in G$. This interpretation of the structure constants holds 
for any $G/P$, although the relation between the Schubert expansions of CSM and SM classes 
does not extend beyond $G/B$. (Still, the SM classes are known to be Schubert alternating; see 
\cite{AMSS:ssmpos}.) 

Due to its statement involving only `classical' objects, 
perhaps the most remarkable positivity conjecture in this paper is the next.
\begin{conj}[Alternation of Euler characteristic]\label{conj:euler-alt} The Euler characteristic of the intersection
of three Schubert cells in general position in $G/P$ is alternating, i.e., for any $u,v,w \in W^P$,
\[ (-1)^{\ell(u)+\ell(v) +\ell(w)} \chi(g_1 Y(uW_P)^\circ \cap g_2 Y(vW_P)^\circ \cap g_3 X(wW_P)^\circ ) \ge 0\/. \]
\end{conj}
Utilizing deep connections between SM classes to the theory of integrable systems, this
was proved by Knutson and Zinn-Justin for $d$-step flag manifolds 
with $d \le 3$, and it was conjectured to hold for $d=4$; cf.~\cite[p.~43]{knutson.justin:segre}.
Independently, and based on multiplications 
of SM classes from \cite{AMSS:shadows}, the authors of this paper stated this conjecture in several 
conference and seminar talks, for partial flag manifolds
$G/P$ in arbitrary Lie type. S. Kumar \cite{kumar:conjpos}
%, and also independently R. Xiong \cite{xiong:pos}, 
conjectured that the CSM class of the Richardson cells
are Schubert positive, that is, if
\[ \csm(Y(u)^\circ \cap X(v)^\circ) = \sum_w f_{u,v}^w [Y(w)] \/, \]
then $f_{u,v}^w \ge 0$. (We also learned about this conjecture independently
from Rui Xiong, and it is now stated in \cite[Conj. 9.2]{xiong:pieri}.)
It is shown in \cite{kumar:conjpos} that this implies \Cref{conj:euler-alt}.
Note that the {\em Segre} class of the Richardson cell 
\[ \ssm(R_v^{u,\circ}) := \ssm(Y(u)^\circ \cap X(v)^\circ)\] 
is Schubert alternating 
by \cite[Theorem 1.1]{AMSS:ssmpos} (the inclusion of $Y(u)^\circ \cap X(v)^\circ$ 
is an affine morphism). 

For $G/B$, the absolute value of the structure constants 
$e_{u,v}^w$ from \eqref{E:ssmmult}
give the structure
constants to multiply CSM classes of Schubert cells. This generalizes the positivity in ordinary Schubert 
Calculus: if $\ell(uW_P) + \ell(vW_P) = \ell(wW_P)$, 
then the intersection in question is $0$ dimensional and reduced, 
and the Euler characteristic counts the number of points in the intersection. 
A different algorithm to calculate the SM structure
constants is given in~\cite{su:structure}.

We end this section by proving a property of the sum of the coefficients $e_{u,v}^w$.

\begin{prop}\label{prop:sumcsmst} Consider the multiplication 
\[ \ssm(Y(uW_P)^\circ) \cdot \ssm(Y(vW_P)^\circ) = \sum_w e_{u,v}^w \ssm(Y(wW_P)^\circ) \/. \]
Then $\sum_w e_{u,v}^w = \delta_{w_0uW_P,vW_P}$.
\end{prop}
\begin{proof} By the transversality formula from
\cite{schurmann:transversality}, 
\[ \ssm(Y(uW_P)^\circ) \cdot \ssm(Y(vW_P)^\circ) = \ssm(X(w_0uW_P)^\circ) \cdot \ssm(Y(vW_P)^\circ) 
= \ssm(X(w_0uW_P)^\circ \cap Y(vW_P)^\circ) \/. \]
Utilizing that $c(T(G/P)) = \sum_w \csm(X(w)^\circ)$ and the duality from \eqref{E:Poincarecsm-ssm}, the
sum of the coefficients $e_{u,v}^w$ equals
\[ \begin{split} \int_{G/P} \ssm(Y(uW_P)^\circ) & \cdot \ssm(Y(vW_P)^\circ) \cdot c(T(G/P)) \\
& =
\int_{G/P} \ssm(X(w_0uW_P)^\circ \cap Y(vW_P)^\circ) \cdot c(T(G/P)) \\ & = 
\int_{G/P} \csm(X(w_0uW_P)^\circ \cap Y(vW_P)^\circ) = \delta_{w_0uW_P,vW_P} \/. \end{split} \]
Here the last equality follows because the Richardson cell
$X(w_0uW_P)^\circ \cap Y(vW_P)^\circ$ is torus-stable, therefore its Euler characteristic equals the
Euler characteristic of 
the fixed locus; see \cite[Corollary 2]{BB:on-fixed}, applied for a general $\mathbb{C}^* \subseteq T$. 
In this case the fixed locus is empty or one point, giving $\delta_{w_0uW_P,vW_P}$.
\end{proof} 

\begin{example} Take $G$ is Lie type $G_2$. Then
\[ \begin{split} \ssm(Y(\id)^\circ) \cdot  \ssm(Y(\id)^\circ) & = \ssm(Y(\id)^\circ) - \ssm(Y(s_1)^\circ)  - \ssm(Y(s_2)^\circ) \\ & + 2 \ssm(Y(s_2 s_1)^\circ) +   4 \ssm(Y(s_1 s_2)^\circ) 
\\ &  -9 \ssm(Y(s_1s_2 s_1)^\circ) -   11 \ssm(Y(s_2 s_1 s_2)^\circ) \\ & 
 +22 \ssm(Y(s_2s_1s_2 s_1)^\circ) + 34 \ssm(Y(s_1s_2 s_1 s_2)^\circ) \\ &
  -57 \ssm(Y(s_1s_2s_1s_2 s_1)^\circ) -51 \ssm(Y(s_2s_1s_2 s_1 s_2)^\circ) \\ &
+ 67 \ssm(Y(s_2s_1s_2s_1s_2 s_1)^\circ) \/. \end{split} \]
Observe that these structure constants are alternating, and add up to $0$, confirming
\Cref{conj:euler-alt} and \Cref{prop:sumcsmst} in this case.
\end{example}
  
\subsection{Unimodality and log concavity for CSM polynomials}
We now turn to some unimodality and log concavity properties of the coefficients $c_{v;w}$.
Following \cite{stanley:log-concave}, a sequence $a_0, \ldots, a_n$ is {\bf unimodal} if there exists $i_0$ such that
\[ a_0 \le a_1 \le \ldots \le a_{i_0} \ge a_{i_0+1} \ge \ldots \ge a_n \/. \]
The sequence is {\bf log-concave} if for any $1 \le i \le n-1$, 
\[ a_i^2 \ge a_{i-1} a_{i+1} \/. \]
A log-concave sequence of nonnegative integers with no internal zeros is unimodal.
A polynomial $P(x) = \sum a_i x^i$ is unimodal, resp., log-concave, if its sequence of coefficients
satisfies the corresponding property.

Consider now any class $\kappa= \sum c_w [X(wW_P)]$ in $H_*(G/P)$. 
We define the {\bf $H$-polynomial} associated to $\kappa$ by 
\[ H(\kappa) := \sum c_w x^{\ell(w)} = \sum c_i x^i \/, \]
determining the coefficients $c_i$. For $w \in W^P$ we denote by 
$H_w(x):= H(\csm(X(wW_P)))$, the $H$-polynomial of the CSM class of the 
Schubert {\em variety}.

\begin{conj}[Unimodality and log concavity]\label{conj:uni-logc}
Let $X(w) \subseteq G/P$ be any Schubert variety. Then the following
hold:

(a) The polynomial $H_w$ is unimodal with no internal zeros.

(b) If $G$ is of Lie type A (i.e., $G/P= \Fl(i_1, \ldots, i_k;n)$ 
is a partial flag manifold) then $H_w$ is log-concave.
\end{conj}
A similar conjecture for Mather classes can be found in 
\cite{mihalcea.singh:conormal}.
\begin{example} Consider the Grassmannian $\Gr(3,6)$ and 
the Schubert variety $X_{(2,1)}$ of dimension $3$. Then
\[ \csm(X_{(2,1)}) = [X_{(2,1)}] + 3 [X_2] + 3 [X_{1,1}] +8 [X_1] + 5 [X_0] \/. \]
Its $H$-polynomial is 
\[ x^3 + 6 x^2 + 8 x +5 \/, \]
which is log-concave. 

Consider now the $5$-dimensional quadric $Q^5$. The
$H$-polynomial of $c(TQ^5)$ is
\[ x^5 + 5x^4 + 11x^3 + 26x^2 + 18x + 6 \]
which is unimodal, but not log-concave.
\end{example}

\subsection{Conjectures about the motivic Chern classes}
The goal of this section is to state some conjectures for the motivic Chern classes of the Schubert cells. 
Given the (proved and conjectural) positivity properties of the CSM classes of Schubert cells,
it is natural to expect that the motivic Chern classes of Schubert cells will also 
satisfy a positivity conjecture. The following 
conjecture was stated by F{e}h{\'e}r, Rim{\'a}nyi, and Weber \cite{feher2018characteristic}
in type A, and in \cite[Conjecture 1]{AMSS:motivic} for arbitrary Lie type. 
\begin{conj}[Positivity of MC classes]\label{conj:k} Consider the Schubert expansion: 
\[ MC_y^{T}(X(w)^\circ)=\sum_{u\leq w}c_{u,w}(y,e^t)\calO_u^{T}\in K_{\mathbb{T}}(G/B)[y].\]
Then for any $u\leq w\in W$, we have
\[(-1)^{\ell(w)-\ell(u)}c_{w,u}(y,e^t)\in \bbZ_{\geq 0}[y][e^{-\alpha_1}, \ldots , e^{-\alpha_r}] \/,\]
i.e., the coefficients $(-1)^{\ell(w) - \ell(u)} c_{u,w}(y,e^t)$ are polynomials in the variables $y$ and $e^{-\alpha_1}, \ldots , e^{-\alpha_r}$  with non-negative coefficients. 
\end{conj}
The conjecture implies that the coefficients of the non-equivariant motivic 
Chern classes of Schubert cells are sign-alternating: $(-1)^{\ell(w) - \ell(u)} c_{u,w}(y) \in \Z_{\ge 0} [y]$. 
If we specialize $y=0$, then the positivity conjecture is known from \Cref{prop:specialize}(b) 
and the fact that
the ideal sheaves are alternating in Schubert classes; see~\eqref{equ:idealstru} above, and e.g., \cite{brion:flagv}.
Some particular coefficients $c_{u,w}(y,e^t)$ are known to be positive.  
For instance, the coefficient 
\[ c_{w,w}(y,e^t) =\prod_{\alpha>0,w(\alpha)<0}(1+ye^{w\alpha}) \/, \] 
calculated in Proposition \ref{prop:specialize}(d) is positive. 
The specialization at $y=-1$ gives the structure sheaf $\cO_{e_w}^{\mathbb{T}}$ (see Proposition \ref{prop:specialize}), which is also  consistent with the conjecture; 
the structure sheaf $\cO_{e_w}^{\mathbb{T}}$ is known to be Schubert 
alternating, using e.g., the positivity in the 
equivariant K theory proved by Anderson, Griffeth and Miller \cite{anderson.griffeth.miller:positivity}. 
Finally, we used a computer program to check Conjecture~\ref{conj:k} for flag manifolds of type $A_n$ 
for $n \le 5$ (i.e., up to $\Fl(5)$), and for the Lie types $B_2, C_2, D_3$ and $G_2$.

In Lie type A, F{e}h{\'e}r, Rim{\'a}nyi, and Weber \cite{feher2018characteristic} 
also observed a log concavity 
result for motivic Chern classes. We checked their conjectures in the same cases. 
\begin{conj}[Log concavity for MC classes]\label{conj:MClogc} Let 
$X(w)^\circ \subseteq G/B$ and consider the Schubert expansion
of the non-equivariant motivic Chern class:
\[ MC^T(X(w)^\circ) = \sum_{v \le w} c_{v;w}(y;e^t) \cO_v^T \/. \]
Then $c_{v;w}(y;e^t)$ is log-concave.\end{conj}
\begin{example} Consider $G$ of Lie type $G_2$. The coefficient of $\cO^{w_0}$ in the expansion of
$MC_y(X(w_0)^\circ)$ is
\[ 64 y^6 + 141y^5 + 125y^4 + 69y^3 + 29y^2 + 8 y + 1\/. \]
This is a log-concave polynomial. Its specialization at $y=-1$ gives $1$, reflecting the
fact that it calculates the Euler characteristic of the big cell in $G_2/B$.
\end{example}
\subsection{An interpretation in the Hecke algebra}\label{sec:Hecke}
In this section, we use Hecke algebra to give a combinatorial interpretation of the coefficients 
$c_{w,u}(y)$ in \Cref{conj:k}. For the cohomology case, see \cite[Theorem 6.2]{lee2018chern}.
Recall the K-theoretic BGG operator $\partial_i$ satisfies $\partial_i^2=\partial_i$ and the braid relation. The operators $\calT_i$ (and $\calT^\vee_i$) satisfy the finite Hecke algebra relation, see Proposition \ref{prop:hecke-relations}. Besides, we also have
\begin{lemma}\label{lem:commutation}
For any simple root $\alpha$ and torus weight $\lambda$,
\[\partial_{s_\alpha}\calL_{\lambda}-\calL_{s_\alpha\lambda}\partial_{s_\alpha}=\frac{\calL_{\lambda}-\calL_{s_\alpha\lambda}}{1-\calL_{\alpha}}\in \End_{K_{\mathbb{T}}(pt)[y]}\K_{T}(G/B)[y],\]
\[\calT_{s_\alpha}\calL_{\lambda}-\calL_{s_\alpha\lambda}\calT_{s_\alpha}=(1+y)\frac{\calL_{s_\alpha\lambda}-\calL_{\lambda}}{1-\calL_{-\alpha}}\in \End_{K_{\mathbb{T}}(pt)[y]}\K_{T}(G/B)[y]\]
and 
\[\calT^\vee_{s_\alpha}\calL_{\lambda}-\calL_{s_\alpha\lambda}\calT^\vee_{s_\alpha}=(1+y)\frac{\calL_{s_\alpha\lambda}-\calL_{\lambda}}{1-\calL_{-\alpha}}\in \End_{\K_{T}(pt)[y]}\K_{T}(G/B)[y].\]
Here $\frac{\calL_{\lambda}-\calL_{s_\alpha\lambda}}{1-\calL_{\alpha}}$ is defined as follows. Suppose $\frac{e^\lambda-e^{s_\alpha \lambda}}{1-e^\alpha}=\sum_\mu e^\mu$, then $\frac{\calL_{\lambda}-\calL_{s_\alpha\lambda}}{1-\calL_{\alpha}}:=\sum_\mu \calL_\mu$. 
\end{lemma}
\begin{proof}
We can check the equalities on the fixed point basis. Then all of them follow from Lemma \ref{lem:actiononfixedpoint} and the equality $\calL_\lambda \otimes\iota_w=e^{w\lambda}\iota_w$.
\end{proof}

Let us recall the definition of the K-theoretic Kostant-Kumar Hecke algebra \cite{kostant1990t} and the affine Hecke algebra. Let $P$ denote the weight lattice of $G$.
\begin{definition}\label{defin:nilhecke}
\begin{enumerate}
\item 
The Kostant-Kumar Hecke algebra $\calH$ is a free $\bbZ$ module with basis $\{D_we^\lambda|w\in W, \lambda\in P\}$, such that 
\begin{itemize}
\item 
For any $\lambda,\mu\in P$, $e^\lambda e^\mu=e^{\lambda+\mu}$.
\item
For any simple root $\alpha$, $D_{s_\alpha}^2=D_{s_\alpha}$.
\item
For any $w, y\in W$ such that $\ell(wy)=\ell(w)+\ell(y)$, $D_wD_y=D_{wy}$
\item 
For any simple root $\alpha$ and $\lambda\in P$, 
\[e^{\lambda}D_i-D_ie^{s_i\lambda}=\frac{e^\lambda-e^{s_i\lambda}}{1-e^{\alpha_i}}.\]
\end{itemize}

\item 
The affine Hecke algebra $\bbH$ is a free $\bbZ[q,q^{-1}]$ module with basis $\{T_we^\lambda|w\in W, \lambda\in P\}$, such that
\begin{itemize}
\item 
For any $\lambda,\mu\in P$, $e^\lambda e^\mu=e^{\lambda+\mu}$.
\item
For any simple root $\alpha$, $(T_{s_\alpha}+1)(T_{s_\alpha}-q)=0$.
\item
For any $w, y\in W$ such that $\ell(wy)=\ell(w)+\ell(y)$, $T_wT_y=T_{wy}$
\item 
For any simple root $\alpha$ and $\lambda\in P$, 
\[T_\alpha e^{s_\alpha\lambda}-e^\lambda T_\alpha=(1-q)\frac{e^\lambda-e^{s_\alpha\lambda}}{1-e^{-\alpha}}.\]
\end{itemize}
\end{enumerate}
\end{definition}
It follows from Lemma \ref{lem:commutation} that the Kostant-Kumar Hecke algebra $\calH$ acts on $K_{T}(G/B)$ by sending $D_i$ to $\partial_i$ and $e^\lambda$ to $\operL_\lambda$, see \cite{kostant1990t}; the affine Hecke algebra $\bbH$ acts on $K_{T}(G/B)[y,y^{-1}]$ by sending $q$ to $-y$, $T_i$ to $\calT_i$ (or $\calT^\vee_i$), and $e^\lambda$ to $\calL_\lambda$, see \cite{lusztig:eqK}. From here on, we always identify $q$ with $-y$. 

For any simple root $\alpha_i$, let
\begin{equation}\label{equ:simplegenerator}
T_i:=(1+ye^{\alpha_i})D_i-1=D_i(1+ye^{-\alpha_i})-(1+y+ye^{-\alpha_i})\in \calH[y].
\end{equation}
Then these $T_i$ and $e^\lambda$ satisfies the relations in the affine Hecke algebra $\bbH$. Therefore, $T_w$'s is well-defined for all $w\in W$.

For any $w$, we can expand $T_{w^{-1}}$ as a linear combination of terms $D_{u^{-1}}$,
\begin{equation}\label{equ:Heckecoeff}
T_{w^{-1}}:=\sum_{u\leq w}D_{u^{-1}}a_{u,w}(y;e^t),
\end{equation}
for some $a_{u,w}(y, e^t)\in \bbC[T][y]$. It is easy to compute $a_{w,w}(y;e^t)=\prod_{\alpha>0,w\alpha<0}(1+ye^{w\alpha})$, which equals $c_{w,w}(y;e^t)$ defined in \Cref{E:schub}. More general statements are true.
\begin{prop}\label{prop:KKH}
For any $u\leq w\in W$, we have
\[a_{u,w}(y; e^t)=c_{u,w}(y; e^t).\]
\end{prop}
\begin{proof}
Recall the Kostant-Kumar Hecke algebra $\calH$ acts on $\K_{T}(G/B)$ by sending $D_i$ to $\partial_i$ 
and $e^\lambda$ to $\calL_\lambda$. Under this action, $T_i$ is sent to the DL operator $\calT_i$
from \eqref{E:TiKdef}. 
By Theorem \ref{thm:MC+dual}, Equation \eqref{equ:BGGonstru} and 
the equality 
$\calL_\lambda\otimes \calO_{\id}^{T}=e^\lambda \calO_{\id}^{T}$, applying 
\eqref{equ:Heckecoeff} to $\calO_{\id}^{\mathbb{T}}\in \K_{T}(G/B)$, we get
\[MC_y(X(w)^\circ)=\sum_{u\leq w}a_{u,w}(y, e^t)\calO_{u}^{T}.\]
Therefore, $a_{u,w}(y; e^t)=c_{u,w}(y; e^t)$.
\end{proof}
Proposition~\ref{prop:KKH} provides a purely combinatorial way to compute the coefficients $c_{u,w}(y; e^t)$. 
In particular, we can check \Cref{conj:k} in the case when $\ell(w)\leq 2$ as follows:
\begin{enumerate}
\item $\ell(w)=1$. From $T_i=D_i(1+ye^{-\alpha_i})-(1+y+ye^{-\alpha_i})$,
we get $c_{1,s_i}=-(1+y+ye^{-\alpha_i})$ and $c_{s_i,s_i}=1+ye^{-\alpha_i}$. 
(This is consistent with \Cref{ex:P1}.)

\item $\ell(w)=2$. Pick two simple roots $\alpha_i, \alpha_j$. We calculate 
\begin{align*}
T_iT_j&=\left(D_i(1+ye^{-\alpha_i})-(1+y+ye^{-\alpha_i})\right)\left(D_j(1+ye^{-\alpha_j})-(1+y+ye^{-\alpha_j})\right)\\
&=D_iD_j(1+ye^{-\alpha_j})(1+ye^{-s_j\alpha_i})-D_j(1+ye^{-\alpha_j})(1+y+ye^{-s_j\alpha_i})\\
&-D_i\left((1+ye^{-\alpha_i})(1+y+ye^{-\alpha_j})-y(1+ye^{-\alpha_j})\frac{e^{-\alpha_i}-e^{-s_j\alpha_i}}{1-e^{\alpha_j}}\right)\\
&+(1+y+ye^{-\alpha_j})(1+y+ye^{-\alpha_i})-y(1+ye^{-\alpha_j})\frac{e^{-\alpha_i}-e^{-s_j\alpha_i}}{1-e^{\alpha_j}},
\end{align*}
where 
\[\frac{e^{-\alpha_i}-e^{-s_j\alpha_i}}{1-e^{\alpha_j}}=-e^{-s_j\alpha_i}-e^{-s_j\alpha_i+\alpha_j}-\cdots-e^{-\alpha_i-\alpha_j}.\]
\end{enumerate}
This implies \Cref{conj:k} when $\ell(w)=2$. 

%We may use this to recover \Cref{ex:FL3}. 
To illustrate this, consider $G=SL(3,\bbC)$, $i=2$, and $j=1$. Then 
\begin{align*}
T_2T_1&=D_2D_1(1+ye^{-\alpha_1})(1+ye^{-s_1\alpha_2})-D_1(1+ye^{-\alpha_1})(1+y+ye^{-s_1\alpha_2})\\
&-D_2\left((1+ye^{-\alpha_2})(1+y+ye^{-\alpha_1})-y(1+ye^{-\alpha_1})\frac{e^{-\alpha_2}-e^{-s_1\alpha_2}}{1-e^{\alpha_1}}\right)\\
&+(1+y+ye^{-\alpha_1})(1+y+ye^{-\alpha_2})-y(1+ye^{-\alpha_1})\frac{e^{-\alpha_2}-e^{-s_1\alpha_2}}{1-e^{\alpha_1}}\\
&=D_2D_1(1+ye^{-\alpha_1})(1+ye^{-\alpha_2-\alpha_1})-D_1(1+ye^{-\alpha_1})(1+y+ye^{-\alpha_2-\alpha_1})\\
&-D_2\left((1+ye^{-\alpha_2})(1+y+ye^{-\alpha_1})+y(1+ye^{-\alpha_1})e^{-\alpha_1-\alpha_2}\right)\\
&+(1+y+ye^{-\alpha_1})(1+y+ye^{-\alpha_2})+y(1+ye^{-\alpha_1})e^{-\alpha_1-\alpha_2}.
\end{align*}
Under the `dictionary' above between the Hecke algebra elements and operators, 
this recovers \Cref{ex:FL3}.

\section{Star duality}\label{sec:stardual} By `star duality' we mean the involution 
$\star: \K_T(X) \to \K_T(X)$ which sends (the class of) a vector bundle $[E]$ to 
$[E^\vee] = [Hom_{\cO_X}(E, \cO_X)]$. This is not an involution of 
$\K_T(pt)$-algebras, but it satisfies $\star(\mathbb{C}_\lambda) = \mathbb{C}_{-\lambda}$, where
$\mathbb{C}_\lambda$ denotes the trivial line bundle with weight~$\lambda$.
We extend $\star$ to $\K_{T}(X)[y,y^{-1}]$ by linearity, requiring $\star(y^i) = y^i$ 
for $i \in \bbZ$.

The goal of this section is to study the effect of this duality 
on the motivic Chern classes for Schubert cells in $X=G/B$.

Recall the DL operators $\mathcal{T}_i$ and $\operL_i$ from \eqref{E:TiKdef} respectively \eqref{equ:defofL}. These determine recursively the motivic Chern classes 
$MC_y(X(w)^\circ)=\calT_{w^{-1}}(\calO_{\id}^{{{T}}})$ and the 
(normalized) classes
$\widetilde{MC}_y(X(w)^\circ)=\operL_{w^{-1}}(\calO_{\id}^{T})$. (By \Cref{thm:MCtilde}, the opposite
classes $\widetilde{MC}_y(Y(w)^\circ)$ are orthogonal to the motivic classes.) 

We state next the main result in this section.
\begin{theorem}\label{thm:inversecoeff} Let $w \in W$. Then the following hold: 

(a) $\bb{C}_{- \rho} \otimes \mathcal{L}_{- \rho} \otimes MC_y(X(w)^\circ) = (-1)^{\mathrm{codim} X(w)} \star(\widetilde{MC}_y(X(w)^\circ))$.

(b)  Consider the Schubert expansions
\[ MC_y (X(w)^\circ) = \sum_{u \leq w} c_{u,w}(y;e^t) \mathcal{O}_{u}^{T} \/; \quad \widetilde{MC}_y(X(w)^\circ) = \sum_{u \leq w} d_{u,w}(y;e^t) \mathcal{I}_u^T \/. \]
Then $c_{u,w}(y;e^t)=(-1)^{\ell(u)-\ell(w)}\star(d_{u,w}(y;e^t))$, or, equivalently, 
\[ \langle MC_y(X(w)^\circ), \mathcal{I}^{u} \rangle = (-1)^{\ell(w) - \ell(u)} *\langle \widetilde{MC}_y(X(w)^\circ), \mathcal{O}^{u, T} \rangle \/. \] 

(c) Consider the Schubert expansions
\[ MC_y (X(w)^\circ) = \sum_{u \leq w} a_{u,w}(y;e^t) \mathcal{I}_{u}^{T} \/; \quad \widetilde{MC}_y(X(w)^\circ) = \sum_{u \leq w} b_{u,w}(y;e^t) \mathcal{O}_u^T \/. \]
Then $a_{u,w}(y;e^t)=(-1)^{\ell(u)-\ell(w)}b_{u,w}(y;e^t)$, or, equivalently, 
\[ \langle MC_y(X(w)^\circ), \mathcal{O}^{u, T} \rangle = (-1)^{\ell(w) - \ell(u)} \langle \widetilde{MC}_y(X(w)^\circ), \mathcal{I}^{u,T} \rangle \/. \]

\end{theorem}
Before we give the proof of this theorem, we note that the $y=0$ specialization 
recovers a known relation between the ideal sheaves and structure 
sheaves; see \Cref{prop:eqideal} below and compare to \cite[Proposition~4.3.4]{brion:flagv}. 
Brion proves the result in the non-equivariant case, and 
for completeness we sketch a proof for the equivariant generalization. 
Aside from the intrinsic interest, we also note that we utilize this result in the proof of
\Cref{thm:inversecoeff}.

\begin{prop}[Brion] \label{prop:eqideal} Let $w \in W$. Then the following holds in $\K_T(X)$:
\[ \bbC_{-\rho} \otimes \mathcal{L}_{- \rho} \otimes \mathcal{I}_w^T = 
(-1)^{\mathrm{codim}X(w)}  \star(\mathcal{O}_w^{T}) \/. \] \end{prop}

\begin{proof} Following Brion's proof, as equivariant sheaves, 
\[ \star(\mathcal{O}_w^T) = 
(-1)^{\mathrm{codim} X(w)} \omega_{X(w)} \cdot \omega_{X}^{-1} \/;\] 
this uses the fact that Schubert varieties are irreducible and Cohen-Macaulay - see \cite[\S 3.3]{brion:flagv}. 
The difference in the equivariant case is that the canonical sheaf of 
$X(w)$ needs to be twisted by a trivial bundle: 
\[ \omega_{X(w)} = \mathcal{O}_{X(w)} ( - \partial X(w)) \otimes \mathcal{L}_{\rho} \otimes \bb{C}_{-\rho} \/. \] This follows from \cite[Proposition 2.2.2]{brion.kumar:frobenius} to which one applies the 
projection formula. (Note that our convention defining  
$\mathcal{L}_\lambda$ is opposite to the one from \cite{brion.kumar:frobenius}.) 
If $X(w) = G/B$ then one obtains $\omega_X = \mathcal{L}_{2 \rho}$. The result follows from this. \end{proof} 
 
Define the $\bbZ$-linear endomorphism 
\[ \Psi: \K_{T}(X) \to \K_{T}(X); \quad [E]_T \mapsto \bb{C}_{\rho} \otimes \mathcal{L}_{\rho} \otimes \star [E]_T \/. \]
The previous proposition shows that \begin{equation}\label{E:psiformula} \Psi(\mathcal{I}_w^T) =  (-1)^{\mathrm{codim}X(w)} \cO_w^{T} \/. \end{equation} We need the following lemma.
\begin{lemma}\label{lemma:psi-iota} Let $w \in W$. Then \[ \Psi(\iota_w) = \frac{(-1)^{\dim G/B}}{e^{w(\rho) - \rho}} \iota_w \/. \]  
\end{lemma}

\begin{proof}
By definition, 
\begin{align*}
\Psi(\iota_w)|_u&=\delta_{w,u} e^{\rho+w\rho}\prod_{\alpha>0}(1-e^{-w\alpha})\\
&=\delta_{u,w}\frac{(-1)^{\dim G/B}}{e^{w(\rho) - \rho}}\prod_{\alpha>0}(1-e^{w\alpha})\\
&=\frac{(-1)^{\dim G/B}}{e^{w(\rho) - \rho}} \iota_w|_u.
\end{align*}
Therefore, $\Psi(\iota_w) = \frac{(-1)^{\dim G/B}}{e^{w(\rho) - \rho}} \iota_w$.
\end{proof}

The map $\Psi$ intertwines with the Hecke algebra action in the following way.
\begin{theorem}\label{thm:compatTL} For any $a \in \K_{T}(X)$, \[ \Psi( \mathcal{T}_i (a)) = - \operL_i (\Psi (a)) \/. \] In particular, if $w \in W$, then \[ \Psi( (\partial_i - \id)(\mathcal{I}_w^T)) = (-1)^{\mathrm{codim} X(w) +1} \partial_i(\cO_w^{T}) \/. \]
\end{theorem}
\begin{proof} The last statement follows from the first after specializing at $y=0$ and using 
\eqref{E:psiformula} and \Cref{lemma:yspec}. Therefore it suffices to prove the first statement. 
By localization, it suffices to prove this for the fixed point basis elements $a:=\iota_w$. 
Let $n= \dim G/B$. We use the formulae from \Cref{lem:actiononfixedpoint} and 
\Cref{lemma:psi-iota} to calculate:
\[ \Psi(\mathcal{T}_i(\iota_w)) = (-1)^n \frac{ -(1+y)}{e^{w(\rho) - \rho} (1- e^{w(\alpha_i)})} \iota_w + (-1)^n \frac{ 1+ y e^{w(\alpha_i)}}{e^{ws_i(\rho) - \rho}(1- e^{w(\alpha_i)})} \iota_{ws_i} \/. \]
By definiton of $\operL_i$, 
\[ \begin{split} - \operL_i \Psi(\iota_w) = - (\mathcal{T}_i^\vee + (1+y)\id) \Psi(\iota(w)) = \frac{(-1)^{n+1}}{e^{w(\rho) - \rho}} (\mathcal{T}_i^\vee (\iota_w) + (1+y) \iota_w) \/. 
\end{split} \] 
Using now the action of $\mathcal{T}_i^\vee$ from 
\Cref{lem:actiononfixedpoint} 
we calculate the last term as 
\[  \frac{(-1)^{n}}{e^{w(\rho) - \rho}} \left( (1+y) (\frac{1}{1- e^{-w(\alpha_i)}}-1 )\iota_w 
- \frac{1+ye^{w(\alpha_i)}}{1- e^{-w(\alpha_i)}}\iota_{ws_i}\right) \/. \] 
A simple algebra calculation shows that the coefficients of $\iota_w$ in 
both calculations are equal. The equality of the coefficients of $\iota_{ws_i}$ 
is proved similarly, using in addition that $s_i(\rho) = \rho - \alpha_i$.
\end{proof}

\begin{remark} \Cref{thm:compatTL} has a particularly natural interpretation 
in terms of the Kostant-Kumar Hecke algebra $\calH$. We keep the notation from
\S \ref{sec:Hecke}.
There is a Hecke algebra automorphism $A: \mathcal{H} \to \mathcal{H}$ sending 
$D_i \mapsto 1-D_i$ and $e^\lambda \mapsto e^{-\lambda}$. Let 
$L_i:= D_i(1+ y e^{\alpha_i}) + y$. Then it follows from definition that 
\[A(T_i)=A((1+ye^{\alpha_i})D_i-1)=-L_i.\]
Therefore, \Cref{thm:compatTL} shows that 
$\Psi: \K_{T}(X) \to \K_{T}(X)$ 
commutes with the Hecke automorphism $A$. 
\end{remark}  
%We are now ready to prove \Cref{thm:inversecoeff}.
\begin{proof}[Proof of \Cref{thm:inversecoeff}] Observe first that 
$\cO_{\id}^{T} = \mathcal{I}_{\id}^T = \iota_{\id}$, and that $\Psi(\iota_{\id}) = (-1)^{\dim G/B} \iota_{\id}$. 
Then, by \Cref{thm:compatTL},
\[ \Psi(\mathcal{T}_{w^{-1}} (\cO_{\id}^{T})) = (-1)^{\ell(w^{-1})} \operL_{w^{-1}}\Psi(\iota_{\id}) = (-1)^{\mathrm{codim} X(w)} \operL_{w^{-1}}(\iota_{\id})=  (-1)^{\mathrm{codim} X(w)} \widetilde{MC}_y(X(w)^\circ) \/. \] Since $\mathcal{T}_{w^{-1}} (\cO_{\id}^{T}) = MC_y(X(w)^\circ)$ 
by \Cref{thm:MC+dual}, this proves (a). 
%the first statement is proved. 

The equality \[c_{u,w}(y;e^t)=(-1)^{\ell(u)-\ell(w)}\star(d_{u,w}(y;e^t)),\] which, by \eqref{E:dualst}, 
is equivalent to 
\[\langle MC_y(X(w)^\circ), \mathcal{I}^{u,T} \rangle = (-1)^{\ell(w) - \ell(u)} *\langle \widetilde{MC}_y(X(w)^\circ), \mathcal{O}^{u, T} \rangle,\] 
follows by applying $\Psi$ to both sides of
\[MC_y (X(w)^\circ) = \sum_{u \leq w} c_{u,w}(y;e^t) \mathcal{O}_{u}^{T}\/,\] 
and using \Cref{prop:eqideal}. 
Finally, for part (c) 
%the last equality 
we utilize that $\cO^{u,T} = \sum_{v\geq u}\mathcal{I}^{v,T}$
(and proved by M{\"o}bius inversion). Then:
\begin{align*}
\langle MC_y(X(w)^\circ), \mathcal{O}^{u, T} \rangle=& \langle MC_y(X(w)^\circ), 
\sum_{v\geq u}\mathcal{I}^{v,T} \rangle\\
=& \sum_{v\geq u}(-1)^{\ell(w) - \ell(v)} *\langle \widetilde{MC}_y(X(w)^\circ), \mathcal{O}^{v, T} \rangle\\
=& (-1)^{\ell(w) - \ell(u)} \langle \widetilde{MC}_y(X(w)^\circ), \mathcal{I}^{u,T}\rangle.
\end{align*}
This finishes the proof.
 \end{proof}

\bibliographystyle{halpha}
\bibliography{AMSSbiblio}

\end{document}